%% file: front.tex
\documentclass[11pt]{amsart}

\input{header} 

\begin{document}
\addtolength{\oddsidemargin}{-0.8cm}
\addtolength{\evensidemargin}{-0.8cm}
\begin{abstract}
We study perturbations of relative cubic Dirac operators for basic classical Lie superalgebras within the uniform formalism of the colour quantum Weil algebra. This perspective leads to three complementary classes of perturbations and resulting invariants.

First, we define semisimple perturbations that assign to each finite-dimensional simple supermodule a finite collection of semisimple orbits, together with canonically defined vector spaces measuring the degree of atypicality. Second, we introduce nilpotent perturbations
parametrized by the self-commuting variety of a quadratic Lie subsuperalgebra; the
resulting family of cohomology theories combines Dirac cohomology and
Duflo--Serganova cohomology.  Third, we deform the cubic Dirac operator by a
Weil-covariant differential built from the universal $1$-form in the colour quantum Weil algebra
and the Weil differential, producing a Chern-type invariant that assigns to each
finite-dimensional module a natural class in the cohomology of the Weil
complex.
\end{abstract}

\maketitle
\setlength{\parindent}{1em}
\setcounter{tocdepth}{1}

\tableofcontents

\input{main} 

\bibliography{literatur}
\bibliographystyle{alpha}
\end{document}

%% file: header.tex
\usepackage{anysize} \marginsize{1.3in}{1.3in}{1in}{1in}
\usepackage{comment}
\usepackage{listings}
\usepackage{xcolor}
\usepackage{amsmath}
\usepackage{float}
\usepackage{mathtools}
\usepackage{booktabs}
\usepackage[all]{xy}
\usepackage[utf8]{inputenc}
\usepackage{varioref}
\usepackage{upgreek}
\usepackage{amsfonts}
\usepackage{amssymb}
\usepackage{bbm}
\usepackage{esint}
\usepackage{graphicx}
\usepackage{subcaption}
\usepackage{tikz}
\usepackage{empheq}
\usepackage{enumitem}
\usepackage{tikz-cd}
\usepackage{upgreek}
\usepackage{todonotes}
\usetikzlibrary{matrix,arrows,decorations.pathmorphing}
\usepackage{mathrsfs}
\usepackage[hypertexnames=false,backref=page,pdftex,
 	pdfpagemode=UseNone,
 	breaklinks=true,
 	extension=pdf,
 	colorlinks=true,
 	linkcolor=blue,
 	citecolor=red,
 	urlcolor=blue,
 ]{hyperref}

\newcommand{\bracket}{{\langle \cdot,\cdot \rangle}}

\newcommand{\HH}{{\mathbb H}}

\newcommand{\tr}{\operatorname{tr}}

\newcommand{\Hom}{\operatorname{Hom}}

\newcommand{\sgn}{\operatorname{sgn}}
\newcommand{\sdim}{\operatorname{sdim}}
\newcommand{\str}{\operatorname{str}}
\newcommand{\DS}{\operatorname{DS}}

\newcommand{\id}{\operatorname{id}}

\renewcommand{\Im}{\operatorname{im}}
\newcommand{\Ker}{\operatorname{ker}}

\newcommand{\End}{\operatorname{End}}
\newcommand{\ad}{\operatorname{ad}}

\newcommand{\Dirac}{\operatorname{D}}

\def\Laplace{{\Updelta_{\gg}(\xi)}}
\def\WW{{\mathcal{W}}}
\def \Gl{{\mathrm{GL}}}
\def\rr{{\mathfrak{r}}}

\def\Weil{{\mathcal{W}(\gg)}}

\def\CC{{\mathbb C}}
\def\RR{{\mathbb R}}

\def\ubar{{\overline{\mathfrak{u}}}}
\def\Sbar{{\overline{S}}}
 
\def\spinbar{{\overline{S}^{\gg,\ll}}}

\def\spin{{S^{\gg,\ll}}}
\def\spinbar{{\overline{S}^{\gg,\ll}}}
\def\rk{{\operatorname{rk}}}

\def\deg{{\operatorname{deg}}}
\def\bideg{{\operatorname{bideg}}}
\def\Cl{{\operatorname{Cl}}}
\def\HH{{\mathcal{H}}}

\def\osp{{\mathfrak{osp}}}

\def\hh{{\mathfrak h}}
\def\pp{{\mathfrak p}}
\def\uu{{\mathfrak u}}

\def\ZG{{\mathfrak{Z}(\gg)}}
\def\HC{{\operatorname{HC}}}
\def\bra{{\langle}}
\def\aut{{\operatorname{aut}}}
\def\H{{\operatorname{H}}}
\def\ket{{\rangle}}
\def\Dx{{\operatorname{D}_{\gg,\ll}^{x}}}

\def\ll{{\mathfrak l}}
\def\YY{{\mathcal Y}}

\def \AA {{\mathbb{A}}}

\def \at{{\operatorname{at}}}

\def\bb{{\mathfrak b}}
\def\even{{\mathfrak{g}_{\bar{0}}}}
\def\odd{{\mathfrak{g}_{\bar{1}}}}
\def\pp{{\mathfrak p}}
\def\tt{{\mathfrak t}}

\def\nn {{\mathfrak{n}}}
\def\pp{{\mathfrak{p}}}
\def\gg{{\mathfrak{g}}}

\def\ZZ{{\mathbb Z}}

\def\gsmod{{\mathfrak{g}}\textbf{-smod}}

\def\Weyl{{\mathscr{W}(\odd)}}

\def\UE{{\mathfrak U}}

\newcommand{\Sym}{\operatorname{Sym}}

\newcommand{\Ad}{\operatorname{Ad}}

\newcommand{\norm}[1]{\lVert#1\rVert}

\newcommand{\res}{\operatorname{res}}

\makeatletter
\newcommand*{\da@rightarrow}{\mathchar"0\hexnumber@\symAMSa 4B }
\newcommand*{\da@leftarrow}{\mathchar"0\hexnumber@\symAMSa 4C }
\newcommand*{\xdashrightarrow}[2][]{%
  \mathrel{%
    \mathpalette{\da@xarrow{#1}{#2}{}\da@rightarrow{\,}{}}{}%
  }%
}
\newcommand{\xdashleftarrow}[2][]{%
  \mathrel{%
    \mathpalette{\da@xarrow{#1}{#2}\da@leftarrow{}{}{\,}}{}%
  }%
}
\newcommand*{\da@xarrow}[7]{%
  \sbox0{$\ifx#7\scriptstyle\scriptscriptstyle\else\scriptstyle\fi#5#1#6\m@th$}%
  \sbox2{$\ifx#7\scriptstyle\scriptscriptstyle\else\scriptstyle\fi#5#2#6\m@th$}%
  \sbox4{$#7\dabar@\m@th$}%
  \dimen@=\wd0 %
  \ifdim\wd2 >\dimen@
    \dimen@=\wd2 %
  \fi
  \count@=2 %
  \def\da@bars{\dabar@\dabar@}%
  \@whiledim\count@\wd4<\dimen@\do{%
    \advance\count@\@ne
    \expandafter\def\expandafter\da@bars\expandafter{%
      \da@bars
      \dabar@ 
    }%
  }%
  \mathrel{#3}%
  \mathrel{%
    \mathop{\da@bars}\limits
    \ifx\\#1\\%
    \else
      _{\copy0}%
    \fi
    \ifx\\#2\\%
    \else
      ^{\copy2}%
    \fi
  }%
  \mathrel{#4}%
}
\makeatother

\makeatletter
\newsavebox\myboxA
\newsavebox\myboxB
\newlength\mylenA

\newcommand*\xtilde[2][0.8]{%
    \sbox{\myboxA}{$\m@th#2$}%
    \setbox\myboxB\null
    \ht\myboxB=\ht\myboxA%
    \dp\myboxB=\dp\myboxA%
    \wd\myboxB=#1\wd\myboxA
    \sbox\myboxB{$\m@th\widetilde{\copy\myboxB}$}
    \setlength\mylenA{\the\wd\myboxA}
    \addtolength\mylenA{-\the\wd\myboxB}%
    \ifdim\wd\myboxB<\wd\myboxA%
       \rlap{\hskip 0.5\mylenA\usebox\myboxB}{\usebox\myboxA}%
    \else
        \hskip -0.5\mylenA\rlap{\usebox\myboxA}{\hskip 0.5\mylenA\usebox\myboxB}%
    \fi}

\newbox\usefulbox

\def\getslant #1{\strip@pt\fontdimen1 #1}

\def\xxtilde #1{\mathchoice
 {{\setbox\usefulbox=\hbox{$\m@th\displaystyle #1$}%
    \dimen@ \getslant\the\textfont\symletters \ht\usefulbox
    \divide\dimen@ \tw@ 
    \kern\dimen@ 
    \xtilde{\kern-\dimen@ \box\usefulbox\kern\dimen@ }\kern-\dimen@ }}
 {{\setbox\usefulbox=\hbox{$\m@th\textstyle #1$}%
    \dimen@ \getslant\the\textfont\symletters \ht\usefulbox
    \divide\dimen@ \tw@ 
    \kern\dimen@ 
    \xtilde{\kern-\dimen@ \box\usefulbox\kern\dimen@ }\kern-\dimen@ }}
 {{\setbox\usefulbox=\hbox{$\m@th\scriptstyle #1$}%
    \dimen@ \getslant\the\scriptfont\symletters \ht\usefulbox
    \divide\dimen@ \tw@ 
    \kern\dimen@ 
    \xtilde{\kern-\dimen@ \box\usefulbox\kern\dimen@ }\kern-\dimen@ }}
 {{\setbox\usefulbox=\hbox{$\m@th\scriptscriptstyle #1$}%
    \dimen@ \getslant\the\scriptscriptfont\symletters \ht\usefulbox
    \divide\dimen@ \tw@ 
    \kern\dimen@ 
    \xtilde{\kern-\dimen@ \box\usefulbox\kern\dimen@ }\kern-\dimen@ }}%
 {}}

\newcommand*\xoverline[2][0.75]{%
    \sbox{\myboxA}{$\m@th#2$}%
    \setbox\myboxB\null
    \ht\myboxB=\ht\myboxA%
    \dp\myboxB=\dp\myboxA%
    \wd\myboxB=#1\wd\myboxA
    \sbox\myboxB{$\m@th\overline{\copy\myboxB}$}
    \setlength\mylenA{\the\wd\myboxA}
    \addtolength\mylenA{-\the\wd\myboxB}%
    \ifdim\wd\myboxB<\wd\myboxA%
       \rlap{\hskip 0.5\mylenA\usebox\myboxB}{\usebox\myboxA}%
    \else
        \hskip -0.5\mylenA\rlap{\usebox\myboxA}{\hskip 0.5\mylenA\usebox\myboxB}%
    \fi}

\def\xxoverline #1{\mathchoice
 {{\setbox\usefulbox=\hbox{$\m@th\displaystyle #1$}%
    \dimen@ \getslant\the\textfont\symletters \ht\usefulbox
    \divide\dimen@ \tw@ 
    \kern\dimen@ 
    \overline{\kern-\dimen@ \box\usefulbox\kern\dimen@ }\kern-\dimen@ }}
 {{\setbox\usefulbox=\hbox{$\m@th\textstyle #1$}%
    \dimen@ \getslant\the\textfont\symletters \ht\usefulbox
    \divide\dimen@ \tw@ 
    \kern\dimen@ 
    \xoverline{\kern-\dimen@ \box\usefulbox\kern\dimen@ }\kern-\dimen@ }}
 {{\setbox\usefulbox=\hbox{$\m@th\scriptstyle #1$}%
    \dimen@ \getslant\the\scriptfont\symletters \ht\usefulbox
    \divide\dimen@ \tw@ 
    \kern\dimen@ 
    \xoverline{\kern-\dimen@ \box\usefulbox\kern\dimen@ }\kern-\dimen@ }}
 {{\setbox\usefulbox=\hbox{$\m@th\scriptscriptstyle #1$}%
    \dimen@ \getslant\the\scriptscriptfont\symletters \ht\usefulbox
    \divide\dimen@ \tw@ 
    \kern\dimen@ 
    \xoverline{\kern-\dimen@ \box\usefulbox\kern\dimen@ }\kern-\dimen@ }}%
 {}}
\makeatother

\makeatletter
\newcommand{\mylabel}[2]{#2\def\@currentlabel{#2}\label{#1}}
\makeatother

\makeatletter
\newcommand{\Mac}{}
\DeclareRobustCommand{\Mac}{%
  M%
  \raisebox{\dimexpr\fontcharht\font`M-\height}{%
    \check@mathfonts\fontsize{\sf@size}{0}\selectfont
    c%
  }%
}
\makeatother

\newtheoremstyle{citing}
  {}
  {}
  {\itshape}
  {}
  {\bfseries}
  {\textbf{.}}
  {.5em}
  {\thmnote{#3}}

\theoremstyle{plain}
\newtheorem{theorem}{Theorem}

\newtheorem{lemma}[theorem]{Lemma}
\newtheorem{corollary}[theorem]{Corollary}

\newtheorem{proposition}[theorem]{Proposition}

\theoremstyle{remark}
\newtheorem{example}[theorem]{Example}

\theoremstyle{definition}
\newtheorem{conjecture}[theorem]{Conjecture}
\newtheorem{definition}[theorem]{Definition}

\numberwithin{equation}{section}

\theoremstyle{remark}
\newtheorem{remark}[theorem]{Remark}

{\theoremstyle{citing}
}

{\theoremstyle{definition}
}

\title{Perturbations of Dirac Operators}

\author{Steffen Schmidt}
\address{Center for Quantum Mathematics, 
University of Southern Denmark, DK-5230 Odense, Denmark}
\email{stschmidt@imada.sdu.dk}

\let\origmaketitle\maketitle
\def\maketitle{
  \begingroup
  \def\uppercasenonmath##1{} 
  \let\MakeUppercase\relax 
  \origmaketitle
  \endgroup
}


%% file: main.tex
\section{Introduction}
\subsection{Vue d'Ensemble}
Symmetry serves as a guiding principle in mathematics and physics. The language of symmetry groups and their actions
organizes a wide range of phenomena. In this setting, Dirac operators form a
natural bridge between quantum mechanics, geometry, and representation theory.

In representation theory, Dirac operators provide conceptual and computational tools
for addressing basic problems concerning a (semisimple) Lie group $G$: they furnish explicit constructions,
give effective criteria for unitarizability, contribute to classification results, and
relate the topology of locally symmetric spaces to representation-theoretic data.
A foundational instance is Parthasarathy's use of a representation-theoretic Dirac
operator in the study of discrete series representations~\cite{parthasarathy1972dirac}.

Interest in this circle of ideas intensified in the late 1990s, following Kostant's
introduction of the \emph{cubic Dirac operator}~\cite{Kostant_cubic_Dirac} and the
subsequent development of \emph{Dirac cohomology}, including Huang--Pand\v zi\'c's proof
of Vogan's conjecture~\cite{huang2002dirac}. More generally, Kostant defined cubic Dirac
operators $D_{\gg,\mathfrak{u}}$ for a pair consisting of a quadratic Lie algebra
$\gg$ and a quadratic Lie subalgebra $\mathfrak{u}$. When $\gg$ is
semisimple and $\mathfrak{u}$ has equal rank, these operators yield, among other
consequences, generalizations of the Bott--Borel--Weil theorem and of the Weyl character
formula. In~\cite{Alekseev_Meinrenken}, Alekseev and Meinrenken placed this construction
into a uniform framework by characterizing cubic Dirac operators as the unique elements
of the quantum Weil algebra whose induced differential is compatible with contractions
and Lie derivatives. 

In parallel with the representation-theoretic developments, (cubic) Dirac operators have
also been studied from a geometric viewpoint. Following Freed--Hopkins--Teleman~\cite{FT2}, one attaches to a representation a family of Dirac-type operators, and the resulting Dirac family encodes the representation in K-theoretic terms. This construction ties in beautifully with Kirillov’s orbit method: for an irreducible representation the associated K-class localizes on (essentially) a single coadjoint orbit, together with its prequantum line bundle (and corresponding twisting data). Meinrenken reformulated this observation for unitarizable modules over reductive Lie algebras (see~\cite[Section~8.6]{Meinrenken_Clifford_algebras}). 

On the analytic side, Quillen introduced the notion of a \emph{superconnection} as a
framework for a local family index theorem for Dirac operators~\cite{Quillen_superconnections},
a program that was subsequently realized by Bismut \cite{Bismut_superconnection}. In Quillen's principle, Dirac
operators are a quantization of the theory of connections, and the supertrace of the heat
kernel of the square of a Dirac operator quantizes the Chern character of the corresponding
connection. Superconnections have since proved fruitful in further contexts, including the
definition of Chern characters on non-compact manifolds and for infinite-dimensional vector
bundles.

We use these ideas as a guiding principle throughout the present article, where we consider extensions of the above theory of Dirac operators to Lie superalgebras.

\subsection{Dirac Operators and Lie Superalgebras} Huang and Pand\v zi\'c introduced Dirac operators and Dirac cohomology for quadratic Lie superalgebras relative to the even subalgebra $\even$
and proved that Dirac cohomology determines the infinitesimal character of a simple module~\cite{huang2005dirac}. Further aspects and the relation to unitarity were studied by Xiao~\cite{xiao2015dirac}
and by the author~\cite{schmidt}. Kostant-type cubic Dirac operators for quadratic
Lie superalgebras were developed by Kang--Chen~\cite{Chen_Dirac_quadratic} and Meyer~\cite{Meyer},
and the corresponding Dirac cohomology was treated in~\cite{Schmidt_Noja}; in particular,
much of the classical theory extends to the super setting. 

Despite these developments, the existing results are not yet organized within a uniform
formalism in which cubic Dirac operators appear as canonical objects. The present article uses such a framework by importing the
Alekseev--Meinrenken viewpoint: cubic Dirac operators are realized as distinguished
elements of the colour quantum Weil algebra. We implement this construction for quadratic Lie
superalgebras and use it as a uniform language for the ensuing Dirac-theoretic statements.

We now turn to the precise algebraic setting in which our construction takes place. Let $\gg$ be a Lie superalgebra equipped with a non-degenerate invariant supersymmetric bilinear form $B$ (a \emph{quadratic} Lie superalgebra). Let $\ll\subset\gg$ be a quadratic Lie subsuperalgebra with $B_{\ll}\coloneqq B|_{\ll}$; we call $(\gg,\ll)$ a \emph{quadratic pair}. Then
\begin{equation}
\gg=\ll\oplus\pp,\qquad \pp\coloneqq \ll^{\perp},
\end{equation}
where $\pp$ is $\ad(\ll)$-stable and $B_{\pp}\coloneqq B|_{\pp}$ is non-degenerate. 

Relative to a quadratic pair $(\gg,\ll)$, we study the cubic Dirac operator in the uniform framework of the colour quantum Weil algebra, extending the approach of Meinrenken and Alekseev~\cite{Alekseev_Meinrenken} to Lie superalgebras. This formalism provides a natural construction of relative Dirac operators and yields a streamlined proof that the cubic Dirac operator has a well-behaved square. The colour quantum Weil algebra of $\gg$ is the $\ZZ_{2}$-graded tensor product
\begin{equation}
\Weil\coloneqq \UE(\gg)\otimes \Cl(\gg),
\end{equation}
where $\UE(\gg)$ and $\Cl(\gg)$ carry the $\ZZ_{2}$-grading induced by the parity on $\gg$. We equip $\Weil$ with the $\ZZ_{2}\times\ZZ_{2}$-bigrading determined by
placing the generators $\gamma^{\WW}(x)$ in bidegree $(\bar{0},p(x))$ and the generators $1\otimes x$ in bidegree $(\bar{1},p(x))$. 
The bracket on $\Weil$ is defined as the graded commutator with respect to the associated symmetric bicharacter. The algebra $\Weil$ is naturally endowed with contractions $\iota_x$ and Lie derivatives $L_x$ for $x\in\gg$.

Inside $\Weil$ there is a distinguished element $\Dirac_{\gg}$, the \emph{cubic Dirac operator}. For a homogeneous basis $\{e_a\}$ of $\gg$ with $B$-dual basis $\{e^a\}$ it is
\begin{equation}\label{eq:cubic-dirac}
\Dirac_{\gg}\coloneqq \sum_a e^a\otimes e_a-\frac{1}{12}\sum_{a,b,c}(-1)^{p(e_a)p(e_b)+p(e_c)}f_{abc}\,e^ae^be^c.
\end{equation}
It induces the unique differential
$
d^{\WW}\coloneqq [\Dirac_{\gg},\cdot]_{\WW}
$
on $\Weil$ compatible with contractions and Lie derivatives. Moreover, its square is central:
\begin{equation}\label{eq:dirac-square}
\Dirac_{\gg}^2=\Omega_{\gg}\otimes 1+\frac{1}{24}\str_{\gg}\!\bigl(\ad_{\gg}(\Omega_{\gg})\bigr)(1\otimes 1),
\end{equation}
where $\Omega_{\gg}$ is the quadratic Casimir. The construction has a relative formulation in terms of a quadratic subalgebra $\ll\subset\gg$, which is viewed as an element of $\Weil$ under the canonical embedding $j:\WW(\ll)\hookrightarrow\Weil$. The \emph{relative cubic Dirac operator} is
\begin{equation}
\Dirac_{\gg,\ll}\coloneqq \Dirac_{\gg}-j(\Dirac_{\ll}).
\end{equation}
It lies in the subalgebra of $\ll$-basic elements $\WW(\gg,\ll)$ of $\Weil$ and squares to a central element. This places \cite{Chen_Dirac_quadratic, Meyer} into a common framework.

In this article we focus on the case where $\gg$ is a basic classical Lie superalgebra.
Let $\overline{M}(\mathfrak{p})$ be the oscillator supermodule of $\Cl(\mathfrak{p})$.
For a $\gg$-supermodule $M$, we consider the induced action of
$\Dirac_{\gg,\ll}$ on $M\otimes \overline{M}(\mathfrak{p})$. The associated Dirac cohomology is
\begin{equation}
\H_{\Dirac_{\gg,\ll}}(M)
\coloneqq \ker(\Dirac_{\gg,\ll})
\Big/ \Bigl(\ker(\Dirac_{\gg,\ll})\cap \Im(\Dirac_{\gg,\ll})\Bigr).
\end{equation}
It is a fundamental invariant: if $\H_{\Dirac_{\gg,\ll}}(M)\neq 0$, then
the central character of any simple $\ll$-submodule occurring in
$\H_{\Dirac_{\gg,\ll}}(M)$ determines the central character of $M$.
Applications of Dirac cohomology in the super setting are developed systematically
in~\cite{Schmidt_Noja,SchmidtDirac}; for instance, $\H_{\Dirac_{\gg,\ll}}(M)$
vanishes unless $M$ is a highest weight $\gg$-supermodule, and in that case it
embeds into Kostant's cohomology, with explicit examples computed in loc.\ cit.

\subsection{Results}
The main results are organized into three parts, corresponding to three families of deformations of the relative cubic Dirac operator associated with a quadratic pair. Throughout, $(\gg,\ll)$ denotes a quadratic pair, and $\Dirac_{\gg,\ll}\in\WW(\gg,\ll)$ the corresponding relative cubic Dirac operator. Unless stated otherwise, $\gg$ is assumed to be basic classical.

\subsubsection{Semisimple perturbations} The semisimple perturbations extend to Lie superalgebras the localization ideas of Freed--Hopkins--Teleman~\cite{FT2}, and single out the $\even$-constituents of finite-dimensional supermodules. Comparing the resulting Laplace family with the value at the origin yields a natural energy operator; together with the semisimple perturbations it detects atypicality, and the degree of atypicality at each \(\even\)-constituent.

Let $\Dirac_{\gg}\in\Weil$ be the absolute cubic Dirac operator and let $\hh_{\RR}^{\ast}$ be the real span of the root system. Using $B$ to identify $\hh_{\RR}^{\ast}\cong \hh_{\RR}$, we obtain a semisimple perturbation family
\begin{equation}\label{eq:Dirac-Laplace-family}
\hh_{\RR}^{\ast}\longrightarrow \Weil,\qquad
\xi\longmapsto \Dirac_{\gg}(\xi)\coloneqq \Dirac_{\gg}+1\otimes h_{\xi},\qquad \Updelta_{\gg}(\xi)\coloneqq \Dirac_{\gg}(\xi)^{2}.
\end{equation}
If $G_{\bar 0}$ denotes the connected simply connected Lie group with Lie algebra $\even$, then these families are $\hh_{\RR}$-invariant and $G_{\bar 0}$-equivariant (coadjoint action on $\gg^{\ast}$, adjoint action on $\Weil$).

Fix a positive system in $\Delta$ and let $M$ be a finite-dimensional simple $\gg$-supermodule. The semisimple perturbations~\eqref{eq:Dirac-Laplace-family} act on $M\otimes \overline{M}(\nn^{-})$ and lead to a family of Laplace operators whose kernel behaviour is the main object of study. In the super case the form $B$ on the real span of roots is typically indefinite, so the resulting Laplace equations do not localize a unique parameter. We therefore pass to the corresponding Laplace families for the even subalgebra $\even$, which recover a rigid parameter set and thereby detect the $\even$-decomposition of $M$.

Since $\even$ is reductive, it decomposes into simple and abelian factors, and for each factor one has the corresponding family of Dirac and Laplace operators. This permits the construction of a modified family of non-negative operators, denoted $\widetilde{\Updelta}_{\even}(\xi)$, whose kernel is the intersection of the kernels of the factorwise Laplace operators.

The main result is that this modified family detects precisely the $G_{\bar 0}$-coadjoint orbit attached to a finite-dimensional simple $\even$-module. The following theorem is Theorem~\ref{thm::main_family_even} in the main text.

\begin{theorem}
 Let $L(\Lambda)$ be a finite-dimensional simple $\gg$-supermodule with highest weight $\Lambda$ such that $L(\Lambda)\big\vert_{\even}=\bigoplus_{\mu} L_{0}(\mu)^{n(\mu)}$. Then 
 \[
 \ker \widetilde{\Updelta}_{\even}(\xi) \neq \{0\} \Leftrightarrow \xi \in \bigcup_{\mu \,:\, n(\mu)\neq 0}\Ad_{G_{\bar{0}}}^{\ast}(-\mu-\rho_{\bar{0}}).
 \]
 If $\xi=- \mu-\rho_{\bar{0}}$, then the kernel is $L(\Lambda)^{\mu}\otimes S^{\rho_{\bar{0}}} \otimes \overline{M}(\nn_{\bar{1}}^{-})$ with $S := S^{\gg, \nn_{\bar{0}}^{-}}$.
\end{theorem}

To detect atypicality, we compare the Laplace family with its value at the origin. This defines a canonical \emph{energy} family $T(\eta)$, parametrized by isotropic directions in $\hh_{\RR}^{\ast}$ and thus adapted to the defect of $\gg$. For a finite-dimensional simple $\gg$-supermodule $L(\Lambda)$, we then consider the joint family $(T(\eta),\widetilde{\Updelta}_{\even}(\xi))$, parametrized by $(\eta,\xi)\in (\hh_{\RR}^{\mathrm{iso}})^{\ast}\times\hh_{\RR}^{\ast}$, acting on the $\even$-submodule of $L(\Lambda)\otimes \overline{M}(\nn^{-})$ generated by $L(\Lambda)\otimes S\otimes 1_{\overline{M}(\nn_{\bar{1}}^{-})}$ with $S := S^{\gg,\nn_{\bar{0}}^{-}}$. The even parameter $\xi$ localizes at the $\even$-constituents of $L(\Lambda)$, while the isotropic parameter $\eta$ detects atypicality at each such constituent; in particular, the joint kernel is non-zero precisely on the corresponding product of isotropic constraints and $G_{\bar 0}$-orbits. The following theorem is Theorem~\ref{thm::main_resul_family_tuple} in the main text.

\begin{theorem}
Let $L(\Lambda)$ be a finite-dimensional simple $\gg$-supermodule. Then
$\ker((T(\eta),\widetilde{\Updelta}_{\even}(\xi))=\{0\}$ unless \emph{both}
\begin{enumerate}
\item[a)] $\xi=-\mu-\rho_{\bar{0}}$, where $\mu$ is the highest weight of a $\even$-constituent of $L(\Lambda)$,
\item[b)] $\eta\in X_{\mu+\rho}\coloneqq \{\eta' \in (\hh_{\RR}^{\text{iso}})^{\ast} : B(\mu+\rho, \eta')=0\}$.
\end{enumerate}
In particular,
\[
\ker (T(\eta), \widetilde{\Updelta}_{\even}(\xi)) \neq \{0\} \Leftrightarrow (\eta,\xi) \in \bigcup_{\mu\,:\,n(\mu)\neq 0}X_{\mu+\rho} \times \Ad_{G_{\bar{0}}}^{\ast}(-\mu-\rho_{\bar{0}}).
\]
\end{theorem}

\subsubsection{Nilpotent Perturbations} 

For a basic classical Lie superalgebra $\gg$ with $\odd\neq\{0\}$ there are two complementary constructions attached to a $\gg$-supermodule: Dirac cohomology, which controls the infinitesimal character, and Duflo--Serganova cohomology, a symmetric monoidal functor preserving the superdimension. We construct a single family of relative cubic Dirac operators whose cohomology combines these two theories.

Fix a quadratic pair $(\gg,\ll)$ and assume $\hh\subset\ll$ and $\ll_{\bar 1}\neq\{0\}$. Set
\begin{equation}\label{eq:self-commuting-variety}
 \YY_{\ll}\coloneqq \{x\in\ll_{\bar 1}:\ [x,x]=2x^{2}=0\},
\end{equation}
the self-commuting variety of $\ll$. In contrast with the reductive (purely even) case, the bigrading of $\WW(\gg,\ll)$ permits square-zero perturbations of the relative cubic Dirac operator while leaving its square unchanged. Let $\Dirac_{\gg,\ll}\in \WW(\gg,\ll)$ be the relative cubic Dirac operator. For $x\in\YY_{\ll}$ define
\begin{equation}\label{eq:Dirac-x-def}
\Dirac^{x}_{\gg,\ll}\coloneqq \Dirac_{\gg,\ll}+j\bigl(\gamma^{\WW}(x)\bigr)\in \WW(\gg,\ll).
\end{equation}
This element is odd, invariant under the centralizer
$
 C_{\ll}(x)\coloneqq \{y\in\ll:\ [x,y]=0\},
$
and it is $L_{\bar 0}$-equivariant in the sense that
$
 \Ad_{g}(\Dirac^{x}_{\gg,\ll})=\Dirac^{\Ad_{g}(x)}_{\gg,\ll}
$
for any $g \in L_{\bar{0}}$.
Here, $L_{\bar{0}}$ denotes the connected simply connected Lie group with Lie algebra $\ll_{\bar{0}}$. Moreover, the square is independent of $x$:
\begin{equation}\label{eq:square-Dirac-x}
 (\Dirac^{x}_{\gg,\ll})^{2}=\Dirac_{\gg,\ll}^{2}.
\end{equation}

For a $\gg$-supermodule $M$ we define the cohomology of $\Dirac^{x}_{\gg,\ll}$ by
\begin{equation}\label{eq:H-Dirac-x}
 \H_{\Dirac^{x}_{\gg,\ll}}(M)\coloneqq \ker \Dirac^{x}_{\gg,\ll}\big/\bigl(\ker \Dirac^{x}_{\gg,\ll}\cap \Im \Dirac^{x}_{\gg,\ll}\bigr).
\end{equation}
Our aim is to describe $\H_{\Dirac^{x}_{\gg,\ll}}(M)$ in terms of Dirac cohomology (Section~\ref{subsec::Dirac_cohomology}) and Duflo--Serganova cohomology (Section~\ref{subsec::Duflo_Serganova_Cohomology}). For the morphisms
\begin{equation}\label{eq:upeta-dual}
 \upeta_{x}^{\ast}:\Hom\bigl(\mathfrak Z(\ll_{x}),\CC\bigr)\longrightarrow \Hom\bigl(\mathfrak Z(\ll),\CC\bigr),
\end{equation}
as well as the homomorphism $\eta_{\ll}:\ZG\to\mathfrak Z(\ll)$, which arise naturally in Duflo--Serganova cohomology and Dirac cohomology, we establish the following result in the unitarizable case (\emph{cf.}~Theorem \ref{thm::Dirac_and_DS} in the main text): 

\begin{theorem}\label{thm:H-Dirac-x-unitary}
Assume that $M$ is a unitarizable highest weight $\gg$-supermodule with highest weight $\Lambda$. Then for every $x\in\YY_{\ll}$ one has
\begin{equation}\label{eq:H-Dirac-x-DS}
 \H_{\Dirac^{x}_{\gg,\ll}}(M)=\DS_{x}\!\bigl(\H_{\Dirac_{\gg,\ll}}(M)\bigr).
\end{equation}
In particular, $\H_{\Dirac^{x}_{\gg,\ll}}(M)$ is an $\ll_{x}$-supermodule. If $V$ is an $\ll_{x}$-supermodule occurring in $\H_{\Dirac^{x}_{\gg,\ll}}(M)$ with infinitesimal character $\chi_{\nu}^{\ll_{x}}$, then
\begin{equation}\label{eq:inf-char-inclusion}
 \chi_{\nu}^{\ll_{x}}\in (\upeta_{x}^{\ast})^{-1}\bigl(\chi_{\Lambda}\circ \eta_{\ll}\bigr).
\end{equation}
\end{theorem}

If $M$ is not unitary, having a spectral sequence argument would lead to the following conjecture.

\begin{conjecture}
Let $x\in \YY_{\ll}$. If $\DS_{x}\bigl(\H_{\Dirac_{\gg,\ll}}(M)\bigr)=0$, then
\begin{equation*}
\H_{\Dirac^{x}_{\gg,\ll}}(M)=0.
\end{equation*}
\end{conjecture}

\subsubsection{Bismut--Quillen Superconnection} 

Let $\gg$ be a semisimple complex Lie algebra and let $(\gg,\ll)$ be a quadratic pair. For a finite-dimensional $\gg$-module $M$, consider the finite-dimensional $\WW(\gg,\ll)$-supermodule $E\coloneqq M\otimes S$, where $S$ is the spin module associated with the pair. A naive Chern–Weil-type expression formed directly from the relative cubic Dirac operator $\Dirac_{\gg,\ll}$ vanishes identically. To obtain a non-trivial invariant, we therefore replace $\Dirac_{\gg,\ll}$ by a Bismut–Quillen-type superconnection $\AA_{\gg,\ll}^{M}(t)$, obtained by perturbing $\Dirac_{\gg,\ll}$ by a canonical term involving the universal connection $1$-form of $\Weil$ and the Weil differential. This yields an odd, $\ll$-equivariant element of $\End_{\CC}(\widehat{\WW}(\gg,\ll)\otimes E)$, and we define
\begin{equation}
\operatorname{ch}_{M}(t)\coloneqq \str_{E}\!\bigl(e^{-\AA_{\gg,\ll}^{M}(t)^{2}}\bigr)\in \widehat{\WW}(\gg,\ll).
\end{equation}
We prove that $\operatorname{ch}_{M}(t)$ determines a cohomology class independent of $t$. The following theorem appears in the main text as Theorem~\ref{thm::main_Bismut_Quillen}.

\begin{theorem}
For any finite-dimensional $\gg$-module $M$ the class
\begin{equation}
[\operatorname{ch}_{M}(t)]\in \H\bigl(\widehat{\WW}(\gg,\ll),d^{\WW_{\gg,\ll}}\bigr)
\end{equation}
is independent of $t$.
\end{theorem}

We finally introduce the appropriate replacement in the super setting. Since the oscillator module is infinite-dimensional, the supertrace on $E$ is not available in general. For unitarizable $M$, however, the ``heat operator" $e^{-t\Dirac_{\gg,\ll}^{2}}$ localizes as $t\to\infty$ on $\ker \Dirac_{\gg,\ll}$, and the argument producing the class $[\operatorname{ch}_{M}(t)]$ in the finite-dimensional case applies whenever a supertrace is available. This leads us to define a canonical substitute for $\operatorname{ch}_{M}(t)$ in the super case.

\subsection{Conventions and Notation} \label{subsec::conventions}The ground field is $\CC$. Let $\ZZ_{2}\coloneqq \ZZ/2\ZZ$ be the ring of integers modulo $2$, and denote by $\bar{0}$ and $\bar{1}$ the residue classes of even and odd integers, respectively. For a super vector space $V=V_{\bar{0}}\oplus V_{\bar{1}}$, the parity of an element $x\in V$ is written $p(x)\in\ZZ_{2}$. Expressions of the form $(-1)^{p(v)p(w)}$, for general $v,w\in V$, are understood by restricting first to homogeneous elements, inserting their parities in the exponent, and extending linearly. 

For any two super vector spaces $V$ and $W$, let $\Hom(V,W)$ denote the space of all parity–preserving linear maps. The elements of $\Hom(V,W)$ are called morphisms of super vector spaces.

For any two super vector spaces $V$ and $W$, define the $\ZZ_{2}$–graded tensor product $V\otimes W$ by
\[
(V\otimes W)_{\bar{0}}\coloneqq (V_{\bar{0}}\otimes W_{\bar{0}})\oplus(V_{\bar{1}}\otimes W_{\bar{1}}),\quad
(V\otimes W)_{\bar{1}}\coloneqq (V_{\bar{0}}\otimes W_{\bar{1}})\oplus(V_{\bar{1}}\otimes W_{\bar{0}}).
\]
The assignment $(V,W)\mapsto V\otimes W$ is additive in each variable. Moreover, $\otimes$ is associative, and the map $V\otimes W\to W\otimes V$, $v\otimes w\mapsto(-1)^{p(v)p(w)}w\otimes v$, is an isomorphism. If $V$ and $W$ are superalgebras, the product in $V \otimes W$ is $(v_{1}\otimes w_{1})(v_{2}\otimes w_{2})=(-1)^{p(w_{1})p(v_{2})}(v_{1}v_{2}\otimes w_{1}w_{2})$ for any homogeneous $v_{1},v_{2}\in V$ and $w_{1},w_{2} \in W$.

\subsection{Leitfaden}
The paper is organized as follows. Section~\ref{sec::preliminaries} collects the necessary background on basic classical Lie superalgebras, highest weight supermodules, infinitesimal characters, atypicality, and unitarizable supermodules. In Section~\ref{sec::cubic_Dirac_operator} we develop the cubic Dirac operator in the required generality. After introducing the colour quantum Weil algebra, we define the cubic and relative cubic Dirac operators, discuss the oscillator supermodule, and formulate the corresponding Dirac cohomology, with particular emphasis on the unitary setting. Section~\ref{sec::semisimple_perturbations} studies semisimple perturbations of the cubic Dirac operator and introduces the associated family of Laplace operators, culminating in a detecting family. Section~\ref{sec::nilpotent_perturbations} treats nilpotent perturbations parametrized by the self-commuting variety of a quadratic subalgebra and relates Duflo–Serganova cohomology to Dirac cohomology for unitarizable supermodules. Finally, Section~\ref{sec::Bismut_Quillen} places these perturbations in the framework of Bismut–Quillen superconnections, first for semisimple Lie algebras and then for basic classical Lie superalgebras.

\section{Preliminaries} \label{sec::preliminaries}
This section fixes conventions and notation used throughout the article. We recall basic classical Lie superalgebras, highest weight supermodules, infinitesimal characters and atypicality, and then formulate unitarity.

\subsection{Basic Classical Lie Superalgebra} 
 
Throughout, we work with quadratic Lie superalgebras $\gg$, \emph{i.e.}, Lie superalgebras admitting a non-degenerate invariant supersymmetric consistent bilinear form $B$. Consistency means $B(x,y)=0$ whenever $p(x)\neq p(y)$, invariance means $B([x,y],z)=B(x,[y,z])$ for all $x,y,z\in\gg$, and supersymmetry means $B(x,y)=(-1)^{p(x)p(y)}B(y,x)$ for all homogeneous $x,y\in\gg$.

Our main examples are basic classical Lie superalgebras $\gg=\even\oplus\odd$: $\gg$ is simple, $\even$ is reductive, and $\gg$ carries such a form $B$, unique up to a nonzero scalar.

According to Kac \cite{Kac}, the basic classical Lie superalgebras comprise the complex simple Lie algebras together with
\begin{equation}
A(m\vert n),\ B(m\vert n),\ C(n),\ D(m\vert n),\ F(4),\ G(3),\ D(2,1;\alpha).
\end{equation}
We assume throughout that $\alpha\in\RR$ for $D(2,1;\alpha)$, so that $\gg$ is contragredient, unless otherwise stated. 

For the simple Lie algebras we take the usual Killing form $B$, while for the basic classical Lie superalgebras $A(m\vert n)$ with $m\neq n$, $B(m\vert n)$, $C(n+1)$, $D(m\vert n)$ with $m\neq n+1$, $F(4)$, and $G(3)$, we take as $B(\cdot,\cdot)$ the \emph{super Killing form}
\begin{equation}
B(x,y)\coloneqq \str(\ad_x\circ\ad_y),\qquad x,y\in\gg.
\end{equation}
For the remaining cases, the super Killing form vanishes identically; an alternative invariant consistent form is then chosen as in \cite[Section~5.4]{Musson}, and will again be referred to simply as the super Killing form.

Let $\hh\subset\gg$ be a Cartan subalgebra and let $\Delta\coloneqq \Delta(\gg,\hh)$ be the set of roots. Then
\begin{equation}
\gg=\hh\oplus\bigoplus_{\alpha\in\Delta}\gg^\alpha,
\qquad
\gg^\alpha\coloneqq \{X\in\gg:[H,X]=\alpha(H)X\ \text{for all }H\in\hh\}.
\end{equation}
Since $\gg$ is basic classical, one has $\hh\subset\even$. Moreover, we fix some positive system $\Delta^{+}=\Delta_{\bar{0}}^{+} \sqcup \Delta_{\bar{1}}^{+}$ and define the \emph{Weyl vector} to be 
\begin{equation}
\rho\coloneqq \rho_{\bar{0}}-\rho_{\bar{1}}, \qquad \rho_{\bar{0}}\coloneqq \frac{1}{2}\sum_{\alpha \in \Delta^{+}_{\bar{0}}}\alpha, \qquad \rho_{\bar{1}}\coloneqq \frac{1}{2}\sum_{\alpha \in \Delta^{+}_{\bar{1}}}\alpha.
\end{equation}
Associated with $\Delta^{+}$, $\gg$ has a \emph{triangular decomposition}
\begin{equation}
 \gg=\nn^{-} \oplus \hh \oplus \nn^{+}, \qquad \nn^{\pm}\coloneqq \bigoplus_{\alpha \in \Delta^{+}}\gg^{\pm \alpha}.
\end{equation}

The Weyl group $W$ is the Weyl group of the Lie algebra $\even$. The associated fundamental system $\Uppi$ is the set of roots in $\Delta^+$ not expressible as sums of two positive roots; its elements are called \emph{simple roots}.

\subsection{Highest Weight Supermodules}
\label{subsubsec::Highest_weight_supermodules} The main class of $\gg$-supermodules considered in this article is the class of highest weight supermodules, which we now define. Fix a positive system $\Delta^{+}$. This choice determines a triangular decomposition $\gg=\nn^{-}\oplus\hh\oplus\nn^{+}$ and the corresponding Borel subsuperalgebra $\bb\coloneqq \hh\oplus\nn^{+}$. In what follows, $\Delta^{+}$ is kept fixed and omitted from the notation.
\begin{definition}
 A $\gg$-supermodule $M$ is called a \emph{highest weight} $\gg$\emph{-supermodule} with respect to a positive system $\Delta^{+}$ if there exists a nonzero vector $v_{\Lambda} \in M$ with $\Lambda \in \hh^{\ast}$ such that the following holds:
\begin{enumerate}
 \item[a)] $Xv_{\Lambda}=0$ for all $X \in \nn^{+}$,
 \item[b)] $Hv_{\Lambda}=\Lambda(H)v_{\Lambda}$ for all $H \in \hh$, and 
 \item[c)] $\mathfrak{U}(\gg)v_{\Lambda}=M$.
\end{enumerate}
The vector $v_{\Lambda}$ is referred to as the \emph{highest weight vector} of $M$, and $\Lambda$ is referred to as the \emph{highest weight}.
\end{definition}

Every simple highest weight $\gg$-supermodule is the unique simple quotient of a universal highest weight supermodule, the \emph{Verma supermodule}. Since $\mathfrak{U}(\gg)$ is a right $\UE(\bb)$-supermodule by right multiplication, for $\lambda\in\hh^{\ast}$ we set
\begin{equation}\label{eq::Verma_supermodules}
M_{\bb}(\lambda)\coloneqq \mathfrak{U}(\gg)\otimes_{\UE(\bb)}\CC_{\lambda},
\end{equation}
where $\CC_{\lambda}$ is the one-dimensional $\bb$-supermodule on which $\nn^{+}$ acts trivially and $\hh$ acts by $\lambda$.
Since $\bb$ is fixed, we omit the subscript. Then $M(\lambda)$ is a highest weight $\gg$-supermodule with highest weight $\lambda$ and highest weight vector $[1_{\mathfrak{U}(\gg)}\otimes 1]$, and $\UE(\nn^{-})[1_{\mathfrak{U}(\gg)}\otimes 1]=M(\lambda)$. Moreover, if $M$ is any $\gg$-supermodule and $v\in M$ has weight $\lambda$ with $\nn^{+}v=0$, there is a unique $\gg$-morphism $M(\lambda)\to M$ sending $[1_{\mathfrak{U}(\gg)}\otimes 1]$ to $v$.

The following properties of highest weight $\gg$-supermodules are immediate from their realization as quotients of Verma supermodules.

\begin{lemma}
 Let $M$ be a highest weight $\gg$-supermodule with highest weight $\Lambda \in \hh^{\ast}$.
 \begin{enumerate}
 \item[a)] $M$ is a \emph{weight supermodule}, that is, $M$ is $\hh$-semisimple.
\item[b)] For all weights $\lambda$ of $M$, we have $\dim(M^{\lambda}) < \infty$, while $\dim(M^{\Lambda})=1$.
\item[c)] Any nonzero quotient of $M$ is again a highest weight supermodule.
\item[d)] $M$ has a unique maximal subsupermodule and a unique simple quotient. 
\item[e)] Any two simple highest weight $\gg$-supermodules with highest weight $\Lambda$ are isomorphic.
 \end{enumerate}
\end{lemma}

In what follows, we denote the unique simple quotient of a highest weight $\gg$-supermodule of highest weight $\Lambda$ by $L(\Lambda)$.

An important class of highest weight $\gg$-supermodules is formed by the finite-dimensional ones. For our fixed Borel subalgebra
$\bb=\bb_{\bar 0}\oplus\bb_{\bar 1}$ with $\Delta^{+}=\Delta_{\bar 0}^{+}\sqcup\Delta_{\bar 1}^{+}$, the finite-dimensional simple $\gg$-supermodules are parametrized by dominant integral weights $\lambda\in\hh^{\ast}$, \emph{i.e.},
\begin{equation}
 B(\lambda+\rho_{\bar 0},\alpha)>0\qquad\text{for all }\alpha\in\Delta_{\bar 0}^{+}.
\end{equation}
Let $P^{++}$ denote the set of such weights (with respect to our fixed $\bb$). For $\lambda\in P^{++}$, let $L(\lambda)$ be the
simple highest weight $\gg$-supermodule of highest weight $\lambda$ with even highest weight vector. Then the full set of finite-dimensional
simple $\gg$-supermodules is
\begin{equation}
 \{\,L(\lambda),\ \Pi L(\lambda)\mid \lambda\in P^{++}\,\}.
\end{equation}

\subsection{Infinitesimal Characters and Atypicality} \label{subsec::atyicality}
Any highest weight $\gg$-supermodule admits an \emph{infinitesimal character}, \emph{i.e.}, an algebra homomorphism
$\chi:\mathfrak Z(\gg)\to\CC$. We describe these via the Harish--Chandra homomorphism. By PBW,
\begin{equation}
\mathfrak{U}(\gg)\cong \UE(\hh)\oplus\bigl(\nn^{-}\mathfrak{U}(\gg)+\mathfrak{U}(\gg)\nn^{+}\bigr),
\end{equation}
and we let $p:\mathfrak{U}(\gg)\to \UE(\hh)$ be the corresponding projection. Its restriction
$
p|_{\mathfrak Z(\gg)}:\ \mathfrak Z(\gg)\to \UE(\hh)\cong \operatorname{S}
$
is an algebra homomorphism. Define the twist $\zeta:\operatorname{S}\to \operatorname{S}$ by
\begin{equation}
\lambda(\zeta(f))=(\lambda-\rho)(f)\qquad(f\in \operatorname{S},\ \lambda\in\hh^{\ast}),
\end{equation}
and set $\operatorname{HC}\coloneqq\zeta\circ p|_{\mathfrak Z(\gg)}$. The map $\operatorname{HC}:\mathfrak Z(\gg)\to \operatorname{S}$ is an injective ring homomorphism. Moreover, its image can be described as follows. Let $\operatorname{S}(\hh)^{W}\coloneqq \{ f \in \operatorname{S}(\hh) : w(\lambda)(f)=\lambda(f) \, \text{for all } w \in W, \, \lambda \in \hh^{\ast} \}$ and, for any $\lambda \in \hh^{\ast}$, define
$
 A_{\lambda}\coloneqq \{ \alpha \in \Delta_{1}^{+} : (\lambda + \rho, \alpha)=0 \}.
$
 Then, the image of $\text{HC}$ is given by (\cite{kac1984laplace,sergeev1987enveloping,sergeev1988invariant}):
 \begin{equation*}
 \Im(\text{HC})=\bigl\{ f \in \operatorname{S}(\hh)^{W} : (\lambda + t\alpha)(f)=\lambda(f) \ 
 \text{for all } t \in \CC, \, \lambda \in \hh^{\ast}, \, \alpha \in A_{\lambda - \rho} \bigr\}.
 \end{equation*}

For $\lambda\in\hh^{\ast}$ define $\chi_{\lambda}(z)\coloneqq (\lambda+\rho)\bigl(\operatorname{HC}(z)\bigr)$ for $z\in\mathfrak Z(\gg)$.

\begin{definition}
A $\gg$-supermodule $M$ \emph{has infinitesimal character} if $\mathfrak Z(\gg)$ acts on $M$ via $\chi_{\lambda}$ for some
$\lambda\in\hh^{\ast}$. In this case, $\chi_{\lambda}$ is the infinitesimal character of $M$.
\end{definition}

Examples include highest weight $\gg$-supermodules $M$ with highest weight $\Lambda \in \hh^{\ast}$. They have infinitesimal character $\chi_{\Lambda}$. 

As a direct consequence of the description of $\Im(\operatorname{HC})$, for any $\lambda,\lambda'\in\hh^{\ast}$ one has
$\chi_{\lambda}=\chi_{\lambda'}$ if and only if
\begin{equation*}
 \lambda'=w\biggl(\lambda+\rho+\sum_{i=1}^{k} t_{i}\alpha_{i}\biggr)-\rho,
\end{equation*}
where $w\in W$, $t_i\in\CC$, and $\alpha_1,\dots,\alpha_k\in A_{\lambda}\coloneqq \{\alpha\in\Delta_{1}^{+}:(\lambda+\rho,\alpha)=0, \ B(\alpha,\alpha)=0\}$
are linearly independent odd isotropic roots. This leads to the definition of typicality and atypicality.

\begin{definition}
 A weight $\Lambda \in \hh^{\ast}$ is called \emph{typical} if $A_\Lambda=\emptyset$, that is,
$(\Lambda+\rho,\alpha)\neq 0$ for
all $\alpha \in \Delta_{\bar{1}}^{+}$. Otherwise, $\Lambda$ is called \emph{atypical}. The
\emph{degree of atypicality} of $\Lambda$, denoted by $\at(\Lambda)$, is the maximal number of
linearly independent mutually orthogonal positive odd isotropic roots 
$\alpha \in \Delta_{\bar{1}}^{+}$
such that $(\Lambda + \rho,\alpha)=0$. In brief, $\at(\Lambda)$ is the dimension of a maximal
isotropic subspace of $\operatorname{span}_\CC(A_{\Lambda})\subset \hh^{\ast}$. We call a highest weight 
$\gg$-supermodule $M$ with highest weight $\Lambda$ \emph{typical} if $\at(\Lambda)=0$, and 
otherwise \emph{atypical}.
\end{definition}

\begin{remark}
 The degree of atypicality is independent of the choice of positive root system. 
\end{remark}

The \emph{defect} of $\gg$, denoted $\operatorname{def}(\gg)$, is the dimension of a maximal isotropic subspace in the $\RR$-span of $\Delta$.
For $\gg$ of type $A(m-1|n-1)$, $B(m|n)$, or $D(m|n)$ (so that $\Delta_{\RR}\cong \RR^{m,n}$), one has
$\operatorname{def}(\gg)=\min(m,n)$. For $\gg$ of type $C(n)$, $\operatorname{def}(\gg)=1$. A simple Lie algebra and $\osp(1|2n)$ have defect $0$.
In all cases,
\begin{equation}
 0\le \at(\lambda)\le \operatorname{def}(\gg)
\end{equation}
for any $\lambda\in\hh^{\ast}$.

\subsection{Unitarizable Supermodules} \label{subsec::unitarizable_supermodules}

In this subsection, we define unitarizable $\gg$-supermodules, which are formulated relative to real forms of $\gg$. Our notation follows
\cite{Carmeli_Fioresi_Varadarajan_HW,Chuah_Fioresi_real,Fioresi_real_forms}.

Let $V=V_{\bar 0}\oplus V_{\bar 1}$ be a complex super vector space. A \emph{super Hermitian form} on $V$ is a sesquilinear map
$\bracket:V\times V\to\CC$, linear in the first variable and conjugate-linear in the second, satisfying
\begin{equation}
 \bra v,w\ket=(-1)^{p(v)p(w)}\,\overline{\bra w,v\ket}
\end{equation}
for all homogeneous $v,w\in V$, and we assume it is consistent, \emph{i.e.}, $\bra v,w\ket=0$ whenever $p(v)\neq p(w)$.

A super Hermitian form $\bracket$ decomposes as $\bracket=\bracket_{\bar{0}}+i\bracket_{\bar{1}}$, where
\begin{equation}
 \bracket_{\bar{s}}\coloneqq (-1)^s\,\bracket|_{V_{\bar s}\times V_{\bar s}},\qquad \bar{s}\in\ZZ_2.
\end{equation}
Then $\bracket_{\bar{0}}$ and $\bracket_{\bar{1}}$ are ordinary Hermitian forms on $V_{\bar 0}$ and $V_{\bar 1}$, respectively. We call $\bracket$
\emph{non-degenerate} (resp.\ \emph{positive definite}) if both $\bracket_{\bar{0}}$ and $\bracket_{\bar{1}}$ are non-degenerate (resp.\ positive definite).
We say that $\bracket$ is \emph{super positive definite} if $\bracket_{\bar{0}}$ is positive definite and $\bracket_{\bar{1}}$ is negative definite; in this case,
$\bracket$ is a \emph{Hermitian product} on $V$.

For $T\in\End_{\CC}(V)$, we define the \emph{adjoint} $T^{\dagger}$ by
\begin{equation}
 \bra Tv,w\ket=(-1)^{p(v)p(T)}\,\bra v,T^{\dagger}w\ket
\end{equation}
for all homogeneous $v,w\in V$, and extend by linearity.

Unitarity for $\gg$-supermodules is defined relative to real forms of $\gg$, which we now set up. For $r\in\{2,4\}$, let
$\aut_{2,r}^{\RR}(\gg)$ be the set of automorphisms $\theta$ of $\gg$ viewed as a real super vector space such that
$\theta|_{\gg_{\bar 0}\oplus\gg_{\bar 1}}\neq \id_{\gg_{\bar 0}\oplus\gg_{\bar 1}}$ and
\begin{equation}
 \theta^{2}\big|_{\gg_{\bar 0}}=\id_{\gg_{\bar 0}},\qquad
 \theta^{2}\big|_{\gg_{\bar 1}}=
 \begin{cases}
 \ \id_{\gg_{\bar 1}}, & r=2,\\
 -\id_{\gg_{\bar 1}}, & r=4.
 \end{cases}
\end{equation}
We set 
\begin{equation}
 \begin{aligned}
 \aut_{2,4}(\gg) &\coloneqq \{ \theta \in \aut_{2,4}^{\RR}(\gg) : \theta \ \text{is} \ \CC\text{-linear}\}, \\ 
 \overline{\aut}_{2,r}(\gg) &\coloneqq \{ \theta \in \aut_{2,r}^{\RR}(\gg) : \theta \ \text{is} \ \text{conjugate-linear}\}.
 \end{aligned}
\end{equation}

A \emph{real structure} on $\gg$ is a conjugate-linear Lie superalgebra morphism $\phi:\gg\to\gg$ such that
$\phi\in\overline{\aut}_{2,2}(\gg)$, \emph{i.e.}, $\phi$ is a conjugate-linear involution. The fixed-point subalgebra
$\gg^{\phi}$ is a \emph{real form} of $\gg$. Real structures in $\overline{\aut}_{2,2}(\gg)$ are related to elements of $\aut_{2,4}(\gg)$ as follows.

\begin{proposition}[{\cite{Fioresi_real_forms}}]\label{prop::real_form_omega}
There exists a unique $\omega\in\overline{\aut}_{2,4}(\gg)$, up to inner automorphisms of $\gg$, together with a positive system
$\Delta^{+}$ and root vectors $e_{\pm\alpha}$ for $\alpha\in\Delta^{+}$ such that
\[
\omega(e_{\pm\alpha})=-e_{\mp\alpha}\ \text{for all even simple }\alpha,\qquad
\omega(e_{\pm\alpha})=\pm e_{\mp\alpha}\ \text{for all odd simple }\alpha.
\]
Moreover:
\begin{enumerate}
\item[a)] $\omega$ induces a bijection
\[
\overline{\aut}_{2,2}(\gg)\setminus\{\theta:\theta|_{\even}=\omega|_{\even}\}\ \longrightarrow\ \aut_{2,4}(\gg),
\qquad \theta\longmapsto \omega^{-1}\circ\theta.
\]
\item[b)] For the Killing form $B(\cdot,\cdot)$ on $\gg$, one has $\overline{B(X,Y)}=B(\omega(X),\omega(Y))$ for all $X,Y\in\gg$,
and $B(\cdot,\omega(\cdot))$ is positive definite.
\end{enumerate}
\end{proposition}

\begin{remark}
 The positive system $\Delta^{+}$ is the \emph{distinguished positive system}, meaning that there is exactly one odd positive root that cannot be expressed as the sum of two other positive roots.
\end{remark}

As a consequence, there is a one-to-one correspondence, up to equivalence~\cite{Serganova_real,Pellegrini,Fioresi_real_forms,Chuah_real}
\begin{equation}
 \{\text{real forms }\gg^{\RR}\text{ of }\gg\}\ \longleftrightarrow\ \{\theta\in\aut_{2,4}(\gg)\}.
\end{equation}
Here a real form is realized as the fixed-point subspace of a suitable involution, and two real forms are identified whenever they are isomorphic (equivalently, whenever the corresponding involutions are conjugate by an automorphism of $\gg$).

Fix a real form $\gg^{\RR}$ of a basic classical Lie superalgebra $\gg$, \emph{i.e.}, $\gg^{\RR}=\gg^{\theta}$ for some
$\theta\in\aut_{2,4}(\gg)$. Set $\sigma\coloneqq \omega\circ\theta\in\overline{\aut}_{2,2}(\gg)$, the corresponding conjugate-linear
involution (\emph{cf.}~ Proposition~\ref{prop::real_form_omega}). We call $\theta$ a \emph{Cartan automorphism} of $\gg$ (or of
$\gg^{\RR}$) if
\begin{equation}\label{eq::inner_product}
 B_{\theta}(\cdot,\cdot)\coloneqq -B(\cdot,\theta(\cdot))
\end{equation}
is an inner product on $\gg^{\RR}$. For each $\theta\in\aut_{2,4}(\gg)$ there is a unique real form $\gg^{\RR}$ for which
$\theta$ restricts to a Cartan automorphism, and conversely each real form admits a unique Cartan automorphism~\cite[Theorem~1.1]{Chuah_real}.
Hence, we may assume throughout that the chosen $\theta$ is Cartan.

With the notion of a real form established, we can now define unitarizable supermodules.

\begin{definition}
 Let $\gg^{\RR}$ be a real form of $\gg$, and let $\HH$ be a complex $\gg^{\RR}$-supermodule. The supermodule $\HH$ is called a \emph{unitarizable} $\gg$\emph{-supermodule} (or \emph{unitarizable} $\gg^{\RR}$\emph{-supermodule}) if there exists a Hermitian product $\bracket$ on $\HH$ such that $X^{\dagger}=-X$ for all $X \in \gg^{\RR}$. Explicitly, this means:
 \[
 \langle Xv,w\ket=-(-1)^{p(v)p(X)}\bra v,Xw\ket
 \]
 for all $v,w \in \HH$ and $X \in \gg^{\RR}$.
\end{definition}

The following standard result is central.

\begin{proposition}\label{prop::unitarizable_completely_reducible}
Unitarizable $\gg$-supermodules are completely reducible, that is, the orthogonal complement of any $\gg^{\RR}$-submodule is again a $\gg^{\RR}$-submodule.
\end{proposition}

\section{The Cubic Dirac Operator} \label{sec::cubic_Dirac_operator}
We introduce the cubic Dirac operator, its relative version, and Dirac cohomology for quadratic Lie superalgebras in the language of the colour quantum Weil algebra.

Throughout this section, let $\gg$ be a quadratic Lie superalgebra with non-degenerate invariant supersymmetric bilinear form $B$. Via $B$, we identify $\gg$ with $\gg^{*}$ by the musical isomorphisms
\begin{equation}\label{eq::musical_isomorphisms}
\flat:\gg\to\gg^{*},\quad x\mapsto x^{\flat},\quad x^{\flat}(y)\coloneqq B(x,y),\qquad
\sharp:\gg^{*}\to\gg,\quad \alpha\mapsto\alpha^{\sharp},\quad B(\alpha^{\sharp},y)=\alpha(y)
\end{equation}
for all $y\in\gg$. This identification will be used without further comment.

\subsection{Colour Quantum Weil Algebra} The cubic Dirac operators and their perturbations take values in the colour quantum Weil algebra, which we now introduce.

\subsubsection{Exterior and Clifford Algebra over \texorpdfstring{$\gg$}{}} \label{subsec::exterior_and_clifford_superalgebra} We briefly recall the definition and properties of the \emph{exterior} and \emph{Clifford algebra} over $\gg$.

Let $T(\gg)$ be the tensor algebra of $\gg$ with unit $1_{T(\gg)}$. It carries the canonical $\ZZ$-grading
\begin{equation}
T(\gg)=\bigoplus_{k\ge 0}T^{k}(\gg),\qquad T^{0}(\gg)=\CC,\quad
T^{k}(\gg)=\mathrm{span}\{x_{1}\otimes\cdots\otimes x_{k}:x_i\in\gg\}.
\end{equation}
Reducing the degree mod $2$ yields a superalgebra structure with $T(\gg)_{\bar{0}}\coloneqq \bigoplus_{k\in 2\ZZ_{+}}T^{k}(\gg)$ and $
T(\gg)_{\bar{1}}\coloneqq \bigoplus_{k\in 2\ZZ_{+}+1}T^{k}(\gg).$ Thus $T(\gg)$ carries two $\ZZ_{2}$-gradings, hence two superalgebra structures: one from tensor degree modulo $2$, and one from the intrinsic parity of $\gg$. In this article we use the latter, that is, we regard $T(\gg)$ as the superalgebra induced by the $\gg$-parity.
\medskip

\noindent
\textbf{Exterior Algebra over $\gg$.}
Let $I_{\wedge}\subset T(\gg)$ be the homogeneous ideal generated by
\begin{equation}
x\otimes y+(-1)^{p(x)p(y)}\,y\otimes x,\qquad x,y\in\gg.
\end{equation}
The quotient $\bigwedge(\gg)\coloneqq T(\gg)/I_{\wedge}$ is the \emph{exterior algebra} of $\gg$, with unit $1_{\bigwedge(\gg)}$ and product $\wedge$ (exterior multiplication). Equivalently, it is generated by $\gg$ subject to
\begin{equation}
x\wedge y+(-1)^{p(x)p(y)}\,y\wedge x=0,\qquad x,y\in\gg.
\end{equation}

The $\ZZ$-grading of $T(\gg)$ descends to $\bigwedge(\gg)$, so that $\bigwedge(\gg)=\bigoplus_{k\ge 0}\bigwedge^{\!k}(\gg)$, where
$\bigwedge^{\!k}(\gg)$ is spanned by exterior products of $k$ elements of $\gg$. Reducing the degree mod $2$ gives the natural superalgebra grading
$\bigwedge(\gg)=\bigwedge(\gg)_{\bar{0}}\oplus \bigwedge(\gg)_{\bar{1}}$. We also have the $\ZZ_{2}$-grading induced by the intrinsic parity of $\gg$, and we fix this grading from now on.

On $\bigwedge(\gg)$ we use the standard operators $\epsilon$, $\iota$, and $L$.
For $v\in\gg$, let $\epsilon(v)$ denote \emph{left exterior multiplication} by $v$.
Define a derivation $\iota_v$ on $T(\gg)$ by
\begin{equation}\label{eq::contraction}
\iota_{v}(x_{1}\otimes\cdots\otimes x_{l})
\coloneqq \sum_{k=1}^{l}(-1)^{k-1}(-1)^{p(v)(p(x_{1})+\cdots+p(x_{k-1}))}\,B(v,x_{k})\,
x_{1}\otimes\cdots\otimes \widehat{x}_{k}\otimes\cdots\otimes x_{l}.
\end{equation}
Then $\iota_v(I_{\wedge})\subset I_{\wedge}$ (\emph{cf.}~\cite[Proposition~4.5]{generalized_Clifford}), hence $\iota_v$ descends to
$\bigwedge(\gg)$; it is called \emph{contraction}, has degree $-1$ with respect to the $\ZZ$-grading, and parity $p(v)$.
Finally, for homogeneous $T\in\End(\gg)$ define a degree-$0$ derivation $L_T$ on $T(\gg)$ by
\begin{equation}\label{eq::Lie_derivative}
L_{T}(x_{1}\otimes\cdots\otimes x_{l})
\coloneqq \sum_{k=1}^{l}(-1)^{p(T)(p(x_{1})+\cdots+p(x_{k-1}))}\,
x_{1}\otimes\cdots\otimes T(x_{k})\otimes\cdots\otimes x_{l}.
\end{equation}
It preserves $I_{\wedge}$ and therefore induces a degree-$0$ derivation on $\bigwedge(\gg)$, the \emph{Lie derivative}.

Moreover, $B$ induces a non-degenerate bilinear form on $\bigwedge(\gg)$. For $k\ge 0$ and $x_i,y_i\in\gg$, define on $T^{k}(\gg)$
\begin{equation}
\langle x_{1}\otimes\cdots\otimes x_{k},\, y_{1}\otimes\cdots\otimes y_{k}\rangle_{T^{k}}
\coloneqq \prod_{i=0}^{k-1} B(x_{k-i},\, y_{1+i}),
\end{equation}
and set
\begin{equation}
\langle x_{1}\otimes\cdots\otimes x_{k},\, y_{1}\otimes\cdots\otimes y_{k}\rangle_{\wedge^{k}}
\coloneqq \sum_{\sigma\in S_{k}} p(\sigma;x_{1},\ldots,x_{k})\,
\langle x_{\sigma(1)}\otimes\cdots\otimes x_{\sigma(k)},\, y_{1}\otimes\cdots\otimes y_{k}\rangle_{T^{k}},
\end{equation}
where $p(\sigma; x_{1},\ldots, x_{k})=\sgn(\sigma)\prod_{1 \leq i < j \leq k, \ \sigma^{-1}(i) > \sigma^{-1}(j)} (-1)^{p(x_{i})p(x_{j})}$.
Then $\langle\cdot,\cdot\rangle_{\wedge^{k}}$ vanishes on $I_{\wedge}\cap T^{k}(\gg)$ and hence descends to $\bigwedge^{\!k}(\gg)$,
yielding a bilinear form $\langle\cdot,\cdot\rangle_{\wedge}$ on $\bigwedge(\gg)$.

\begin{lemma}[{\cite[Section~5]{generalized_Clifford}}]
The form $\langle\cdot,\cdot\rangle_{\wedge}$ is non-degenerate and supersymmetric, and for homogeneous $x,y,z\in\gg$ one has
\[
\langle \iota_{x} y,\, z \rangle_{\wedge}
=(-1)^{p(x)p(y)}\, \langle y,\, x\wedge z\rangle_{\wedge}.
\]
\end{lemma}

We focus on $\bigwedge^{2}(\gg)$ and its relation to the orthosymplectic superalgebra $\osp(\gg)$. Define $\osp(\gg)\subset\End(\gg)$ to be the subspace of endomorphisms $T$ that are skew-supersymmetric with respect to $B$, that is,
\begin{equation}
B(Tx,y)+(-1)^{p(T)p(x)}\,B(x,Ty)=0 \qquad \forall\,x,y\in\gg,
\end{equation}
with Lie superbracket $[\cdot,\cdot]_{\osp(\gg)}$ given by the supercommutator.

\begin{lemma} \label{lemm::ad_x_is_orthosymplectic}
For $x\in\gg$, one has $\ad_{x}\in\osp(\gg)$.
\end{lemma}

\begin{proof}
For $y,z\in\gg$, invariance and supersymmetry of $B$ give
\[
B(\ad_{x}y,z)=B([x,y],z)=-(-1)^{p(x)p(y)}B(y,[x,z])=-(-1)^{p(x)p(y)}B(y,\ad_{x}z). \qedhere
\]
\end{proof}

With respect to $\ad:\gg\to\osp(\gg)$, define the \emph{moment map} $\mu:\gg\times\gg\to\osp(\gg)$ by
\begin{equation} \str\bigl(\ad_{x}\circ \mu(y,z)\bigr)=B([x,y],z)\qquad x,y,z\in\gg,
\end{equation}
where $\str$ is the supertrace on $\gg$. The map $\mu$ is skew-supersymmetric and satisfies
\begin{equation}
 \mu(x,y)(z)=-\iota_{z}(x\wedge y)
=-B(z,x)y+(-1)^{p(z)p(x)}B(z,y)x\qquad x,y,z\in\gg.
\end{equation}
For $\nu\in\bigwedge^{2}(\gg)$ we also write $A_{\nu}\coloneqq \mu(\nu)$.

\begin{proposition}[{\cite[Prop.~2.13]{Meyer}}]
 $\mu:\bigwedge^{2}(\gg)\to\osp(\gg)$ is an isomorphism of super vector spaces, with inverse $\lambda\coloneqq \mu^{-1}:\osp(\gg)\to\bigwedge^{2}(\gg)$, given in a basis $\{e_{a}\}$ of $\gg$, with $B$-dual basis $\{e^{a}\}$, by
\begin{equation*}
 \lambda(T)\coloneqq \mu^{-1}(T)=-\tfrac{1}{2}\sum_{a}T(e^{a})\wedge e_{a}.
\end{equation*}
\end{proposition}

Via the identification $\gg\cong \gg^{\ast}$ induced by $B$, we can describe $\lambda$ in a basis-free notation. For $T\in\osp(\gg)$ define
$\omega_{T}\in\bigwedge^{2}\gg^{*}$ by $\omega_{T}(x,y)\coloneqq B(Tx,y)$; it is skew-supersymmetric.

\begin{lemma}\label{lemm::identification_lambda_omega}
Under $\gg\cong \gg^{\ast}$ one has $\lambda(T)=\omega_{T}$ for all $T\in\osp(\gg)$.
\end{lemma}

\begin{proof}
Let $\flat:\gg\to\gg^{*}$ be given by $u^{\flat}=B(u,\cdot)$ and extend to
$\bigwedge^{2}\gg\to\bigwedge^{2}\gg^{*}$ by $(u\wedge v)^{\flat}=u^{\flat}\wedge v^{\flat}$.
For homogeneous $y,z$ we compute, using the convention of Section~\ref{subsec::conventions},
\[
\begin{split}
(2\lambda(T))^{\flat}(y\wedge z)
&=-\sum_{a}(-1)^{p(e_a)p(y)}\Bigl(B(Te^{a},y)\,B(e_a,z)
-(-1)^{p(y)p(z)+p(T)p(y)}B(Te^{a},z)\,B(e_a,y)\Bigr)\\
&=-(-1)^{p(y)p(z)}\sum_{a}\Bigl(B(Te^{a},y)\,B(e_a,z)-(-1)^{p(y)p(z)}B(Te^{a},z)\,B(e_a,y)\Bigr).
\end{split}
\]
Using $\sum_{a}B(e_a,\cdot)\,e^{a}=\id_{\gg}$, this becomes
\[
(-1)^{p(y)p(z)}(2\lambda(T))^{\flat}(y\wedge z)
=-\,B(Tz,y)+(-1)^{p(y)p(z)}B(Ty,z).
\]
Since $T\in\osp(\gg)$, we have $B(Tz,y)=-(-1)^{p(T)p(z)}B(z,Ty)$, and by supersymmetry of $B$ this yields
\[
(2\lambda(T))^{\flat}(y\wedge z)=2\,B(Ty,z)=2\,\omega_{T}(y,z).
\]
Hence $(\lambda(T))^{\flat}=\omega_{T}$, as claimed. \qedhere
\end{proof}

A further direct calculation yields the following lemma.

\begin{lemma} \label{lemm::Lie_derivative_and_lambd}
 For any $T_{1},T_{2} \in \osp(\gg)$, one has 
 \[
 L_{T_{1}} \lambda(T_{2})=\lambda([T_{1},T_{2}]_{\osp(\gg)}).
 \]
\end{lemma}

\noindent
\textbf{Clifford Superalgebra.} Let $B$ be the fixed invariant supersymmetric consistent bilinear form on $\gg$. Let $I_{\Cl}\subset T(\gg)$ be the two-sided ideal generated by
\begin{equation}
x\otimes y+(-1)^{p(x)p(y)}\,y\otimes x-2B(x,y)\,1_{T(\gg)},\qquad x,y\in\gg.
\end{equation}
The quotient
$
\Cl(\gg)\coloneqq T(\gg)/I_{\Cl}
$
is the \emph{Clifford algebra} of $\gg$. Identifying $\gg$ with its image in $\Cl(\gg)$, the defining relations read
\begin{equation}\label{eq::Clifford_relation}
xy+(-1)^{p(x)p(y)}\,yx=2B(x,y), \qquad x,y\in\gg.
\end{equation} 
From now on, we equip $\Cl(\gg)$ with the $\ZZ_{2}$-grading induced from $\gg$.

The operations \eqref{eq::contraction} and \eqref{eq::Lie_derivative} preserve the ideal $I_{\Cl}$ and hence induce \emph{contractions} $\iota_{x}$ and \emph{Lie derivatives} $L_{T}$ for any $x \in \gg$ and $T \in \End(\gg)$. 

We relate $\bigwedge(\gg)$ and $\Cl(\gg)$ via the usual Clifford action. Define
\begin{equation}
 \tau:\gg\to \End\bigl(\bigwedge(\gg)\bigr),\qquad \tau(x)\coloneqq \epsilon(x)+\iota_x,
\end{equation}
where $\epsilon(x)$ is left exterior multiplication and $\iota_x$ is contraction. Then, for $x,y\in\gg$,
\begin{equation}
 \tau(x)\tau(y)+(-1)^{p(x)p(y)}\tau(y)\tau(x)=2B(x,y).
\end{equation}
By the universal property of $\Cl(\gg)$ (\emph{cf.}~\cite[Proposition~3.1]{generalized_Clifford}), $\tau$ extends uniquely to a
superalgebra homomorphism $\Cl(\gg)\to \End(\bigwedge(\gg))$, hence makes $\bigwedge(\gg)$ a $\Cl(\gg)$-module.

\begin{theorem}[{\cite{generalized_Clifford}}]\label{thm::quantization_map}
The map $\eta:\Cl(\gg)\to \bigwedge(\gg)$ defined by
$ \eta(v)\coloneqq \tau(v)\,1_{\bigwedge(\gg)}
$
is an isomorphism of super vector spaces. Its inverse is the quantization map $q=\sum_k q_k:\bigwedge(\gg)\to \Cl(\gg)$, where
\begin{equation*}
q_k(x_1\wedge\cdots\wedge x_k)\coloneqq \frac1{k!}\sum_{\sigma\in S_k} p(\sigma;x_1,\ldots,x_k)\,
 x_{\sigma(1)}\cdots x_{\sigma(k)}.
\end{equation*}
Moreover, $q$ intertwines contractions, that is, $\iota_x\circ q=q\circ \iota_x$ for all $x \in \gg$.
\end{theorem}

The Lie derivative on $\Cl(\gg)$ by elements of $\osp(\gg)$ is implemented by a commutator. For $T\in\osp(\gg)$, we set
\begin{equation}
\gamma'(T)\coloneqq q(\lambda(T)).
\end{equation}
In particular, $\gamma'(\ad_{x})$ is defined for every $x\in\gg$ by Lemma~\ref{lemm::ad_x_is_orthosymplectic}. In what follows, we write $\gamma'(x)$ for $\gamma'(\ad_{x})$. Let $[\cdot,\cdot]_{\Cl(\gg)}$ denote the supercommutator on $\Cl(\gg)$, \emph{i.e.},
$[v,w]_{\Cl(\gg)}=vw-(-1)^{p(v)p(w)}wv$ for homogeneous $v,w\in \Cl(\gg)$. Then a direct computation shows:

\begin{lemma}\label{lemm::Lie_derivative_as_commutator} For $S,T\in\osp(\gg)$, one has:
\begin{enumerate}
 \item[a)] $-2L_T=[\gamma'(T),\,\cdot\,]_{\Cl(\gg)}$.
 \item[b)] $[\gamma'(S),\gamma'(T)]_{\Cl(\gg)}=-2\gamma'([S,T]_{\osp(\gg)})$.
\end{enumerate}
\end{lemma}

\subsubsection{Colour Quantum Weil Algebra}
With the notation fixed, we introduce the colour quantum Weil algebra. We regard $\mathfrak{U}(\gg)$ and $\Cl(\gg)$ as $\ZZ_{2}$-graded algebras, with grading induced by that of $\gg$.

\begin{definition}
 The \emph{colour quantum Weil algebra} associated with $ \gg $ is the $\ZZ_{2}$-graded tensor product
\[
\mathcal{W}(\gg)\coloneqq \mathfrak{U}(\gg) \otimes \Cl(\gg).
\]
\end{definition}

We take as generators of $\Weil$ the elements $1\otimes x$ and $\gamma^{\WW}(x)\coloneqq x\otimes 1-\tfrac12(1\otimes\gamma'(x))$, where $x \in \gg$. Recall that $\gamma'(x)\coloneqq \gamma'(\ad_{x})=q(\lambda(\ad_{x}))$ for any $x \in \gg$. This choice is natural, since the map $x\mapsto x\otimes 1-\tfrac{1}{2}(1\otimes\gamma'(x))$ defines an embedding of $\gg$ into $\Weil$. We endow $\Weil$ with a $\ZZ_{2}\times\ZZ_{2}$-grading by setting, for $x\in\gg$,
\begin{equation} \label{def::bidegree}
\bideg(1\otimes x)=(\bar 1,p(x)),\qquad \bideg(\gamma^{\WW}(x))=(\bar 0,p(x)).
\end{equation}

\begin{remark}\label{rmk::TP_generators_QW} Another common choice of generators is given by the tensor product generators
\begin{equation*}
u^{x}\coloneqq x\otimes 1, \qquad \theta^{x}\coloneqq 1\otimes x, \qquad x \in \gg.
\end{equation*}
They are used only in Section~\ref{sec::Bismut_Quillen}. They are of bidegree $\bideg(u^{x})=(\bar{0},p(x))$ and $\bideg(\theta^{x})=(\bar{1},p(x))$.
\end{remark}

For $\alpha=(\bar a,\bar b)$ and $\beta=(\bar c,\bar d)$ in
$\ZZ_{2}\times\ZZ_{2}$ we use the standard symmetric bicharacter
\begin{equation}
\chi(\alpha,\beta)\coloneqq (-1)^{\bar a\bar c+\bar b\bar d}.
\end{equation}
For homogeneous $A,B\in\Weil$ of bidegrees $\alpha$ and $\beta$, we set
\begin{equation} \label{eq::colour_commutator}
[A,B]_{\WW}\coloneqq AB-\chi(\alpha,\beta)\,BA,
\end{equation}
and extend $\CC$-linearly. The total degree is the sum of the two components of the bidegree in $\ZZ_{2}$. In particular, if $A\in\Weil$ has total degree
$\bar 1$, then $A^{2}=\tfrac12[A,A]_{\WW}$. This endows $\Weil$ with the structure of a colour Lie algebra (\emph{cf.}~\cite{Ree}), which we call the \emph{colour quantum Weil algebra}.

\begin{lemma}\label{lemm::commutation_relations_Weil_generators}
For $x,y\in\gg$ one has
\[
\begin{gathered}
[1\otimes x,\,1\otimes y]_{\WW}=2B(x,y)(1\otimes 1),\qquad
[\gamma^{\WW}(x),\gamma^{\WW}(y)]_{\WW}=\gamma^{\WW}([x,y]_{\gg}),\\
[\gamma^{\WW}(x),\,1\otimes y]_{\WW}=1\otimes [x,y]_{\gg}.
\end{gathered}
\]
\end{lemma}

\begin{proof}
Since $\bideg(1\otimes x)=(\bar 1,p(x))$ and $\bideg(1\otimes y)=(\bar 1,p(y))$, we have
$\chi(\bideg(1\otimes x),\bideg(1\otimes y))=-(-1)^{p(x)p(y)}$. Hence
\[
\begin{split}
[1\otimes x,1\otimes y]_{\WW}
&=(1\otimes x)(1\otimes y)+(-1)^{p(x)p(y)}(1\otimes y)(1\otimes x)\\
&=1\otimes\bigl(xy+(-1)^{p(x)p(y)}yx\bigr)
=2B(x,y)\,(1\otimes 1),
\end{split}
\]
by the Clifford relation. Next,
\[
\begin{split}
[\gamma^{\WW}(x),1\otimes y]_{\WW}
&=\gamma^{\WW}(x)(1\otimes y)-(-1)^{p(x)p(y)}(1\otimes y)\gamma^{\WW}(x)\\
&=-\tfrac12\Bigl((1\otimes \gamma'(x))(1\otimes y)-(-1)^{p(x)p(y)}(1\otimes y)(1\otimes \gamma'(x))\Bigr)
\\ &=-\tfrac12\bigl(1\otimes(\gamma'(x)y-(-1)^{p(x)p(y)} y\gamma'(x))\bigr)\\
&=1\otimes (-\tfrac{1}{2}[\gamma'(x),y)]_{\Cl(\gg)}) \\ &=1\otimes L_{x}(y)=1\otimes [x,y]_{\gg}
\end{split}
\]
using Lemma~\ref{lemm::Lie_derivative_as_commutator} in the last steps. The identity
$[\gamma^{\WW}(x),\gamma^{\WW}(y)]_{\WW}=\gamma^{\WW}([x,y]_{\gg})$ is proved similarly.
\end{proof}

On $\mathcal{W}(\gg)$ there exist two canonical operations. The first is the contraction $\iota_{x}$, an odd derivation characterized by the relations 
\begin{equation} \label{eq::contraction_Weil}
\iota_{x}(1\otimes y)=B(x,y)(1\otimes 1), \qquad \iota_{x}(\gamma^{\WW}(y))=-\frac{1}{2}(1\otimes [x,y]_{\gg}),
\end{equation}
for $x,y \in \gg$. The second is the Lie derivative, which is a combination of the Lie derivative $L^{U}$ on $\mathfrak{U}(\gg)$ and $L^{C}$ on $\Cl(\gg)$:
\begin{equation} \label{eq::Lie_derivative_Weil}
L_{x}(y \otimes z)\coloneqq L_{x}^{U}(y) \otimes z + y \otimes L_{x}^{C}(z)
\end{equation}
With these definitions, the following lemma is immediate by comparison with Lemma~\ref{lemm::commutation_relations_Weil_generators}.

\begin{lemma} \label{lemm::contraction_and_Lie_derivative_in_Weil}
For $x\in\gg$ one has:
\begin{enumerate}
\item[a)] $\iota_{x}=\tfrac12[1\otimes x,\,\cdot\,]_{\WW}$.
\item[b)] $L_{x}=[\gamma^{\WW}(x),\,\cdot\,]_{\WW}$.
\end{enumerate}
\end{lemma}

\begin{proof}
a) This is immediate from the definition of $\iota_x$ and the commutator \eqref{eq::colour_commutator}.
b) It suffices to check the identity on the generators of $\Weil$. For $y\in\gg$,
\[
\begin{split}
L_x(1\otimes y)
&=1\otimes L^{C}_x y
=1\otimes\Bigl(-\tfrac12[\gamma'(x),y]_{\Cl(\gg)}\Bigr)
=[\gamma^{\WW}(x),1\otimes y]_{\WW},
\end{split}
\]
using $[x\otimes 1,1\otimes y]_{\WW}=0$ and Lemma~\ref{lemm::commutation_relations_Weil_generators}. Moreover,
\[
\begin{split}
L_x(\gamma^{\WW}(y))
&=L_x\Bigl(y\otimes 1-\tfrac12(1\otimes \gamma'(y))\Bigr)\\
&=[x,y]_{\gg}\otimes 1-\tfrac12\bigl(1\otimes L_x\gamma'(y)\bigr)
=[x,y]_{\gg}\otimes 1+\tfrac14\bigl(1\otimes[\gamma'(x),\gamma'(y)]_{\Cl(\gg)}\bigr)\\
&=[x,y]_{\gg}\otimes 1-\tfrac12\bigl(1\otimes \gamma'([x,y]_{\gg})\bigr)
=\gamma^{\WW}([x,y]_{\gg})
=[\gamma^{\WW}(x),\gamma^{\WW}(y)]_{\WW},
\end{split}
\]
where we used Lemma~\ref{lemm::Lie_derivative_as_commutator} in the penultimate step and
Lemma~\ref{lemm::commutation_relations_Weil_generators} in the last step. Since the generators span $\Weil$ as an algebra, this
implies $L_x=[\gamma^{\WW}(x),\cdot]_{\WW}$.
\end{proof}

\subsection{Cubic Dirac Operator}
We construct the cubic Dirac operator for quadratic Lie superalgebras, including, in particular, all basic classical Lie superalgebras. Recall that any \emph{quadratic Lie superalgebra} $\gg$ is equipped with a non-degenerate invariant supersymmetric bilinear form $B$. In what follows, we fix such a pair $(\gg,B)$. Then the cubic Dirac operator is a distinguished element of the colour quantum Weil algebra $\mathcal{W}(\gg)$, obtained by adding to the canonical quadratic element the Clifford quantization of the structure-constants tensor associated with $(\gg,B)$. The resulting element is invariant under the diagonal adjoint action of $\gg$ and encodes the full Lie superalgebraic structure of $\gg$. 

\subsubsection{Structure Constants Tensor}
Let $\{e_{a}\}$ be a homogeneous basis of $\gg$ with $B$-dual basis $\{e^{a}\}$, that is, $B(e_{a},e^{b})=\delta_{ab}$. The bilinear form $B$ identifies $\gg\cong \gg^{\ast}$ and defines the Cartan $3$-form
\begin{equation}
\phi_{\gg}(x,y,z)\coloneqq -\tfrac{1}{2}B\bigl(x,[y,z]\bigr),\qquad x,y,z\in\gg .
\end{equation}
By invariance and supersymmetry of $B$, the form $\phi_{\gg}$ is $\gg$-invariant and totally skew-supersymmetric. Via the identification $\bigwedge(\gg)\to \Cl(\gg)$ (\emph{cf.}~\ref{thm::quantization_map}), $\phi_{\gg}$ determines a canonical cubic element $\phi_{\gg}\in\bigwedge^{3}(\gg)_{\bar{0}}$ characterized by
\begin{equation}\label{eq::definition_phi}
(\phi_{\gg},x\wedge y\wedge z)_{\wedge}=-\tfrac{1}{2}B([x,y],z)
\end{equation}
for all $x,y,z\in\gg$. Here $(\cdot,\cdot)_{\wedge}$ denotes the non-degenerate supersymmetric bilinear form on $\bigwedge(\gg)$ introduced in Section~\ref{subsec::exterior_and_clifford_superalgebra}. The element $\phi_{\gg}$ is called the \emph{structure constants tensor} of $\gg$. If $\gg$ is clear from the context, we omit the subscript $\gg$.

\begin{lemma} \label{lemm::properties_phi_g_fd} For all $x,y,z\in \gg$, the following hold:
 \begin{enumerate}
 \item[a)] $2\iota_{x}\phi=-\lambda(\ad_{x})$ for all $x \in \gg$.
 \item[b)] $\iota_{x}\iota_{y}\phi=B([x,y], \cdot)$. In particular, under the identification $\gg \cong \gg^{\ast}$, $\iota_{x}\iota_{y}\phi=[x,y]$.
 \item[c)] In the homogeneous basis, 
 \[
 \phi=-\frac{1}{12} \sum_{a,b,c} (-1)^{p(e_{a})p(e_{b})+p(e_{c})}f_{abc}e^{a}\wedge e^{b} \wedge e^{c}.
 \]
 \item[d)] $\phi$ is invariant under the adjoint action of $\gg$.
 
 \end{enumerate}
\end{lemma}

\begin{proof}
 a) 
For homogeneous $ x,y,z $, we obtain using Lemma~\ref{lemm::identification_lambda_omega}
\[
(\iota_x\phi)(y \wedge z)=\phi(x,y,z)
=-\tfrac{1}{2} B([x,y],z)
=-\tfrac{1}{2} \omega_{\ad_x}(y,z)=-\tfrac{1}{2}(\lambda(\ad_{x}))^{\flat}(y \wedge z).
\]
Hence $
2(\iota_x\phi)^\flat=-(\lambda(\ad_x))^\flat.
$
Using the $B$-identification $ \bigwedge^{2}(\gg) \cong \bigwedge^{2}(\gg^{\ast}) $, we finally obtain the statement.

 b) One has by definition:
 \[
 \begin{aligned}
 \iota_{x}\iota_{y}\phi(z) &=(\iota_{x}\iota_{y}\phi,z)_{\wedge}=(-1)^{p(x)p(y)}(\iota_{y}\phi, x \wedge z)_{\wedge}=(-1)^{p(x)p(y)}(\phi, y \wedge x \wedge z)_{\wedge} \\ &=-(-1)^{p(x)p(y)}B([y,x],z)=B([x,y],z).
 \end{aligned}
 \]

 c) is a straightforward computation and will be omitted. 

 d) One has $\phi(x,y,z)=-\tfrac12 B([x,y],z)$. For homogeneous $w,x,y,z$,
the Lie derivative on forms gives
\[
(L_w\phi)(x,y,z)
=\phi([w,x],y,z)+(-1)^{p(w)p(x)}\phi(x,[w,y],z)
+ (-1)^{p(w)(p(x)+p(y))}\phi(x,y,[w,z]).
\]
Hence, using invariance of $B$, one has for any $w \in \gg$
\[
\begin{aligned}
-2(L_w\phi)(x,y,z)
&=B([[w,x],y],z)
+ (-1)^{p(w)p(x)}B([x,[w,y]],z)
 +(-1)^{p(w)(p(x)+p(y))}B([x,y],[w,z])\\
&=B([w,x],[y,z])
 +(-1)^{p(w)p(x)}B(x,[[w,y],z])
 +(-1)^{p(w)(p(x)+p(y))}B(x,[y,[w,z]])\\
&=(-1)^{p(w)p(x)}\Big(
 -B(x,[w,[y,z]]) + B(x,[[w,y],z]) + (-1)^{p(w)p(y)}B(x,[y,[w,z]])
 \Big).
\end{aligned}
\]
By the $\ZZ_{2}$-graded Jacobi identity
$[w,[y,z]]=[[w,y],z] + (-1)^{p(w)p(y)}[y,[w,z]]$,
the parentheses vanish. Therefore $L_w\phi=0$ for all $w$, \emph{i.e.}, $\phi$ is invariant under the adjoint action of $\gg$.
\end{proof}

We are interested in the quantization of $\phi$, that is, 

\begin{equation}
 \phi'\coloneqq q(\phi) \in \Cl(\gg).
\end{equation}

\begin{theorem}[{\cite[Theorem 1.1]{Chen_Dirac_quadratic}}] \label{prop::phi_squared} The square of $\phi'$ is a constant given by 
\[
(\phi')^{2}=\tfrac{1}{24}\str (\ad_{\gg}(\Omega_{\gg})).
\] 
\end{theorem}

Whenever $\gg$ is basic classical, it is useful to rewrite the proposition using the Freudenthal–de Vries formula (see \cite{Meyer}). Fix a positive root system of $\gg$, and denote by $\rho$ the associated Weyl vector, that is the half sum of the positive roots. Then, the Freudenthal--de-Vries formula is
\begin{equation} \label{eq::Freudenthal_de_Vries_fd}
 \tfrac{1}{24}\str(\ad_{\gg}(\Omega_{\gg}))=B(\rho, \rho),
\end{equation}
and the square of $\phi'$ equals $B(\rho,\rho)$.

\subsubsection{Cubic Dirac Operator} Let $\{e_{a}\}$ be a homogeneous basis of $\gg$ with $B$-dual basis $\{e^{a}\}$. Note that $p(e_{a})=p(e^{a})$ for all $a$. The \emph{cubic Dirac operator} of $\gg$ is 
\begin{equation}
\Dirac_{\gg}\coloneqq \sum_{a}e^{a}\otimes e_{a}+1\otimes\phi'\in\mathcal{W}(\gg).
\end{equation}

\begin{lemma} \label{lemm::properies_D_l}
 The cubic Dirac operator $\Dirac_{\gg}$ satisfies the following properties:
 \begin{enumerate}
 \item[a)] The definition of $\Dirac_{\gg}$ is independent of the choice of a basis.
 \item[b)] $\Dirac_{\gg}$ is $\gg$-invariant, that is, $L_{x}\Dirac_{\gg}=0$.
 \item[c)] $\Dirac_{\gg}$ has bidegree $(\bar{1},\bar{0})$, \emph{i.e.}, it is an odd element of $\Weil$ relative to the total degree.
 \end{enumerate}
\end{lemma}

\begin{proof} 
a) We fix two bases $\{e_{a}\}$ and $\{f_{a}\}$ of $\gg$ with $B$-dual bases $\{e^{a}\}$ and $\{f^{a}\}$, respectively. We can express any $x \in \gg$ by $x=\sum_{a}B(x,f^{a})f_{a}=\sum_{a} B(f_{a},x)f^{a}$ for any $x \in \gg$ such that
\[
\begin{aligned}
\sum_{a} e^{a}\otimes e_{a}
 &=\sum_{a,b,c} B(f_{b},e^{a})B(e_{a},f^{c})\, f^{b}\otimes f_{c}
=\sum_{b,c} B\!\left(\sum_{a} B(f_{b},e^{a})e_{a},\, f^{c}\right) f^{b}\otimes f_{c}
 \\&=\sum_{b,c} B(f_{b},f^{c})\, f^{b}\otimes f_{c}
=\sum_{b} f^{b}\otimes f_{b}.
 \end{aligned}
\]
In particular, the definition $\Dirac_{\gg}$ is independent of the choice of a basis.

b) One has for any $x \in \gg$
\[
\begin{aligned}
L_{x}\left(\sum_{a}e^{a} \otimes e_{a}\right) &=\sum_{a}[x,e^{a}]\otimes e_{a} + \sum_{a}(-1)^{p(x)p(e_{a})}e^{a}\otimes [x,e_{a}] \\ &=\sum_{a,b}B(e_{b},[x,e^{a}])e^{b}\otimes e_{a} + \sum_{a,b} (-1)^{p(x)p(e_{a})}B([x,e_{a}],e^{b})e^{a}\otimes e_{b} \\ &=\sum_{a,b}B(e_{b},[x,e^{a}])e^{b}\otimes e_{a} - \sum_{a,b}B(e_{a},[x,e^{b}])e^{a}\otimes e_{b} \\ &=0,
\end{aligned}
\]
where we used invariance and supersymmetry of $B$. Moreover, $\phi'$ is $\gg$-invariant by Lemma~\ref{lemm::properties_phi_g_fd} and Theorem~\ref{thm::quantization_map}. This completes the proof of b).

c) By Lemma~\ref{lemm::properties_phi_g_fd} and the definition of $q$ (\emph{cf.} Theorem~\ref{thm::quantization_map}), it is immediate that $\phi$ has bidegree $(\bar{1},\bar{0})$ and total degree $\bar 1$. Moreover, parity is additive on simple tensors, e.g.
$p(e^{a}\otimes e_{a})=p(e^{a}\otimes 1)+p(1\otimes e_{a})$, so $\sum_{a}e^{a}\otimes e_{a}$ has bidegree $(\bar 1,\bar 0)$ and hence total degree $\bar 1$. Therefore $\Dirac_{\gg}$, being a sum of homogeneous elements of total degree $\bar 1$, is itself odd.
\end{proof}

The cubic Dirac operator $\Dirac_{\gg}$ has a particularly simple square. This theorem also appears in \cite{Meyer}. We give an alternative proof using the quantum Weil algebra.

\begin{theorem} \label{thm::square_absolute_Dirac}
The square of the cubic Dirac operator satisfies
\[
\Dirac_{\gg}^{2}
=
\Omega_{\gg}\otimes 1
+\tfrac{1}{24}\,\str_{\gg}\!\bigl(\ad_{\gg}(\Omega_{\gg})\bigr)(1\otimes 1)
\]
where $\Omega_{\gg}$ denotes the quadratic Casimir element of $\gg$.
\end{theorem}

\begin{proof}
We first show that $\Dirac_{\gg}^{2}\in\UE(\gg)\otimes 1$. It suffices to check $\iota_{x}\Dirac_{\gg}^{2}=0$ for all
$x\in\gg$. Indeed,
\[
\begin{split}
\iota_{x}\Dirac_{\gg}^{2}
&=\tfrac12[\iota_{x}\Dirac_{\gg},\Dirac_{\gg}]_{\WW}
=-\tfrac14[\gamma'(x),\Dirac_{\gg}]_{\WW}
=\tfrac12 L_{x}\Dirac_{\gg}=0,
\end{split}
\]
since $\Dirac_{\gg}$ is $\gg$-invariant.

Next we identify $\Dirac_{\gg}^{2}$ inside $\UE(\gg)$. Let $\pi:\WW(\gg)\to\UE(\gg)$ be the projection
$\id_{\UE(\gg)}\otimes\epsilon$, where $\epsilon:\Cl(\gg)=\CC\cdot 1\oplus\Cl^{>0}(\gg)\to\CC\cdot 1$. For brevity write
$s(x)\coloneqq x\otimes 1$. By Proposition~\ref{prop::phi_squared},
\[
\begin{split}
\Dirac_{\gg}^{2}\bmod\ker\pi
&=\sum_{a,b}(-1)^{p(e_{a})p(e_{b})}s(e^{a})s(e^{b})\,e_{a}e_{b}+(\phi')^{2}\bmod\ker\pi\\
&=\sum_{a,b}(-1)^{p(e_{b})}s(e^{a})s(e^{b})\,B(e_{a},e_{b})
+\tfrac{1}{24}\str_{\gg}\!\bigl(\ad_{\gg}(\Omega_{\gg})\bigr)\bmod\ker\pi\\
&=\Omega_{\gg}\otimes 1+\tfrac{1}{24}\str_{\gg}\!\bigl(\ad_{\gg}(\Omega_{\gg})\bigr)\bmod\ker\pi,
\end{split}
\]
where we used the graded product in $\WW(\gg)$ and the Clifford identity
\[
e_{a}e_{b}
=\tfrac12\bigl(e_{a}e_{b}+(-1)^{p(e_{a})p(e_{b})}e_{b}e_{a}\bigr)+\tfrac12[e_{a},e_{b}]
=B(e_{a},e_{b})+\tfrac12[e_{a},e_{b}]
\quad\text{in }\Cl(\gg),
\]
together with $\sum_{a}B(e_{a},e^{b})\,e_{a}=e_{b}$. This completes the proof.
\end{proof}

\begin{remark}
Using the Freudenthal--de~Vries formula \eqref{eq::Freudenthal_de_Vries_fd}, for a basic classical Lie superalgebra $\gg$ with Weyl vector $\rho$ one can rewrite the square as
\[
\Dirac_{\gg}^{2}=\Omega_{\gg}\otimes 1+B(\rho,\rho).
\]
\end{remark}

The following lemma will be used in subsequent constructions.

\begin{lemma} \label{lemm::commutation_relations_quantum_Weil_algebra}
In $\mathcal{W}(\gg)$, the following commutation relations hold:
\[
\begin{aligned}
[\Dirac_{\gg},\Dirac_{\gg}]_{\WW}&=2(\Omega_{\gg}\otimes1+B(\rho,\rho)(1\otimes1)),\\
[\gamma^{\mathcal{W}}(x),\Dirac_{\gg}]_{\WW}&=0
\end{aligned}
\]
for all $x \in \gg$. In addition, if $x \in \gg$ is even, one has 
\[
[1\otimes x,\Dirac_{\gg}]_{\WW}=2\gamma^{\mathcal{W}}(x).
\]
\end{lemma}

\begin{proof}
The cubic Dirac operator $\Dirac_{\gg}\in\mathcal{W}(\gg)$ has bidegree $(\bar{1},\bar{0})$, that is, total degree $\bar{1}$. Hence
\[
[\Dirac_{\gg},\Dirac_{\gg}]_{\WW}=2\Dirac_{\gg}^{2}
=2\bigl(\Omega_{\gg}\otimes1+B(\rho,\rho)(1\otimes1)\bigr).
\]
The commutation relation $[\gamma^{\mathcal{W}}(x),\Dirac_{\gg}]_{\WW}=0$ for all $x\in\gg$ follows from the identity
$
L_{x}=[\gamma^{\mathcal{W}}(x),\cdot]_{\WW}
$
in $\mathcal{W}(\gg)$ (see Lemma~\ref{lemm::contraction_and_Lie_derivative_in_Weil}). Since $\Dirac_{\gg}$ is $\gg$-invariant, the claim follows. Next, for $x\in\even$ the element $1\otimes x$ is odd with respect to the total degree. Since both $\Dirac_{\gg}$ and $1\otimes x$
are odd, we obtain
\[
\begin{split}
[1\otimes x,\Dirac_{\gg}]_{\WW}
&=2\,\iota_{x}\Dirac_{\gg}
=2\bigl(\iota_{x}(\sum_{a}e^{a}\otimes e_{a})+1\otimes q(\iota_{x}\phi)\bigr)\\
&=2\Bigl(x\otimes1-\tfrac{1}{2}\bigl(1\otimes q(\lambda(\ad_{x}))\bigr)\Bigr)
=2\gamma^{\WW}(x),
\end{split}
\]
where we used
\[
\iota_{x}\Bigl(\sum_{a}e^{a}\otimes e_{a}\Bigr)
=\sum_{a}(-1)^{p(e_{a})p(x)}\,e^{a}\otimes B(x,e_{a})
=\sum_{a}(-1)^{p(e_{a})}\,B(x,e_{a})\,e^{a}\otimes 1
=x\otimes 1,
\]
since $B(x,e_{a})=0$ unless $p(x)=p(e_{a})$.
\end{proof}

\subsection{Relative Cubic Dirac Operator and Weil Differential} \label{subsec::relative_Dirac}
We define relative cubic Dirac operators with respect to quadratic Lie subsuperalgebras $\ll$ of basic classical Lie superalgebras $\gg$. This generalizes the constructions in the literature in the case where $\ll$ is a Levi subsuperalgebra; see \cite{huang2005dirac,Schmidt_Noja}.

Let $\gg$ be a basic classical Lie superalgebra and let $\ll\subset \gg$
be a quadratic Lie subsuperalgebra, with non-degenerate supersymmetric invariant bilinear
form $B_{\ll}\coloneqq B|_{\ll}$. We refer to $(\gg,\ll)$ as a
\emph{quadratic pair}. For a quadratic pair $(\gg,\ll)$, there is an orthogonal decomposition
\begin{equation}
\gg=\ll\oplus\pp,\qquad \pp=\ll^{\perp},
\end{equation}
with respect to $B$. The restriction $B_{\pp}\coloneqq B|_{\pp}$ is non-degenerate. In particular, one can form the Clifford superalgebra $\Cl(\pp)\coloneqq \Cl(\pp;B_{\pp})$. For simplicity, assume that $\dim\pp=2(p_{0}+p_{1})$ is even, where $2p_{0}=\dim\pp_{\bar{0}}$ and $2p_{1}=\dim\pp_{\bar{1}}$. By invariance of $B$, one has $[\ll,\pp]\subseteq\pp$. 

As above, denote by $\Dirac_{\gg}\in\mathcal{W}(\gg)$ and $\Dirac_{\ll}\in\mathcal{W}(\ll)$ the cubic Dirac operators associated with $\gg$ and $\ll$, respectively. The relative cubic Dirac operator is obtained by combining these into a single operator via a natural embedding
$
j\colon\mathcal{W}(\ll)\longrightarrow\mathcal{W}(\gg),
$
induced by the adjoint action of $\ll$ on $\pp$.

For any $x \in \ll$, $\ad_{x} \in \osp(\gg)$ decomposes as the sum of $\ad_{x}^{\ll} \in \osp(\ll)$ and $\ad_{x}^{\pp} \in \osp(\pp)$. Set $\gamma_{\ll}'(x)\coloneqq -\tfrac{1}{2}q(\lambda(\ad_{x}^{\ll}))$ and $\nu_{\ast}(x)\coloneqq \gamma_{\pp}'(x)\coloneqq -\tfrac{1}{2}q(\lambda(\ad_{x}^{\pp})) \in \Cl(\pp)$. The latter defines a Lie superalgebra homomorphism $\nu_{\ast} : \ll \to \Cl(\pp)$. This map has an explicit description. As the dimension of $\pp$ is even, we can decompose $\pp$ in isotropic super subspaces $\pp=\uu \oplus \ubar$ compatible with the choice of a fixed positive system, that is $\uu\subset \nn^{+}$ and $\ubar\subset \nn^{-}$. Let $\Delta_{\ll} \subset \Delta$ denote the associated root system, and set $\Delta_{\ll}^{+}\coloneqq \Delta_{\ll} \cap \Delta^{+}$. The weights of $\uu$ and $\ubar$ then belong to $\Delta_{\pp}^{+}\coloneqq \Delta^{+}\cap \Delta \setminus \Delta_{\ll}$ and $\Delta_{\pp}^{-}\coloneqq \Delta^{-}\cap \Delta \setminus \Delta_{\ll}$. Set $\Delta^{+}_{\pp,\bar{0}}\coloneqq \Delta_{\pp}^{+}\cap \Delta_{\bar{0}}$ and $\Delta^{+}_{\pp,\bar{1}}\coloneqq \Delta_{\pp}^{+}\cap \Delta_{\bar{1}}$. Finally, set $\rho^{\uu}\coloneqq \tfrac{1}{2}\sum_{\alpha \in \Delta_{\pp,\bar{0}}^{+}} \alpha-\tfrac{1}{2}\sum_{\alpha \in \Delta_{\pp,\bar{1}}^{+}} \alpha$.

\begin{lemma} \label{lemm::embedding}
 The map $\nu_{\ast} : \ll \to \Cl(\pp)$ is a Lie superalgebra homomorphism. In particular, if $u_{1},\ldots, u_{p}$ is a basis of $\uu$ and $\overline{u}_{1}, \ldots, \overline{u}_{p}$ is a basis of $\overline{\uu}$, then 
 \[
 \nu_{\ast}(X)=\frac{1}{2} \sum_{j,k=1}^{p}(-1)^{p(u_{j})}B(X,[\overline{u}_{j},u_{k}])\overline{u}_{k}u_{j} + \begin{cases}
 \rho^{\uu}(X), \qquad &\text{if} \ X \in \hh, \\
 0, \qquad &\text{else}.
 \end{cases}
 \]
\end{lemma}

\begin{proof}
 The proof is verbatim the proof of Lemma~2.3.7 in \cite{Schmidt_Noja}.
\end{proof}

To distinguish the generators of $\WW(\ll)$ from those of $\Weil$, we write $\gamma^{\WW(\ll)}$ instead of $\gamma^{\WW}$ when necessary. Consider the natural embedding $j:\mathcal{W}(\ll)\to\mathcal{W}(\gg)$ defined on generators by
\begin{equation}\label{eq::embedding_j}
j(1\otimes x)=1\otimes x,\qquad j(\gamma^{\WW(\ll)}(x))=\gamma^{\WW}(x),\qquad x\in\ll.
\end{equation}
\begin{proposition} \label{lemm::properties_map_j_fd}
 The following hold: \begin{enumerate}
 \item[a)] The map $j : \mathcal{W}(\ll) \to \mathcal{W}(\gg)$ is a Lie superalgebra homomorphism. 
 
 \item[b)] $j: \mathcal{W}(\ll) \to \mathcal{W}(\gg)$ intertwines with contractions and Lie derivatives, that is, 
 \[
 \iota_{x}\circ j=j \circ \iota_{x}, \qquad L_{x}\circ j=j \circ L_{x} \qquad \forall x \in \ll.
 \]
 \end{enumerate}
\end{proposition}

\begin{proof} a) This is immediate from $[\gamma^{\WW}(x),\gamma^{\WW}(y)]_{\WW}=\gamma^{\WW}([x,y]_{\gg})$ for $x,y\in\gg$ (Lemma~\ref{lemm::commutation_relations_Weil_generators}).

 b) We prove the statement for Lie derivatives $L_{x}$ for $x \in \gg$. It is enough to prove the statement on the generators. One has 
 \[
 (L_{x}\circ j)(1\otimes y)=L_{x}(1\otimes y)=1\otimes L_{x}(y)=(j \circ L_{x})(1\otimes y).
 \]
 Moreover, using that $\gamma^{\WW}$ and $\nu_{*}$ are Lie superalgebra homomorphisms, and that $L_{x}=[\gamma^{\WW}(x),\cdot]_{\WW}$ by Lemma~\ref{lemm::contraction_and_Lie_derivative_in_Weil}, one obtains with Lemma~\ref{lemm::commutation_relations_Weil_generators} for all $x \in \ll$:
\[
(L_{x}\circ j)(\gamma^{\WW(\ll)}(y))=[\gamma^{\WW}(x), j(\gamma^{\WW(\ll)}(y))]_{\WW}=[\gamma^{\WW}(x),\gamma^{\WW}(y)]_{\WW}=\gamma^{\WW}([x,y]_{\gg})=(j \circ L_{x})(\gamma^{\WW(\ll)}(x)).
\]
 That $j$ intertwines with contractions is another direct calculation and will be omitted.
\end{proof}

Having fixed the notation, define the \emph{relative cubic Dirac operator} by
\begin{equation} \label{eq::definition_relative_Dirac_operator}
\Dirac_{\gg,\ll}\coloneqq \Dirac_{\gg}-j(\Dirac_{\ll})\in \Weil.
\end{equation}
In $\mathcal{W}(\gg)$, set
$
\mathcal{W}(\gg,\ll)\coloneqq (\mathfrak{U}(\gg)\otimes \Cl(\pp))^{\ll},
$
the $\ll$-basic subalgebra of the colour quantum Weil algebra, that is, the subalgebra consisting of all elements annihilated by $L_{x}$ and $\iota_{x}$ for all $x\in\ll$.

\begin{theorem}\label{thm::square_Dirac_fd} The following hold:
\begin{enumerate}
\item[a)] The operator $\Dirac_{\gg,\ll}$ lies in $\WW(\gg,\ll)$.
\item[b)] One has
\[
\Dirac_{\gg,\ll}^{2}
=\Omega_{\gg}-j(\Omega_{\ll})
+\tfrac{1}{24}\str_{\gg}(\Omega_{\gg})
-\tfrac{1}{24}\str_{\ll}(\Omega_{\ll}).
\]
\end{enumerate}
\end{theorem}

\begin{proof}
 a) Note $\gamma'(x)=-\tfrac{1}{2}q(\lambda(\ad_{x}))$, so that $\gamma'(x)=\gamma'_{\ll}(x)+\gamma'_{\pp}(x)$. One first verifies that
$
\iota_{x}\Dirac_{\gg}=\gamma^{\WW}(x).
$
Indeed,
\[
\begin{aligned}
\iota_{x}\Dirac_{\gg}
&=\sum_{a}e^{a}\otimes(-1)^{p(x)p(e_{a})}B(x,e_{a})
+1\otimes\iota_{x}\phi_{\gg}\\
&=\sum_{a}(-1)^{p(e_{a})}B(x,e_{a})\,e^{a}\otimes1
+q(\iota_{x}\phi_{\gg})\\
&=x\otimes1-\tfrac{1}{2}\bigl(1\otimes q(\lambda(\ad_{x}))\bigr),
\end{aligned}
\]
where the identity $2\iota_{x}\phi_{\gg}=-\lambda(\ad_{x})$ from Lemma~\ref{lemm::properties_phi_g_fd} is used.

By Lemma~\ref{lemm::properties_map_j_fd}, the Lie superalgebra homomorphism
$j\colon\mathcal{W}(\ll)\to\mathcal{W}(\gg)$ intertwines contractions and Lie derivatives. Hence, by definition of $j$,
\[
\begin{aligned}
\iota_{x}\Dirac_{\gg,\ll}
&=\iota_{x}\Dirac_{\gg}-j(\iota_{x}\Dirac_{\ll})=\gamma^{\WW}(x)-j(\gamma^{\WW(\ll)}(x))=0.
\end{aligned}
\]
Moreover, $L_{x}\Dirac_{\gg,\ll}=0$ for all $x\in\ll$ follows directly from the $\ll$-invariance of the cubic Dirac operators; see Lemma~\ref{lemm::properies_D_l}.

 b) In a), we have proven that $\Dirac_{\gg,\ll}$ lies in $\mathcal{W}(\gg,\ll)$. In particular, it commutes with the image of $j$ since 
 \[
[j(1\otimes x),\Dirac_{\gg,\ll}]_{\WW}=2\iota_{x}\Dirac_{\gg,\ll}=0, \qquad [j(\gamma^{\WW(\ll)}(x)),\Dirac_{\gg,\ll}]_{\WW}=[\gamma^{\WW}(x), \Dirac_{\gg,\ll}]_{\WW}=L_{x} \Dirac_{\gg,\ll}=0
 \]
where we used $\iota_{x}=\tfrac{1}{2}[1\otimes x,\cdot]_{\WW}$ by the Clifford relations, and Lemma~\ref{lemm::contraction_and_Lie_derivative_in_Weil}.
Thus, $[\Dirac_{\gg}, j(\Dirac_{\ll})]_{\WW}=0$. As $\Dirac_{\gg,\ll}$ is an odd element in $\mathcal{W}(\gg,\ll)$, we conclude 
 \[
 \Dirac_{\gg,\ll}^{2}=\tfrac{1}{2}[\Dirac_{\gg,\ll}, \Dirac_{\gg,\ll}]_{\WW}=\tfrac{1}{2}[\Dirac_{\gg}, \Dirac_{\gg}]_{\WW}-\tfrac{1}{2}j([\Dirac_{\ll},\Dirac_{\ll}]_{\WW})=\Dirac_{\gg}^{2}-j(\Dirac_{\ll}^{2}). 
 \]
 The statement now follows with Theorem~\ref{thm::square_absolute_Dirac}.
\end{proof}

\subsubsection{Weil Differential}\label{subsubsec::Weil_differential} We grade $\mathfrak{U}(\gg)\otimes \Cl(\pp)$ by the $\ZZ_{2}$-grading of $\Cl(\pp)$,
so that $(\mathfrak{U}(\gg)\otimes \Cl(\pp))_{\bar s}\coloneqq \mathfrak{U}(\gg)\otimes \Cl(\pp)_{\bar s}$. Via the diagonal embedding
$\ll\to \mathfrak{U}(\gg)\otimes \Cl(\pp)$, the adjoint action makes $\mathfrak{U}(\gg)\otimes \Cl(\pp)$ into an
$\ll$-supermodule. The $\ll$-invariants then coincide with $\WW(\gg,\ll)$. The space $\WW(\gg,\ll)$ is endowed with the obvious $\ZZ_{2}\times \ZZ_{2}$-grading from $\Weil$, making it into a colour Lie algebra.

Since $\Dirac_{\gg,\ll}$ has a simple square, it induces a differential on $\WW(\gg,\ll)$. For homogeneous
$A\in \mathfrak{U}(\gg)\otimes \Cl(\pp)$ set
\begin{equation}
 (d')^{\WW}(A)\coloneqq [\Dirac_{\gg,\ll},A]
=\Dirac_{\gg,\ll}A-(-1)^{p(A)}A\Dirac_{\gg,\ll}.
\end{equation}
As $\Dirac_{\gg,\ll}$ is $\ll$-invariant, $(d')^{\WW}$ preserves $\WW(\gg,\ll)$ and restricts to
$ d^{\WW}:\WW(\gg,\ll)\to \WW(\gg,\ll).
$ This operator is nilpotent, $(d^{\WW})^{2}=0$ (see \cite{Chen_Dirac_quadratic}), and compatible with the Lie derivative and contraction of $\Weil$ defined in \eqref{eq::contraction_Weil} and \eqref{eq::Lie_derivative_Weil}.

\begin{lemma} \label{lemm::properties_Weil_differential}
 The differential $d^{\WW}$ is compatible with $L_{x}$ and $\iota_{x}$ for any $x \in \ll$, that is, on $\WW(\gg,\ll)$ 
 \[
 d^{\mathcal{W}}L_{x}-(-1)^{p(x)}L_{x}d^{\mathcal{W}}=0, \qquad d^{\WW}\iota_{x}+(-1)^{p(x)}\iota_{x}d^{\WW}=L_{x}.
 \]
 Moreover, $d^{\mathcal{W}}$ is unique with this property.
\end{lemma}
\begin{proof}
We note that $L_{x}=[\gamma^{\WW}(x),\cdot]_{\WW}$. If we write $d^{\WW}=\ad_{\Dirac_{\gg,\ll}}^{\WW}$ and $L_{x}=\ad_{\gamma^{\WW}(x)}^{\WW}$, then 
\[
d^{\WW}L_{x}-(-1)^{p(x)}L_{x}d^{\WW}=\ad^{\WW}_{\Dirac_{\gg,\ll}}\ad^{\WW}_{\gamma^{\WW}(x)}-(-1)^{p(x)}\ad^{\WW}_{\gamma^{\WW}(x)}\ad^{\WW}_{\Dirac_{\gg,\ll}}=\ad^{\WW}_{[\gamma^{\WW}(x),\Dirac_{\gg,\ll}]_{\WW}}=\ad^{\WW}_{L_{x}\Dirac_{\gg,\ll}}=0
\]
where we used $\ll$-invariance of $\Dirac_{\gg,\ll}$ in the last equality. Similarly, one proves $d^{\WW}\iota_{x}+(-1)^{p(x)}\iota_{x}d^{\WW}=L_{x}.$
\end{proof}

Since $(d^{\WW})^{2}=0$, we may form its cohomology
$\ker d^{\WW}/\Im d^{\WW}$. Recall that $\mathfrak{Z}(\ll)$ denotes the center of the universal enveloping subalgebra of $\ll$. Let $\mathfrak Z(\ll_{\Delta})$ denote the image of $\mathfrak Z(\ll)$ in $\WW(\gg,\ll)$ under the
diagonal embedding. Then:

\begin{theorem}[{\cite[Theorem~6.2]{Chen_Dirac_quadratic}}]\label{thm::ker_d_center}
One has $\ker d^{\WW}=\mathfrak Z(\ll_{\Delta})\oplus \Im d^{\WW}$. Equivalently, the cohomology of $d^{\WW}$ is
isomorphic to $\mathfrak Z(\ll_{\Delta}) \cong \mathfrak{Z}(\ll)$.
\end{theorem}

\subsection{Oscillator Supermodule} \label{subsec::oscillator_supermodule}

At this stage the Dirac operator is only a ''\!\emph{Dirac element}`` $\Dirac_{\gg,\ll}$ in the Weil algebra
$\Weil$. To extract representation-theoretic information, we view it as an odd endomorphism of a suitable
representation space. For this, one must choose a $\Cl(\mathfrak p)$-supermodule on which the Clifford
generators act. The \emph{oscillator supermodule} provides the canonical choice (after fixing a
polarization $\mathfrak p=\uu\oplus \ubar$), realizing the Clifford action by creation and annihilation
operators. We briefly introduce it and refer to \cite{Schmidt_Noja} for all details.

Let $\gg$ be a basic classical Lie superalgebra. We adapt the notation from Section~\ref{subsec::relative_Dirac} and fix a quadratic Lie subsuperalgebra
$(\ll, B_{\ll}\coloneqq B\vert_{\ll}) \subset (\gg,B)$ such that $\gg=\ll\oplus \pp$, $\hh\subset \ll$ and $\pp$ is even-dimensional. Since we work over $\CC$ and $B_{\pp}$ is non-degenerate,
there exist subsuperspaces $\uu,\ubar\subset \pp$ with $\ubar\subset \nn^{-}$ and $\uu\subset \nn^{+}$ such that
\begin{equation}
\pp=\uu\oplus \ubar,\qquad \uu=\ubar^{\perp},\qquad \ubar=\uu^{\perp},
\end{equation}
with respect to $B_{\pp}$, \emph{i.e.}, $\uu$ and $\ubar$ are complementary Lagrangian subsuperspaces.
Moreover, $\pp$ decomposes into its even and odd parts as $\pp=\pp_{\bar 0}\oplus \pp_{\bar 1}$, and this
splitting is compatible with the polarization:
\begin{equation}
\pp_{\bar 0}=\uu_{\bar 0}\oplus \ubar_{\bar 0},\qquad
\pp_{\bar 1}=\uu_{\bar 1}\oplus \ubar_{\bar 1}.
\end{equation}
Note that these decompositions are not orthogonal direct sums with respect to $B$. Recall the definitions of $\Delta_{\ll}^{+}$ and $\rho^{\uu}$ in Section~\ref{subsec::relative_Dirac}.

We consider the Clifford superalgebra
$
\Cl(\pp)=\Cl(\pp_{\bar 0})\otimes \Cl(\pp_{\bar 1}).
$
To construct the oscillator supermodule, we treat the Clifford algebra $\Cl(\pp_{\bar 0})$ and the Weyl
algebra $\Cl(\pp_{\bar 1})$ separately.

First, consider the Clifford algebra $\Cl(\pp_{\bar 0})$. The subspaces $\uu_{\bar 0}$ and $\ubar_{\bar 0}$
are isotropic and complementary in $\pp_{\bar 0}$ with respect to $B_{\pp}$. Fix a basis
$u_{1},\dots,u_{p_0}$ of $\uu_{\bar 0}$ and $\overline{u}_{1},\dots,\overline{u}_{p_0}$ of $\ubar_{\bar 0}$ such that
\begin{equation}
B(\overline{u}_{j},u_{i})=B(u_{i},\overline{u}_{j})=\delta_{ij}\qquad (1\le i,j\le p_0).
\end{equation}
Via $B_{\pp}$ we identify $\ubar_{\bar 0}\cong \uu_{\bar 0}^{*}$, so that the dual basis element
$u_i^{*}$ corresponds to $\overline{u}_i$.

We consider the exterior (spinor) modules
\begin{equation} \label{eq::def_spinor_module}
\spin\coloneqq \bigwedge \uu_{\bar 0}, \qquad \spinbar\coloneqq \bigwedge \ubar_{\bar 0}.
\end{equation}
We focus on $\spinbar$; the discussion for $\spin$ is analogous.

On $\bigwedge \ubar_0$, we have a natural action of $\pp_{\bar{0}}$, where $\overline{u} \in \ubar_{\bar{0}}$ acts as the left exterior multiplication $\epsilon(\overline{u})$, and $u \in \uu_{\bar{0}}$ acts as the contraction $\iota(u)$ defined in Equation \eqref{eq::contraction}. By the universal property of the Clifford superalgebra $\Cl(\pp_{\bar{0}})$, this extends to an action of $\Cl(\pp_{\bar{0}})$ on $\spinbar$, which realizes $\spinbar$ as a $\Cl(\pp_{\bar{0}})$-module, called the \emph{spin module}. The following lemma is standard. 

\begin{lemma} \label{lemm::spin_simple}
 The spin module $\spinbar$ is the unique simple $\Cl(\pp_{\bar{0}})$-module, up to isomorphism. Additionally, $\spinbar$ contains a highest weight vector with respect to $\Delta_{\ll}^{+}$, whose weight is $\rho^{\uu_{\bar{0}}}$.
\end{lemma}

We endow $\spinbar$ with the $\ZZ_{2}$-grading induced from the tensor algebras $T(\ubar_{\bar{0}})$, so that $\spinbar=\Sbar^{\gg,\ll}_{\bar{0}}\oplus \Sbar^{\gg,\ll}_{\bar{1}}$. 
Moreover, we equip the spin module $\spinbar$ with a non-degenerate Hermitian form $\bracket_{\spinbar}$ for which
$u_i$ and $\overline{u}_i$ are mutually adjoint. To this end, fix a $\theta$-real form $\gg^{\RR}$ such that
\[
B_\theta(x,y)\coloneqq -B(x,\theta(y))
\]
is positive definite on $\gg^{\RR}$, and extend $B_\theta$ sesquilinearly to $\gg$, hence to $\pp$.
We may choose the dual bases $u_1,\dots,u_{p_0}\in \uu_{\bar 0}$ and $\overline{u}_1,\dots,\overline{u}_{p_0}\in \ubar_{\bar 0}$
orthonormal for $B_\theta$, so that $\theta(u_i)=-\overline{u}_i$ and $\theta(\overline{u}_i)=-u_i$.

The induced Hermitian form on tensor powers $T^n(\ubar_{\bar 0})$ descends to a Hermitian form on
$\spinbar=\bigwedge \ubar_{\bar 0}$, denoted $\bracket_{\spinbar}$. In particular:

\begin{lemma}[{\cite[Lemma 3.2.3]{Schmidt_Noja}}]\label{lemm::adjoint_spin} The following holds:
\begin{enumerate}
\item[a)]  $\bracket_{\spinbar}$ is a super positive definite super Hermitian form, that is,
 \[
\bra v,w \ket_{\spinbar}=(-1)^{p(v)p(w)}\overline{\bra w,v\ket}_{\spinbar}, \quad v,w \in \spinbar, 
 \]
 $\bra v,w \ket_{\spinbar}=0$ whenever $p(v) \neq p(w)$, and $\bra\cdot,\cdot\ket_{\spinbar}$ is positive definite on $\overline{S}^{\gg,\ll}_{\bar{0}}$ and $-i$-times positive definite on $\overline{S}^{\gg,\ll}_{\bar{1}}$.
 \item[b)] The adjoint of $u_{i}$ with respect to $\langle \cdot,\cdot\rangle_{\spinbar}$ is $\theta(u_{i})=-\overline{u}_{i}$, and the adjoint of $\overline{u}_{i}$ is $\theta(\overline{u}_{i})=-u_{i}$.
\end{enumerate}
\end{lemma}

We next consider the Weyl algebra $\Cl(\pp_{\bar 1})$. The bilinear form $B|_{\pp_{\bar 1}}$ is symplectic, and $\pp_{\bar 1}=\uu_{\bar 1}\oplus \ubar_{\bar 1}$ is a polarization into maximal isotropic subspaces. Accordingly, let $(X,Y)$ be either $(\uu_{\bar 1},\ubar_{\bar 1})$ or $(\ubar_{\bar 1},\uu_{\bar 1})$. Fix dual bases $e_{1},\dots,e_{p_{1}}$ of $X$ and $f_{1},\dots,f_{p_{1}}$ of $Y$, and set $\Weyl:=\Cl(\pp_{\bar 1})$. Then $\Weyl$ acts on $\CC[X]=\Sym(X)$ by multiplication by elements of $X$ and by contraction by elements of $Y$, determined by
\begin{equation}
f_{i}\cdot e_{j}=B_{\pp}(f_{i},e_{j}) \qquad (1\leq i,j\leq p_{1}).
\end{equation}
The $\Weyl$-module $\CC[X]$ is called the oscillator module. If $v\in \CC[X]$ is annihilated by all $f_{i}$, then $v\in \CC$. Hence $\CC[X]$ is simple. We write
\begin{equation}
M(\pp_{\bar 1}):=\CC[\uu_{\bar 1}]=\Sym(\uu_{\bar 1}), \qquad \overline{M}(\pp_{\bar 1}):=\CC[\ubar_{\bar 1}]=\Sym(\ubar_{\bar 1}).
\end{equation}
In the sequel, $\overline{M}(\pp_{\bar 1})$ will be used. The corresponding statements for $M(\pp_{\bar 1})$ are obtained in the same way. We endow
$\overline M(\pp_{\bar 1})$ with the $\ZZ_2$-grading by polynomial degree modulo $2$.

It is useful to give an equivalent formulation of the Weyl algebra and the oscillator module. Fix bases $\partial_1,\dots,\partial_{p_1}$ of $\uu_{\bar 1}$ and $x_1,\dots,x_{p_1}$ of $\ubar_{\bar 1}$
such that
\begin{equation}
B(x_k,\partial_l)=\tfrac12\delta_{kl}.
\end{equation}
Identifying $\partial_k$ with $\partial/\partial x_k$, the Weyl algebra $\Weyl=\Cl(\pp_{\bar 1})$ becomes the
algebra of polynomial differential operators in the variables $x_1,\dots,x_{p_1}$, with relations
\begin{equation}
[x_k,x_l]_{\mathscr{W}}=0,\qquad [\partial_k,\partial_l]_{\mathscr{W}}=0,\qquad [x_k,\partial_l]_{\mathscr{W}}=\delta_{kl}.
\end{equation}
Hence $\Weyl$ acts naturally on $\overline M(\pp_{\bar 1})=\CC[x_1,\dots,x_{p_1}]$.

The space $\overline{M}(\pp_{\bar{1}})=\CC[x_{1},\ldots,x_{p_{1}}]$ carries the \emph{(Bargmann--/Fischer--)Fock
Hermitian form} $\bracket_{\overline{M}(\pp_{\bar{1}})}$ uniquely determined by
\begin{equation}
\Big\langle \prod_{k=1}^{p_{1}}x_{k}^{p_{k}},\ \prod_{k=1}^{p_{1}}x_{k}^{q_{k}}\Big\rangle_{\overline{M}(\pp_{\bar{1}})}
=
\begin{cases}
\prod_{k=1}^{p_{1}}p_{k}! & \text{if } p_{k}=q_{k}\ \text{for all } k,\\
0 & \text{otherwise}.
\end{cases}
\end{equation}
It is positive definite and consistent, \emph{i.e.},
$\langle \overline{M}(\pp_{\bar{1}})_{0},\overline{M}(\pp_{\bar{1}})_{1}\rangle_{\overline{M}(\pp_{\bar{1}})}=0$.
Moreover, for $v,w\in \overline{M}(\pp_{\bar{1}})$ and $1\le k\le p_{1}$ one has
\begin{equation}
\langle \partial_{k} v,w\rangle_{\overline{M}(\pp_{\bar{1}})}
=
\langle v,x_{k} w\rangle_{\overline{M}(\pp_{\bar{1}})},
\qquad
\langle x_{k} v,w\rangle_{\overline{M}(\pp_{\bar{1}})}
=
\langle v,\partial_{k} w\rangle_{\overline{M}(\pp_{\bar{1}})},
\end{equation}
that is, the adjoint of $x_{k}$ is $\partial_{k}$, and the adjoint of
$\partial_{k}$ is $x_{k}$ for all $1\le k\le p_{1}$.

Finally, we define the \emph{oscillator supermodules} over $\Cl(\pp)$ by
\begin{equation} \label{eq::definition_M_p}
M(\pp)\coloneqq M(\pp_{\bar{1}}) \otimes \spin, \qquad \overline{M}(\pp)\coloneqq \overline{M}(\pp_{\bar{1}}) \otimes \spinbar
\end{equation}
and consider them as $\Cl(\pp)$-supermodules, with the $\ZZ_{2}$-grading induced by that of $\spin$ and $\spinbar$.

We conclude by recording basic properties of $\overline{M}(\pp)$ (the module $M(\pp)$ is analogous). 

\begin{proposition}[{\cite{Schmidt_Noja}}] \label{prop::comparison_action} The following hold:
 \begin{enumerate}
 \item[a)] The $\Cl(\pp)$-supermodule $\overline{M}(\pp)$ is simple. 
 \item[b)] The set of weights of $\overline{M}(\pp)$ is 
 \[
 \mathcal{P}_{\overline{M}(\pp)}=\{\rho^{\uu}-\ZZ_{+}[A] : A \subset \Delta^{+} \setminus \Delta_{\ll}^{+}\}.
 \]
 \item[c)] $\overline{M}(\pp)$ carries a positive definite super Hermitian form.
 \item[d)] Under the action of $\ll$ on $\overline{M}(\pp)$ induced by Lemma~\ref{lemm::embedding}, $\overline{M}(\pp)$ is a unitarizable $\ll$-supermodule. Moreover, there are $\ll$-supermodule isomorphisms
 \[
 \overline{M}(\pp) \cong \bigwedge \ubar \otimes \CC_{\rho^{\uu}} \cong \bigwedge \ubar_{\bar{0}} \otimes \operatorname{S}(\ubar_{\bar{1}}) \otimes \CC_{\rho^{\uu}};
 \]
 \emph{i.e.}, the action of $\ll$ on $\overline{M}(\pp)$ induced by $\nu_{\ast}$ and the adjoint action of $\ll$ on $\overline{M}(\pp)$ differs by a shift of $\rho^{\uu}$.
\end{enumerate}
\end{proposition}

Let $M$ be a $\gg$-supermodule. The $\ll$-relative cubic Dirac operator $\Dirac_{\gg,\ll}$ is naturally an element of $\mathcal{W}(\gg,\ll)$, and hence, via the induced action, an odd operator acting on $M\otimes \overline{M}(\pp)$, where we equip $M\otimes \overline{M}(\pp)$ with the $\ZZ_{2}$-grading of $\overline{M}(\pp)$.

\subsection{Dirac Cohomology} \label{subsec::Dirac_cohomology}
We define Dirac cohomology for the relative cubic Dirac operator $\Dirac_{\gg,\ll}$, where $\ll$ is the quadratic subalgebra from Section~\ref{subsec::relative_Dirac}, and we record the main results of \cite{Schmidt_Noja}. Although stated there for parabolic subalgebras, the arguments carry over verbatim to our setting.

Given an endomorphism $T\in\End_{\CC}(V)$ of a super vector space $V$, we define its cohomology by
\begin{equation}
 \H(T;V)\coloneqq \ker T/(\Im T\cap\ker T).
\end{equation}
We apply this to the (relative) cubic Dirac operator. Let $M$ be a $\gg$-supermodule and let $\overline{M}(\pp)$ be the oscillator
supermodule from Section~\ref{subsec::oscillator_supermodule}. On the $\mathfrak{U}(\gg)\otimes \Cl(\pp)$-supermodule
$M\otimes \overline{M}(\pp)$, the operator $\Dirac_{\gg,\ll}$ acts componentwise.
Moreover, $\Dirac_{\gg,\ll}$ supercommutes with $\ll$, so $\ker \Dirac_{\gg,\ll}$ is naturally an $\ll$-supermodule.

\begin{definition}\label{def::Dirac_cohom}
The \emph{Dirac cohomology} of a $\gg$-supermodule $M$ (with respect to $\Dirac_{\gg,\ll}$) is the $\ll$-supermodule
\[
\H_{\Dirac_{\gg,\ll}}(M)\coloneqq \H(\Dirac_{\gg,\ll};\, M\otimes \overline{M}(\pp))
=\ker \Dirac_{\gg,\ll}\big/\bigl(\ker \Dirac_{\gg,\ll}\cap \Im \Dirac_{\gg,\ll}\bigr).
\]
\end{definition}

We are particularly interested in admissible $(\gg, \ll)$-supermodules, \emph{i.e.}, $\gg$-supermodules that are $\ll$-semisimple. For these supermodules, the Dirac cohomology is $\ll$-semisimple by Section~\ref{subsec::unitarizable_supermodules}, as $\Dirac_{\gg,\ll}$ commutes with the action of $\ll$.

\begin{lemma} \label{lemm::completely_reducible_DC}
 Let $M$ be an admissible $(\gg, \ll)$-supermodule. Then $\H_{\Dirac_{\gg,\ll}}(M)$ is a semisimple $\ll$-supermodule, that is, $\H_{\Dirac_{\gg,\ll}}(M)$ is completely reducible as an $\ll$-supermodule.
\end{lemma}

Given the square formula of Theorem~\ref{thm::square_Dirac_fd}, Dirac cohomology is most effective for supermodules with
infinitesimal character. The following super-analogue of the Casselman--Osborne theorem is due to \cite{Schmidt_Noja}.

\begin{theorem}[\cite{Schmidt_Noja}] \label{thm::Casselman_Osborne} 
Let $\HC_{\gg,\ll}$ denote the Harish--Chandra monomorphism and let
$\res:\operatorname{S}^{W^{\gg}}\to \operatorname{S}^{W^{\ll}}$ be restriction. Then there exists an algebra homomorphism
$\eta_{\ll}:\ZG\to \mathfrak Z(\ll)$ such that for every $z\in\ZG$ one can write
\begin{equation}
 z\otimes 1=\eta_{\ll}(z)+\Dirac_{\gg,\ll}\!A+A\!\Dirac_{\gg,\ll}
\end{equation}
for some $A\in\WW(\gg,\ll)$, and the diagram
\[
\begin{tikzcd}
\ZG \arrow{r}{\eta_{\ll}} \arrow[swap]{d}{HC_{\gg}} & \mathfrak{Z}(\ll) \arrow{d}{HC_{\ll}} \\%
\operatorname{S}^{W^{\gg}} \arrow{r}{\res}& \operatorname{S}^{W^{\ll}}
\end{tikzcd}
\]
commutes. In particular, if $M$ has infinitesimal character $\chi_{\lambda}$, then $z\otimes 1$ acts on $\H_{\Dirac_{\gg,\ll}}(M)$ as
$\eta_{\ll}(z)\otimes 1$. Consequently, if $V\subset \H_{\Dirac_{\gg,\ll}}(M)$ is an $\ll$-subsupermodule with infinitesimal character
$\chi^{\ll}_{\mu}$, then $\chi_{\lambda}=\chi^{\ll}_{\mu}\circ \eta_{\ll}$.
\end{theorem}

Moreover, $\H_{\Dirac_{\gg,\ll}}(M)$ is trivial unless $M$ is a highest weight supermodule, and $\H_{\Dirac_{\gg,\ll}}(M)$ embeds into Kostant-type cohomology;
see \cite{Schmidt_Noja}. Explicit computations appear in \cite{SchmidtDirac,Schmidt_Noja}.

\subsection{The Unitary Case}
Of particular interest are the unitarizable highest weight supermodules over basic classical Lie superalgebras. As shown in
\cite{Carmeli_Fioresi_Varadarajan_HW}, these admit geometric realizations as spaces of sections of holomorphic super vector bundles
over Hermitian superspaces. For such supermodules the Dirac operator is selfadjoint, and the resulting Dirac cohomology is
especially simple. We record the main statements and refer to \cite[Section~5]{Schmidt_Noja} for proofs. We retain the notation introduced above.

\begin{theorem}\label{thm::Dirac_unitarizable}
Let $(\HH,\bracket_{\HH})$ be a unitarizable $\gg$-supermodule. Then:
\begin{enumerate}
\item[a)] $\Dirac_{\gg,\ll}$ is selfadjoint with respect to $\bracket_{\HH\otimes \overline{M}(\pp)}$; in particular $\ker\Dirac_{\gg,\ll}=\ker\Dirac_{\gg,\ll}^{k}$ for all $k\in\ZZ_{+}$.
\item[b)] If $\HH$ is simple, then $\HH\otimes \overline{M}(\pp)=\ker\Dirac_{\gg,\ll}\oplus \Im\Dirac_{\gg,\ll}$; in particular $\H_{\Dirac_{\gg,\ll}}(\HH)=\ker\Dirac_{\gg,\ll}$.
\end{enumerate}
\end{theorem}

\section{Semisimple Perturbations}\label{sec::semisimple_perturbations} \noindent We introduce a family of cubic Dirac operators and associated Laplace operators, parametrized by elements of the real span of the root system, extending the standard cubic Dirac operator. For each finite-dimensional supermodule, the Laplace operator associated to $\even$ determines a finite collection of $G_{\bar0}$-orbits encoding the $\even$-decomposition. On the other hand, comparison of a family of Laplace operators for $\gg$ with the square of the cubic Dirac operator gives energy operators detecting the degree of atypicality. Together, these operators detect both the $\even$-decomposition and the atypicality.

\subsection{Family of Cubic Dirac Operators and Laplace Operators}
We construct a family of cubic Dirac operators parametrized by the real span of the root system and study their properties. This family specializes to the ordinary cubic Dirac operator at the origin. We begin with the case in which $\mathfrak g$ is a complex simple Lie algebra. Our results generalize those of \cite{FT1,FT2,FT3}. Next, we define analogous families for basic classical Lie superalgebras with nontrivial odd part, as well as for the even part of basic classical Lie superalgebras. We study these families as elements in the color quantum Weil algebra $\Weil$ and operators acting on
$
M \otimes \overline{M}(\mathfrak n^-),
$
where $M$ is a finite-dimensional simple (super)module. 

\subsubsection{Complex Simple Lie Algebras} \label{subsubsec::complex_simple_Lie_algebras} Let $\gg$ be a finite-dimensional simple complex Lie algebra; that is, $\gg$ is isomorphic to
$\mathfrak{sl}(n;\mathbb C)$, $\mathfrak{so}(2n;\mathbb C)$, $\mathfrak{sp}(2n;\mathbb C)$,
$\mathfrak{so}(2n+1;\mathbb C)$, or one of the exceptional Lie algebras $\mathfrak g_2, \mathfrak f_4, \mathfrak e_6, \mathfrak e_7, \mathfrak e_8$.
We adapt the notation from above. In particular, $B$ denotes the Killing form on $\gg$. It restricts to a non-degenerate form on $\hh$, and allows us to identify $\hh$ with $\hh^{\ast}$. The induced non-degenerate bilinear form on $\hh^{\ast}$ will be denoted by the same symbol by abuse of notation.

Let $\hh_{\RR}^{\ast}\subset\hh^{\ast}$ denote the real span of $\Delta$. The induced form $B$ is positive definite on $\hh_{\RR}^{\ast}$. Using $B$, we identify $\hh_{\RR}$ with $\hh_{\RR}^{\ast}$ and denote the induced positive definite form on $\hh_{\RR}^{\ast}$ by the same symbol. In particular, for $\xi\in\hh_{\RR}^{\ast}$, we write $h_{\xi}\in\hh_{\RR}$ for the corresponding element, and conversely.

Let $G$ be the connected simply connected Lie group with Lie algebra $\gg$. Then $G$ acts naturally on $\gg$ by the adjoint representation $\Ad$ and on $\gg^{\ast}$ by the coadjoint representation $\Ad^{\ast}$. The musical isomorphism $\sharp: \gg \to \gg^{\ast}$ intertwines the adjoint and coadjoint representations, that is, 
\begin{equation}
 (\Ad_{g}x)^{\sharp}=\Ad^{\ast}_{g}(x^{\sharp}), \qquad g \in G,\ x \in \gg.
\end{equation}
This identification yields a canonical bijection between adjoint and coadjoint orbits.

Define
\begin{equation}
\Dirac_{\gg}(\xi)\coloneqq \Dirac_{\gg}+1\otimes h_{\xi}\in\mathcal{W}(\gg),
\qquad \xi\in\hh_{\RR}^{\ast}.
\end{equation}
Via the identification $\hh_{\RR}\cong\hh_{\RR}^{\ast}$, this gives a \emph{family of cubic Dirac operators} parametrized by $\hh_{\RR}^{\ast}$. For $\xi=0$, one recovers the cubic Dirac operator $\Dirac_{\gg}$, which we shall refer to as the absolute cubic Dirac operator.

\begin{proposition}\label{prop::square_family_even}
For any $\xi\in\hh_{\RR}^{\ast}$, the following hold:
\begin{enumerate}
\item[a)] The element $\Dirac_{\gg}(\xi)$ is $\hh$-invariant. In particular, it preserves all weight spaces.
\item[b)] For every $\xi\in\hh_{\RR}^{\ast}$, the operator $\Dirac_{\gg}(\xi)$ is an odd element in $\Weil$.
\item[c)] One has
\[
\Dirac_{\gg}(\xi)^{2}
=\Dirac_{\gg}^{2}+h_{\xi}\otimes1+2\bigl(1\otimes\nu_{\ast}(h_{\xi})\bigr)
+B(\xi,\xi)(1\otimes1).
\]
\end{enumerate}
\end{proposition}

\begin{proof}
 a) The cubic Dirac operator $D_{\gg}$ is $\gg$-invariant, so in particular $\hh$-invariant. The statement now follows as $\hh$ is abelian; so in particular $[h,\xi]=0$ for all $h \in \hh_{\RR}$. 

\item[b)] By definition of the $\ZZ_{2}$-grading on $\Weil$, the element $1\otimes x$ is odd for every $x\in\gg$; see \eqref{def::bidegree}. In particular, $1\otimes h_{\xi}$ is odd. Since $\Dirac_{\gg}$ is odd by Lemma~\ref{lemm::properies_D_l}, it follows that $\Dirac_{\gg}(\xi)$ is odd.
 
 c) Using b), we compute
 \[
 \begin{aligned}
 \Dirac_{\gg}(\xi)^{2} &=\tfrac{1}{2}[\Dirac_{\gg}(\xi), \Dirac_{\gg}(\xi)]_{\WW}=\Dirac_{\gg}^{2} + [D_{\gg},1\otimes h_\xi]_{\WW} + \tfrac{1}{2}[1\otimes h_\xi, 1 \otimes h_{\xi}]_{\WW} \\ &=\Dirac_{\gg}^{2}+2\gamma^{\mathcal{W}}(h_{\xi})+B(\xi,\xi)(1\otimes 1),
 \end{aligned}
 \]
 where we used the Clifford relations for $[1\otimes h_\xi,1\otimes h_\xi]_{\WW}=2B(h_{\xi},h_{\xi})(1\otimes 1)=2B(\xi,\xi)(1\otimes 1)$ and Lemma~\ref{lemm::commutation_relations_quantum_Weil_algebra}. 
\end{proof}

Fix a finite-dimensional simple $\gg$-supermodule $V$. Such supermodules are highest weight modules, parametrized by dominant integral weights $\Lambda\in P_{\bb}^{++}$; see Section~\ref{subsubsec::Highest_weight_supermodules}. Thus $V\cong L^{\bb}(\Lambda)$ for some Borel subalgebra $\bb$. From now on, we fix some $\bb=\hh\oplus\nn^{+}$ and omit the superscript.

Recall the oscillator module $\overline{M}(\pp)$. Since $\gg$ is even, one has $\overline{M}(\pp)\cong \overline{S}^{\gg,\nn^{-}}$; see \eqref{eq::def_spinor_module}. As the positive system is fixed, we write $S\coloneqq \overline{S}^{\gg,\nn^{-}}$. Then the operators
$\Dirac_{\gg}(\xi)$ define a \emph{family of cubic Dirac operators}
\begin{equation}
\hh^{\ast}_{\RR}\to \End_{\CC}\bigl(L(\Lambda)\otimes S\bigr),\qquad
\xi\mapsto \Dirac_{\gg}(\xi)=\Dirac_{\gg}+1\otimes h_{\xi}.
\end{equation}

Our object of study is the subset
$
\{\xi\in\hh_{\RR}^{\ast}\mid \ker \Dirac_{\gg}(\xi)\neq 0\}.
$
In general, $\ker \Dirac_{\gg}(\xi)$ is difficult to analyze directly, whereas Proposition~\ref{prop::square_family_even} gives effective control of $\Dirac_{\gg}(\xi)^{2}$. For $\Dirac_{\gg}$, one may also use unitarity: every finite-dimensional $\gg$-module is unitarizable, and $\Dirac_{\gg}$ is selfadjoint on $L(\Lambda)\otimes S$. This usual argument is not available for $\Dirac_{\gg}(\xi)$, since $\Dirac_{\gg}(\xi)$ is not selfadjoint in general.

\begin{lemma}
Let $\rr$ be the compact real form of $\gg$. If $h_{\xi}\in\hh_{\RR}\cap\rr$, then $\Dirac_{\gg}(\xi)$ is selfadjoint. In particular,
\[
\ker \Dirac_{\gg}(\xi)=\ker \Dirac_{\gg}(\xi)^{2}.
\]
\end{lemma}

In general, $h_{\xi}\notin\hh_{\RR}\cap\rr$. Accordingly, one considers $\Dirac_{\gg}(\xi)^{2}$ instead of $\Dirac_{\gg}(\xi)$. This is reflected in the notation
\begin{equation}
\hh_{\RR}^{\ast}\longrightarrow\End\bigl(L(\Lambda)\otimes S\bigr),
\qquad
\xi\longmapsto\Updelta_{\gg}(\xi)\coloneqq \Dirac_{\gg}(\xi)^{2}.
\end{equation}
Throughout the remainder of this article, the family $\Updelta_{\gg}$ is referred to as the \emph{family of Laplace operators}.

We compute the kernel of $\Laplace$ on $L(\Lambda)$. The sets of weights of $L(\Lambda)$ and of $S$ are contained in $\hh_{\RR}^{\ast}$. Since the restriction of $B$ to $\hh_{\RR}^{\ast}$ is positive definite, we set $\|\mu\|\coloneqq B(\mu,\mu)$ for $\mu\in\hh_{\RR}^{\ast}$. 

\begin{lemma} \label{lemm::action_Laplace_weight_space}
 $\Laplace$ acts on any weight space $(L(\Lambda)\otimes S)^{\mu}$ of weight $\mu \in \hh_{\RR}^{\ast}$ as the scalar 
\[
\norm{\Lambda+\rho}^{2}-\norm{\mu}^{2}+\norm{\mu+\xi}^{2}.
\]
\end{lemma}

\begin{proof}
By Proposition~\ref{prop::square_family_even},
\[
\Dirac_{\gg}(\xi)^{2}
=\Dirac_{\gg}^{2}+2\bigl(h_{\xi}\otimes1+1\otimes\nu_{\ast}(h_{\xi})\bigr)
+B(\xi,\xi)(1\otimes1).
\]
By Theorem~\ref{thm::square_absolute_Dirac},
$\Dirac_{\gg}^{2}=\Omega_{\gg}\otimes1+B(\rho,\rho)(1\otimes1)$, and $\Omega_{\gg}$
acts on $L(\Lambda)$ by the scalar
$B(\Lambda+\rho,\Lambda+\rho)-B(\rho,\rho)$. Hence $\Dirac_{\gg}^{2}$ acts on
$L(\Lambda)\otimes S$ by $\|\Lambda+\rho\|^{2}$. Let $\mu=\mu_{1}+\mu_{2}$ with
$\mu_{1},\mu_{2}\in\hh_{\RR}^{\ast}$ such that $L(\Lambda)^{\mu_{1}}\neq0$ and
$S^{\mu_{2}}\neq0$. Then $h_{\xi}$ acts on $(L(\Lambda)\otimes S)^{\mu}$ by
$B(\mu_{1},\xi)$ and $\nu_{\ast}(h_{\xi})$ by $B(\mu_{2},\xi)$
(Proposition~\ref{prop::comparison_action}). Consequently,
\[
\Dirac_{\gg}(\xi)^{2}
=\|\Lambda+\rho\|^{2}+2B(\mu,\xi)+B(\xi,\xi)
=\|\Lambda+\rho\|^{2}-\|\mu\|^{2}+\|\mu+\xi\|^{2}. \qedhere
\]
\end{proof}

\begin{theorem} \label{thm::main_theorem_families_complex_simple_even_case}
 The kernel of $\Laplace$ is trivial unless $\xi=-\Lambda-\rho$. In this case the kernel is one-dimensional and given by $L(\Lambda)^{\Lambda} \otimes S^{\rho}$.
\end{theorem}

\begin{proof}
 Using Lemma~\ref{lemm::action_Laplace_weight_space}, the kernel is the direct sum of all weight spaces $(L(\Lambda)\otimes S)^{\mu}$ such that 
 \[
 \norm{\Lambda+\rho}^{2}-\norm{\mu}^{2}+\norm{\mu+\xi}^{2}=0.
 \]
 We first note that $\norm{\mu+\xi}^{2}\geq 0$ since $\mu+\xi \in \hh_{\RR}^{\ast}$ by assumption. Moreover, $\norm{\Lambda+\rho}^{2}-\norm{\mu}^{2}\geq 0$ and equality holds precisely when $\mu=\Lambda+\rho$ \cite[Proposition 11.4]{Kac_infinite}. Thus, the weight space belongs to the kernel precisely when $\mu=\Lambda+\rho$ and $\norm{\Lambda+\rho+\xi}^{2}=0$. Since $B$ is positive definite, we conclude $\xi=- \Lambda-\rho$. As $\Lambda$ is the highest weight of $L(\Lambda)$ and $\rho$ is the highest weight of $S$, the weight space is one-dimensional and precisely given by $L(\Lambda)^{\Lambda}\otimes S^{\rho}$. 
\end{proof}

 The adjoint action $\Ad\colon G\to\Gl(\gg)$ extends canonically to actions $\Ad^{U}$ on $\UE(\gg)$ and $\Ad^{C}$ on $\Cl(\gg)$; the latter uses the invariance of $B$. This induces an action
$\Ad^{\mathcal W}$ of $G$ on $\mathcal{W}(\gg)$: 
\begin{equation}
\Ad^{\mathcal W}_{g}(x\otimes y)=\Ad_{g}^{U}(x)\otimes \Ad_{g}^{C}(y).
\end{equation}

View $\hh_{\RR}$ and $\hh_{\RR}^{\ast}$ as subspaces of $\gg$ and $\gg^{\ast}$ respectively, such that $\Ad_{g}(\hh_{\RR}) \subset \gg$ or $\Ad^{\ast}_{g}(\hh_{\RR}^{\ast}) \subset \gg^{\ast}$. Although our family was defined only on $\hh_{\RR}^{\ast}$, it may be viewed as the restriction of a more general family on $\gg^{\ast}$. More precisely, regarding $\hh_{\RR}^{\ast}$ as a subspace of $\gg^{\ast}$, one has families
\begin{equation}
\gg^{\ast}\to \mathcal{W}(\gg), \qquad \xi\mapsto \Dirac_{\gg}(\xi)
\qquad\text{and}\qquad
\xi\mapsto \Laplace,
\end{equation}
whose restrictions to $\hh_{\RR}^{\ast}$ recover the families considered above.

\begin{lemma} \label{lemm::invariance_family_under_G}
The maps
$
\Dirac_{\mathfrak g}(\xi), \Laplace : \mathfrak g^{\ast} \longrightarrow \mathcal W(\mathfrak g)$
are $G$-equivariant with respect to the coadjoint action $\Ad^{\ast}$ on
$\mathfrak g^{\ast}$ and the adjoint action $\Ad^{\mathcal W}$ on
$\mathcal W(\mathfrak g)$, \emph{i.e.},
\[
\Ad^{\mathcal W}_{g}\bigl(\Dirac_{\mathfrak g}(\xi)\bigr)
=
\Dirac_{\mathfrak g}\bigl(\Ad^{\ast}_{g}\xi\bigr), \quad \Ad^{\mathcal W}_{g}\bigl(\Laplace\bigr)
=
\Updelta_{\mathfrak g}\bigl(\Ad^{\ast}_{g}\xi\bigr)
\qquad g\in G,\ \xi\in\mathfrak g^{\ast}.
\]
\end{lemma}

\begin{proof}
Since $\Dirac_{\gg}$ is $\gg$-invariant, it is fixed by $\Ad^{\mathcal W}_{g}$.
Hence
\[
\Ad^{\mathcal W}_{g}\bigl(\Dirac_{\gg}(\xi)\bigr)
=\Dirac_{\gg}+1\otimes\Ad_{g}(h_{\xi}).
\]
Under the identification $\gg\cong\gg^{\ast}$ via $B$, the adjoint action of $G$
corresponds to the coadjoint action $\Ad^{\ast}$, so $\Ad_{g}(h_{\xi})=h_{\Ad^{\ast}_{g}\xi}$. Analogously, for the Laplace operator one has
\[
\Ad^{\mathcal{W}}_{g}(\Laplace)=\Dirac_{\gg}+2(\Ad_{g}(h_{\xi})\otimes 1 + 1 \otimes q(\lambda(\ad_{\Ad^{\ast}_{g}h_{\xi}})))+B(h_{\xi},h_{\xi})(1\otimes 1)=\Updelta_{\gg}(\Ad_{g}^{\ast}\xi),
\]
where we used $G$-equivariance of $q$ and $\lambda$ and invariance of $B$, that is, \[B(h_{\xi},h_{\xi})=B(\Ad_{g}(h_{\xi}),\Ad_{g}(h_{\xi}))=B(h_{\Ad_{g}^{\ast}\xi},h_{\Ad_{g}^{\ast}\xi}).\qedhere\]
\end{proof}

Since $G$ is connected and simply connected with Lie algebra $\gg$, any finite-dimensional $\gg$-module $M$ integrates to a representation $\pi_{M}\colon G\to\Gl(M)$ whose derived representation coincides with the given $\gg$-action. Let $\pi_{\Lambda}$ and $\pi_{S}$ denote the corresponding representations on $L(\Lambda)$ and $S$. Then $L(\Lambda)\otimes S$ carries the representation
\begin{equation}
\pi\colon G\to\Gl\bigl(L(\Lambda)\otimes S\bigr),\qquad
\pi(g)\coloneqq \pi_{\Lambda}(g)\otimes\pi_{S}(g).
\end{equation}

Using Lemma~\ref{lemm::invariance_family_under_G} and the compatibility of the
$\gg$- and $G$-module structures on $L(\lambda)\otimes S$, one obtains the
following.

\begin{lemma} \label{lemm::invariance_kernel_under_G_action}
For $g\in G$ and $\xi\in\hh_{\RR}^{\ast}$,
\[
\ker\Dirac_{\gg}(\Ad_{g}^{\ast}\xi)=\pi(g)\bigl(\ker\Dirac_{\gg}(\xi)\bigr), \qquad \ker \Updelta_{\gg}(\Ad_{g}^{\ast}\xi)=\pi(g)(\ker \Updelta_{\gg}(\xi)).
\]
In particular, for all $g\in G$ and $\xi\in\hh_{\RR}^{\ast}$, there are vector space isomorphisms
\[
\ker \Dirac_{\gg}(\Ad_{g}^{\ast}\xi)\cong\ker \Dirac_{\gg}(\xi),
\qquad
\ker \Updelta_{\gg}(\Ad_{g}^{\ast}\xi)\cong \ker \Updelta_{\gg}(\xi).
\]
\end{lemma}

For $x\in\gg^{\ast}$, denote by $\Ad_{G}^{\ast}(x)\subset\gg^{\ast}$ its coadjoint orbit. A coadjoint orbit is called \emph{semisimple} if it contains a semisimple element, \emph{i.e.}, an element lying in the dual of a Cartan subalgebra of $\gg$. Equivalently, $\Ad_{G}^{\ast}(x)\cap\tt^{\ast}\neq\{0\}$ for any Cartan subalgebra $\tt\subset\gg$. A semisimple coadjoint orbit is called \emph{$\RR$-split} if $\Ad_{G}^{\ast}(x)\cap\hh_{\RR}^{\ast}\neq\{0\}$. Note that for any other Cartan subalgebra $\tt\subset\gg$ with corresponding real span of roots
$\tt_{\RR}$, the spaces $\hh_{\RR}$ and $\tt_{\RR}$ are conjugate under the adjoint
action. In particular, for $\mu\in\hh_{\RR}^{\ast}$, the coadjoint orbit
$\Ad_{G}^{\ast}(\mu)$ is $\RR$-split and semisimple.

Combining Lemma~\ref{lemm::invariance_kernel_under_G_action} with Theorem~\ref{thm::main_theorem_families_complex_simple_even_case}, one obtains the following result.

\begin{theorem} \label{thm::main_theorem_localization}
Let $L(\Lambda)$ be a finite-dimensional simple $\gg$-module. Then $\ker\Laplace=\{0\}$ unless $\xi\in \Ad_{G}^{\ast}(-\Lambda-\rho)$. In particular, this assignment associates to each finite-dimensional simple $\gg$-module a unique $\RR$-split semisimple coadjoint orbit.
\end{theorem}

\subsubsection{Basic Classical Lie Superalgebra} Let $\gg$ be a basic classical Lie superalgebra with $\odd\neq\{0\}$. Fix the super Killing form $B$ on $\gg$, which identifies $\hh$ with $\hh^{\ast}$; in particular, $\hh\subset\even$. We denote by the same symbol the non-degenerate form on $\hh^{\ast}$. As above, let $\hh^{\ast}_{\RR}$ denote the real span of the root system $\Delta$. Unlike the finite-dimensional simple complex Lie algebra case, the restriction of $B$ to $\hh_{\RR}$ is not positive definite. As an example, consider $\gg=\mathfrak{psl}(2\vert 2)$ and $\mu\coloneqq (1,0\vert 0,-1) \in \hh^{\ast}_{\RR}$, which is isotropic, that is, $B(\mu,\mu)=0$. The even subalgebra $\even$ is reductive; let $G_{\bar{0}}$ be the connected simply connected Lie group with Lie algebra $\even$.

As in the complex simple even case, define a family of cubic Dirac operators and Laplace operators by
\begin{equation}
\hh_{\RR}^{\ast}\to\mathcal{W}(\gg),\qquad
\xi\mapsto
\begin{cases}
\Dirac_{\gg}(\xi)\coloneqq \Dirac_{\gg}+1\otimes h_{\xi},\\
\Laplace\coloneqq \Dirac_{\gg}(\xi)^{2}.
\end{cases}
\end{equation}
By the same argument as in the proof of Proposition~\ref{prop::square_family_even}, one obtains the following.

\begin{proposition}
For any $\xi\in\hh_{\RR}^{\ast}$, the following hold:
\begin{enumerate}
\item[a)] $\Dirac_{\gg}(\xi)$ is $\hh$-invariant.
\item[b)] For every $\xi\in\hh_{\RR}^{\ast}$, the operator $\Dirac_{\gg}(\xi)$ has bidegree $(\bar{1},\bar{0})$; in particular, it is odd with respect to the total degree.
\item[c)] $\Laplace=\Dirac_{\gg}^{2}+h_{\xi}\otimes1-\tfrac{1}{2}(1\otimes\nu_{\ast}(h_{\xi}))+B(\xi,\xi)(1\otimes1).$
\item[d)] $\Ad^{\mathcal{W}}_{g}\Dirac_{\gg}(\xi)=\Dirac_{\gg}(\Ad_{g}^{\ast}\xi)$ for all $g\in G_{\bar{0}}$.
\item[e)] $\Ad_{g}^{\mathcal{W}} \Laplace=\Updelta_{\gg}(\Ad_{g}^{\ast}\xi)$ for all $g \in G_{\bar{0}}$.
\end{enumerate}
\end{proposition}

As in the simple complex even case, the natural question is whether, for each finite-dimensional simple $\gg$-supermodule, the kernel of $\Laplace$ determines a unique $G_{\bar 0}$-coadjoint orbit. Recall from Section~\ref{subsubsec::Highest_weight_supermodules} that such supermodules are parametrized by dominant integral weights $\lambda\in\hh^{\ast}$ with respect to a Borel subalgebra $\bb=\bb_{\bar{0}}\oplus\bb_{\bar{1}}$, that is, $\lambda$ satisfies
\begin{equation}
B(\lambda+\rho_{\bar{0}},\alpha)>0\qquad\text{for all }\alpha\in\Delta_{\bar{0}}^{+},
\end{equation}
where $\Delta^{+}=\Delta_{\bar{0}}^{+}\sqcup\Delta_{\bar{1}}^{+}$. We fix a $\Lambda \in P^{++}$ for a fixed Borel $\bb$ and consider the finite-dimensional simple supermodule $L(\lambda)$ with even highest weight vector such that we view the family as
\begin{equation}
 \hh_{\RR}^{\ast} \to \End_{\CC}(L(\Lambda)\otimes \overline{M}(\pp)), \qquad \xi \mapsto \begin{cases}
\Dirac_{\gg}(\xi)\coloneqq \Dirac_{\gg}+1\otimes h_{\xi},\\
\Laplace\coloneqq \Dirac_{\gg}(\xi)^{2}.
\end{cases}
\end{equation}

Since $B$ restricted to $\hh_{\RR}$ is not positive definite, we cannot conclude as in the proof of Theorem~\ref{thm::main_theorem_families_complex_simple_even_case}. Indeed, we only have the following partial result which is proved as above:

\begin{lemma} The following hold:
 \begin{enumerate}
 \item[a)] For all $g\in G_{\bar{0}}$ and $\xi\in\hh_{\RR}^{\ast}$,
\[
\ker \Dirac_{\gg}(\Ad_{g}^{\ast}\xi)\cong \ker \Dirac_{\gg}(\xi),
\qquad
\ker \Updelta_{\gg}(\Ad_{g}^{\ast}\xi)\cong \ker \Updelta_{\gg}(\xi)
\]
as super vector spaces.
 \item[b)] $\Laplace$ acts on any weight space $(L(\Lambda)\otimes S)^{\mu}$ of weight $\mu \in \hh_{\RR}^{\ast}$ as the scalar
 \[
 B(\Lambda+\rho, \Lambda+\rho)-B(\mu,\mu)+B(\mu+\xi, \mu+\xi).
 \]
 \end{enumerate}
\end{lemma}

For a fixed weight $\mu\in\hh_{\RR}^{\ast}$ of $L(\Lambda)\otimes S$, the expression
\begin{equation}
B(\Lambda+\rho,\Lambda+\rho)-B(\mu,\mu)+B(\mu+\xi,\mu+\xi)
\end{equation}
may admit many solutions in $\xi$ as $B$ is indefinite. To overcome this, we consider appropriate modifications of families of Laplace operators defined with respect to the even subalgebra $\even$, using Theorem~\ref{thm::main_theorem_families_complex_simple_even_case}. This allows one to detect the $\even$-decomposition of a finite-dimensional simple $\gg$-supermodule. 

\subsubsection{Even Lie Subalgebra \texorpdfstring{$\even$}{}} 
Let $\gg=\even \oplus \odd$ be a basic classical Lie superalgebra with $\odd \neq \{0\}$.
The Lie algebra $\even$ is reductive.
Consequently, $\even$ admits a decomposition
\begin{equation}
\even
=
\gg_{\bar{0}}^{0}
\oplus
\gg_{\bar{0}}^{1}
\oplus
\cdots
\oplus
\gg_{\bar{0}}^{r},
\end{equation}
where $\gg_{\bar{0}}^{0}$ is either trivial or an abelian Lie algebra, and $\gg_{\bar{0}}^{i}$ is simple for all
$i=1,\ldots,r$. The possible cases are summarized in Table~\ref{tab:even_part}.

\begin{table}
\centering
\[
\begin{array}{|c|c|c|}
\hline
\gg & \gg_{\bar 0} & \text{Type of }\gg_{\bar 0}\\
\hline
\mathfrak{sl}(m\vert n),\ m\neq n
& \mathfrak{sl}(m)\oplus \mathfrak{sl}(n)\oplus \mathbb C
& \text{reductive, not semisimple}\\
\hline
\mathfrak{psl}(n\vert n), \ n\geq 2
& \mathfrak{sl}(n)\oplus \mathfrak{sl}(n)
& \text{semisimple}\\
\hline
\mathfrak{osp}(2\vert 2n)
& \mathfrak{so}(2)\oplus \mathfrak{sp}(2n)\cong \mathbb C\oplus\mathfrak{sp}(2n)
& \text{reductive, not semisimple}\\
\hline
\mathfrak{osp}(m\vert 2n),\ m\neq 2
& \mathfrak{so}(m)\oplus \mathfrak{sp}(2n)
& \text{semisimple}\\
\hline
D(2,1;\alpha), \ \alpha\neq 0,-1
& \mathfrak{sl}(2)\oplus \mathfrak{sl}(2)\oplus \mathfrak{sl}(2)
& \text{semisimple}\\
\hline
F(4)
& \mathfrak{so}(7)\oplus \mathfrak{sl}(2)
& \text{semisimple}\\
\hline
G(3)
& \gg_2\oplus \mathfrak{sl}(2)
& \text{semisimple}\\
\hline
\end{array}
\]
\caption{Even Lie subalgebras of basic classical Lie superalgebras}
\label{tab:even_part}
\end{table}

The summands $\gg_{\bar{0}}^{i}$ are mutually orthogonal with respect to restriction of the super Killing form $B$. Moreover, for any $i=1,\ldots,r$, the restriction of $B$ to $\gg_{\bar{0}}^{i}$ is a nonzero scalar multiple of the Killing form of
$\gg_{\bar{0}}^{i}$, and the restriction to $\gg_{\bar{0}}^{0}$ is non-trivial.
We denote by $B_{i}$ the induced bilinear form on $\gg_{\bar{0}}^{i}$ and $(\gg_{\bar{0}}^{i})^{\ast}$, respectively.

Fix a Cartan subalgebra $\hh$ of $\gg$; then $\hh\subset\even$. Accordingly, $\hh$ decomposes compatibly with $\even$ as
\begin{equation}
\hh=\hh^{0}\oplus\hh^{1}\oplus\cdots\oplus\hh^{r},
\end{equation}
where $\hh^{i}$ is a Cartan subalgebra of $\gg_{\bar 0}^{i}$ for $i\neq 0$, and $\hh^{0}=\gg_{\bar 0}^{0}$. The restriction of $B_{i}$ to $\hh^{i}$ induces a
non-degenerate symmetric bilinear form on $(\hh^{i})^{\ast}$, defined analogously
to the form obtained by restricting $B$ to $\hh$; by abuse of notation, this form
is again denoted by $B_{i}$. Consequently,
\[
\hh^{\ast}=(\hh^{0})^{\ast}\oplus(\hh^{1})^{\ast}\oplus\cdots\oplus(\hh^{r})^{\ast},
\]
and for $i\neq0$ the space $(\hh^{i})^{\ast}$ is identified with $\hh^{i}$ via
$B_{i}$. With respect to this decomposition, each $\lambda\in\hh^{\ast}$ admits a unique decomposition
$\lambda=(\lambda^{0},\lambda^{1},\ldots,\lambda^{r})$, where $\lambda^{i}$ is the
restriction of $\lambda$ to $\hh^{i}$. For $\lambda,\mu\in\hh^{\ast}$, the bilinear
form $B$ satisfies
\begin{equation}
B(\lambda,\mu)=\sum_{i=0}^{r}B_{i}(\lambda^{i},\mu^{i}).
\end{equation}
In the sequel, the subscript $i$ is omitted whenever no ambiguity arises.

Let $\hh_{\RR}^{\ast}$ denote the real span of $\Delta$ with dual space $\hh_{\RR}$.
Then $\hh_{\RR}
=\hh_{\RR}^{1}
\oplus
\cdots
\oplus
\hh_{\RR}^{r}.
$ In addition, let $G_{\bar{0}}$ be the connected simply connected Lie group with Lie algebra $\even$.
Then
\begin{equation}
G_{\bar{0}}
=
G_{\bar{0}}^{0}
\times
G_{\bar{0}}^{1}
\times
\cdots
\times
G_{\bar{0}}^{r},
\end{equation}
where $G_{\bar{0}}^{i}$ is the connected simply connected Lie group with Lie algebra $\gg_{\bar{0}}^{i}$. We have a natural coadjoint action of $G_{\bar{0}}$ on $\even^{\ast}$.

Let $\Delta$ be the root system of $\gg$ with respect to $\hh$, and let $\Delta_{\bar{0}}$ denote the root
system of $\even$.
The root systems $\Delta_{\bar{0}}^{1}, \ldots, \Delta_{\bar{0}}^{r}$ are identified with subsets of $\Delta_{\bar{0}}$ such that $\Delta_{\bar{0}}=\Delta_{\bar{0}}^{1} \sqcup \ldots \sqcup \Delta_{\bar{0}}^{r}$. A choice of positive system $\Delta^{+}\subset\Delta$ induces positive systems on each
$\Delta_{\bar{0}}^{i}$.
With this choice, the positive system of $\Delta_{\bar{0}}$ is the disjoint union of the positive systems of the
$\Delta_{\bar{0}}^{i}$.
The even Weyl vector $\rho_{\bar{0}}$ is defined accordingly, and decomposes as
$
\rho_{\bar{0}}=(0,\rho^{1}_{\bar{0}},\ldots,\rho^{r}_{\bar{0}}).
$

Fix a Borel subalgebra $\bb_{\bar{0}}\subset\even$.
Then $\bb_{\bar{0}}$ decomposes as
$
\bb_{\bar{0}}=\bb_{\bar{0}}^{0}\oplus\bb_{\bar{0}}^{1}\oplus\cdots\oplus\bb_{\bar{0}}^{r}
$ with $\bb_{\bar{0}}^{0}=\gg_{\bar{0}}^{0}$.
In particular, the set of dominant integral weights parameterizing finite-dimensional simple $\even$-modules is
\begin{equation}
P^{++}_{\bb_{\bar{0}}}
=P^{++}_{\bb_{\bar{0}}^{0}}\times
\ldots \times 
P^{++}_{\bb_{\bar{0}}^{r}},
\end{equation}
where $P_{\bb_{\bar{0}}^{0}}^{++}$ is identified with $(\bb_{\bar{0}}^{0})^{\ast}\coloneqq \Hom_{\CC}(\bb_{\bar{0}}^{0},\CC)$ and each $P_{\bb_{\bar{0}}^{i}}^{+}$ for $i=1,\ldots, r$ is described in Section~\ref{subsubsec::Highest_weight_supermodules}.

Let $L_{0}(\Lambda)$ be a finite-dimensional simple $\even$-module.
Then $L_{0}(\Lambda)$ is a highest weight module with respect to $\bb_{\bar{0}}$ with highest weight $\Lambda \in P_{\bb_{\bar{0}}}^{++}$.
Moreover, $L_{0}(\Lambda)$ decomposes as an outer tensor product
\begin{equation}
L_{0}(\Lambda)
=
L(\Lambda^{0};\gg_{\bar{0}}^{0})\boxtimes L(\Lambda^{1};\gg_{\bar{0}}^{1})
\boxtimes
\cdots
\boxtimes
L(\Lambda^{r};\gg_{\bar{0}}^{r}),
\end{equation}
where each $L(\Lambda^{i};\gg_{\bar{0}}^{i})$ is a finite-dimensional simple module over
$\gg_{\bar{0}}^{i}$.

Since each $(\gg_{\bar 0}^{i},B_{i})$ is a quadratic Lie algebra, there is a cubic Dirac operator $\Dirac_{\gg_{\bar 0}^{i}}\in\WW(\gg_{\bar 0}^{i})$. As above, this defines families
\begin{equation}
(\hh^{i}_{\RR})^{\ast}\to\mathcal{W}(\gg_{\bar 0}^{i}),\qquad
\xi\mapsto
\begin{cases}
\Dirac_{\gg_{\bar 0}^{i}}(\xi^{i})\coloneqq \Dirac_{\gg_{\bar 0}^{i}}+1\otimes h_{\xi^{i}},\\
\Updelta_{\gg_{\bar 0}^{i}}(\xi^{i})\coloneqq \Dirac_{\gg_{\bar 0}^{i}}(\xi^{i})^{2}.
\end{cases}
\end{equation}
For $i\neq j$, the elements $\Dirac_{\gg_{\bar 0}^{i}}(\xi^{i})$ and $\Dirac_{\gg_{\bar 0}^{j}}(\xi^{j})$ anti-commute in $\WW(\even)$, via the embedding \eqref{eq::embedding_j}.
Hence, for $\xi=(\xi^{0},\ldots,\xi^{r})\in\hh_{\RR}^{\ast}$, we define
\begin{equation}
\Updelta_{\even}(\xi)\coloneqq \sum_{i=0}^{r}\Updelta_{\gg_{\bar 0}^{i}}(\xi^{i})\in\WW(\even),
\end{equation}
which equals $\bigl(\sum_{i=0}^{r}\Dirac_{\gg_{\bar 0}^{i}}(\xi^{i})\bigr)^{2}$. 

We are interested in the action of $\Updelta_{\even}(\xi)$ on $L_{0}(\Lambda)\otimes S$, where now $S\coloneqq \overline{M}(\nn_{\bar{0}}^{-})$. In general, the kernel of $\Updelta_{\even}(\xi)$ does not coincide with the intersection of the kernels of the operators $\Updelta_{\gg_{\bar 0}^{i}}(\xi^{i})$. As a consequence, Theorem~\ref{thm::main_theorem_localization} is not directly applicable. We address this by introducing a modified family of operators, using that $S$ is unitarizable and that $L_{0}(\Lambda)$ is finite-dimensional.

Since $L_{0}(\Lambda)$ is finite-dimensional, we can fix a Hermitian inner product on $\langle \cdot,\cdot\rangle_{L_{0}(\Lambda)}$ on $L_{0}(\Lambda)$. Let $\langle \cdot,\cdot\rangle_{L_{0}(\Lambda)\otimes S}\coloneqq \langle \cdot,\cdot\rangle_{L_{0}(\Lambda)}\langle\cdot,\cdot\rangle_{S}$ be the induced Hermitian inner product on $L_{0}(\Lambda)\otimes S$. With respect to this product, we denote by $(\cdot)^{\dagger}$ the adjoint of an endomorphism. We then define 
\begin{equation}
 \widetilde{\Updelta}_{\even}(\xi)\coloneqq \sum_{i=0}^{r}(\Updelta_{\gg_{\bar{0}}^{i}}(\xi^{i}))^{\dagger}\Updelta_{\gg_{\bar{0}}^{i}}(\xi^{i}) \in \End_{\CC}(L_{0}(\Lambda)\otimes S).
\end{equation}

\begin{lemma} \label{lemm::properties_generalized_Laplace}
The operator $\widetilde{\Updelta}_{\even}(\xi)$ satisfies the following properties:
\begin{enumerate}
 \item[a)] $\widetilde{\Updelta}_{\even}(\xi)$ is positive semi-definite, that is,
 \[
 \langle \widetilde{\Updelta}_{\even}(\xi)v,v\rangle_{L_{0}(\Lambda)\otimes S}\geq 0
 \qquad\text{for all } v\in L_{0}(\Lambda)\otimes S.
 \]
 \item[b)] One has
 \[
 \ker \widetilde{\Updelta}_{\even}(\xi)=\bigcap_{i=0}^{r}\ker \Updelta_{\gg_{\bar 0}^{i}}(\xi^{i}).
 \]
\end{enumerate}
\end{lemma}

\begin{proof}
For $v\in L_{0}(\Lambda)\otimes S$, one has
\[
\langle \widetilde{\Updelta}_{\even}(\xi)v,v\rangle_{L_{0}(\Lambda)\otimes S}
=\sum_{i=0}^{r}\bigl\langle (\Updelta_{\gg_{\bar 0}^{i}}(\xi^{i}))^{\dagger}\Updelta_{\gg_{\bar 0}^{i}}(\xi^{i})v,v\bigr\rangle_{L_{0}(\Lambda)\otimes S}
=\sum_{i=0}^{r}\bigl\langle \Updelta_{\gg_{\bar 0}^{i}}(\xi^{i})v,\Updelta_{\gg_{\bar 0}^{i}}(\xi^{i})v\bigr\rangle_{L_{0}(\Lambda)\otimes S},
\]
hence
\[
\langle \widetilde{\Updelta}_{\even}(\xi)v,v\rangle_{L_{0}(\Lambda)\otimes S}\geq 0.
\]
This proves~\textup{a)}. For~\textup{b)}, positivity of the Hermitian inner product implies that
\[
v\in \ker \widetilde{\Updelta}_{\even}(\xi)
\quad\Longleftrightarrow\quad
\langle \widetilde{\Updelta}_{\even}(\xi)v,v\rangle_{L_{0}(\Lambda)\otimes S}=0
\quad\Longleftrightarrow\quad
\Updelta_{\gg_{\bar 0}^{i}}(\xi^{i})v=0 \text{ for all } i=0,\ldots,r,
\]
which is equivalent to
$
v\in \bigcap_{i=0}^{r}\ker \Updelta_{\gg_{\bar 0}^{i}}(\xi^{i}).
$
\end{proof}

\begin{theorem}\label{thm::main_family_even}
Let $L_{0}(\Lambda)$ be a finite-dimensional simple $\even$-module with highest weight
$\Lambda\in\hh_{\RR}^{\ast}$. Then $\widetilde{\Updelta}_{\even}(\xi)$ has trivial kernel unless
$\xi=-\Lambda-\rho_{\bar{0}}$. In particular,
\[
\Ker\widetilde{\Updelta}_{\even}(\xi)\neq\{0\}
\iff
\xi\in\Ad_{G}^{\ast}(-\Lambda-\rho_{\bar{0}})
\coloneqq \Ad_{G_{\bar{0}}}^{\ast}(-\Lambda-\rho_{\bar{0}}).
\]
\end{theorem}

\begin{proof}
By Lemma~\ref{lemm::properties_generalized_Laplace} is the kernel of $\widetilde{\Updelta}_{\even}(\xi)$ the intersection of the kernels of
$\Updelta_{\gg_{\bar{0}}^{i}}(\xi^{i})$. By
Theorem~\ref{thm::main_theorem_families_complex_simple_even_case},
$\Updelta_{\gg_{\bar{0}}^{i}}(\xi^{i})$ has trivial kernel unless
$\xi^{i}=-\Lambda^{i}-\rho_{\bar{0}}^{i}$ for $i=1,\ldots,r$, and
$\ker\Updelta_{\gg_{\bar{0}}^{i}}(\xi^{i})\neq\{0\}$ iff
$\xi\in\Ad_{G_{\bar{0}}^{i}}(-\Lambda^{i}-\rho_{\bar{0}}^{i})$. For the abelian
factor $\gg_{\bar{0}}^{0}$, one has $\rho_{\bar{0}}^{0}=0$ and
$L(\Lambda^{0};\gg_{\bar{0}}^{0})$ is one-dimensional, with
$\Updelta_{\gg_{\bar{0}}^{0}}(\xi^{0})$ acting by the scalar
$B_{0}(\Lambda^{0}+\xi^{0},\Lambda^{0}+\xi^{0})$. Since $B_{0}$ is non-degenerate,
this scalar vanishes only for $\xi^{0}=-\Lambda^{0}$. The claim follows.
\end{proof}

\begin{remark}
Let $\rr$ be the compact real form of $\even$. If $h_{\xi}\in \hh_{\RR}\cap \rr$, then $\Dirac_{\gg_{\bar{0}}^{i}}(\xi)$ is self-adjoint with respect to the Hermitian form. Consequently, $\Updelta_{\gg_{\bar{0}}^{i}}(\xi)$ is non-negative for all $i=0,\ldots,r$. Therefore Theorem~\ref{thm::main_family_even} remains valid with $\widetilde{\Updelta}_{\even}(\xi)$ replaced by $\Updelta_{\even}(\xi)$.
\end{remark}

We now apply the preceding result to finite-dimensional $\gg$-supermodules. We thus embed $\mathcal{W}(\even)$ into $\mathcal{W}(\gg)$. Fix a
finite-dimensional simple $\gg$-supermodule $L(\Lambda)$ with respect to a fixed
positive Borel $\bb$. Upon restriction to the even subalgebra $\even$,
the module $L(\Lambda)$ decomposes as a finite direct sum of simple
finite-dimensional $\even$-modules,
\begin{equation}
L(\Lambda)\Big\vert_{\even}=\bigoplus_{\mu}L_{0}(\mu)^{n(\mu)},
\end{equation}
where the sum runs over the highest weights $\mu$ of the $\even$-constituents and
$n(\mu)$ denotes their multiplicities. Under the embedding $\WW(\even)\hookrightarrow \Weil$, the modified family
$\widetilde{\Updelta}_{\even}(\xi)$ is viewed as a family in $\Weil$
acting on $L(\Lambda)\otimes \overline{M}(\nn^{-})$ by
\begin{equation} \label{eq::action_Laplace_tilde}
\widetilde{\Updelta}_{\even}(\xi) \otimes 1_{\overline{M}(\nn_{\bar{1}}^{-})} \in \End_{\CC}((L(\Lambda)\otimes S)\otimes \overline{M}(\nn_{\bar{1}}^{-})),
\end{equation}
where we use the $\ZZ_{2}$-graded tensor product decomposition
$\overline{M}(\nn^{-})=S\otimes \overline{M}(\nn_{\bar{1}}^{-})$
and identify $S$ with $S^{\gg,\nn_{\bar{0}}^{-}}$. This family then detects precisely the $\even$-constituents of $L(\Lambda)$ via their highest
weights, as described in Theorem~\ref{thm::main_family_even}.

\begin{corollary} \label{cor::kernel_Laplace_even_constituents}
 Let $L(\Lambda)$ be a finite-dimensional simple $\gg$-supermodule with highest weight $\Lambda$ such that $L(\Lambda)\big\vert_{\even}=\bigoplus_{\mu} L_{0}(\mu)^{n(\mu)}$. Then 
 \[
 \ker \widetilde{\Updelta}_{\even}(\xi) \neq \{0\} \Leftrightarrow \xi \in \bigcup_{\mu \,:\, n(\mu)\neq 0}\Ad_{G_{\bar{0}}}^{\ast}(-\mu-\rho_{\bar{0}}).
 \]
 If $\xi=- \mu-\rho_{\bar{0}}$, then the kernel is $L(\Lambda)^{\mu}\otimes S^{\rho_{\bar{0}}} \otimes \overline{M}(\nn_{\bar{1}}^{-})$.
\end{corollary}

\begin{example} We consider the projective special linear Lie superalgebra $\mathfrak{psl}(2\vert 2)$. Fix the distinguished positive system such that the Weyl vector takes the form $\rho=(1,0\vert 0,-1)$. We consider $L(\Lambda)$ with highest weight $\Lambda=(1,0\vert 0,-1)$. Let $\xi \in \hh_{\RR}^{\ast}$, where $\hh$ is the Cartan subalgebra of diagonal matrices. Then 
\[
\ker \widetilde{\Updelta}_{\even}(\xi) \neq \{0\} \Leftrightarrow -\xi \in \Ad_{G_{\bar{0}}}^{\ast}(X+\rho_{\bar{0}}), \quad X=\{\Lambda, (0,0\vert 1,-1), (1,-1\vert 0,0), (0,-1\vert 1,0)\}.
\]
\end{example}

\subsection{A Detecting Family}

Let $\gg=\even\oplus\odd$ be a basic classical Lie superalgebra with $\odd\neq\{0\}$, and retain the notation introduced above. In the previous section, a family of cubic Dirac operators
\begin{equation}
\Dirac_{\gg}(\xi)\coloneqq \Dirac_{\gg}+1\otimes h_{\xi},\qquad \xi\in\hh_{\RR}^{\ast},
\end{equation}
was introduced; its invariance properties were established; and the kernel of the associated family of Laplace operators $\Updelta_{\gg}(\xi)=\Dirac^2_{\gg}(\xi)$ was analyzed. When restricted to $\even$, a resulting Laplace family $\widetilde{\Updelta}_{\even}(\xi)$ detects the $\even$-decomposition of a finite-dimensional simple $\gg$-supermodule, but does not provide further information. 

For $\gg$ with $\odd \neq \{0\}$, the representation theory differs from the classical case because of the presence of isotropic odd roots. This is reflected in Kac’s distinction between typical and atypical simple finite-dimensional supermodules, defined in terms of the relation between the highest weight and the isotropic roots; see Section~\ref{subsec::atyicality}. Typical representations retain many of the features familiar from ordinary semisimple Lie algebras, whereas atypical representations do not. In particular, the even subalgebra no longer determines all relevant structure. This leads to additional invariants, and in this paper we introduce an energy operator for the detection of atypicality.

Recall that a weight $\Lambda$ is \emph{typical} if $A_{\Lambda}=\emptyset$ and \emph{atypical} otherwise, where $A_{\Lambda}=\{\alpha\in\Delta_{\bar{1}}^{+}:B(\Lambda+\rho,\alpha)=0, \ B(\alpha,\alpha)=0\}$. The \emph{degree of atypicality} of $\Lambda$, denoted by $\at(\Lambda)$, is the maximal cardinality of a linearly independent set of mutually orthogonal isotropic roots $\alpha\in\Delta_{\bar{1}}^{+}$ such that $B(\Lambda+\rho,\alpha)=0$.

To detect this information, we introduce an additional family of operators measuring the deviation from the Laplace operator $\Updelta_{\gg}$, called \emph{energy operators}. These encode atypicality and will be combined with $\widetilde{\Updelta}_{\even}(\xi)$ below. We recall that
\begin{equation}
\Dirac_{\gg}(\xi)\coloneqq \Dirac_{\gg}+1\otimes h_{\xi},\qquad
\Laplace\coloneqq \Dirac_{\gg}(\xi)^{2}, \qquad \xi \in \hh_{\RR}^{\ast}.
\end{equation}

\begin{definition}
 Let $\eta \in \hh_{\RR}^{\ast}$. The difference of the family $\Updelta_{\gg}(\eta)$ and the Laplace operator $\Updelta_{\gg}$ is called \emph{energy operator}, that is, 
 \[
 T(\eta)\coloneqq \Updelta_{\gg}(\eta)-\Updelta_{\gg}=2\gamma^{\mathcal{W}}(h_{\eta})+B(\eta,\eta).
 \]
\end{definition}

\begin{remark}
The operator $T(\eta)$ is called \emph{energy operator}, since it generates the Lie derivative in $\mathcal{W}(\gg)$ by Lemma~\ref{lemm::commutation_relations_quantum_Weil_algebra}. Note that $T(0)=0$.
\end{remark}

Let $\hh_{\RR}^{\mathrm{iso}}\subset\hh_{\RR}$ be a maximal isotropic subspace. Using the form $B$, we identify $\hh_{\RR}^{\mathrm{iso}}$ with its dual $(\hh_{\RR}^{\mathrm{iso}})^{\ast}$. The latter is spanned by pairwise distinct, mutually orthogonal odd isotropic roots of $\gg$. In particular, by definition, the dimension of $\hh_{\RR}^{\mathrm{iso}}$ coincides with the defect of $\gg$.

Since $B$ is invariant, it is preserved under the adjoint action of $G_{\bar{0}}$. In particular, for $\lambda\in(\hh_{\RR}^{\mathrm{iso}})^{\ast}$ and $g\in G_{\bar{0}}$, the element $\Ad^{\ast}_{g}(h)\in\gg^{\ast}$ is isotropic.

The operators $T(\eta)$ are regarded as a \emph{family of energy operators} parametrized by $(\hh_{\RR}^{\mathrm{iso}})^{\ast}$, defined by
\begin{equation}
(\hh_{\RR}^{\mathrm{iso}})^{\ast}\rightarrow\mathcal{W}(\gg),\qquad
\eta\mapsto T(\eta)\coloneqq 2\gamma^{\mathcal{W}}(h_{\eta}),
\end{equation}
where we note that $B(\eta,\eta)=0$ since $\eta \in (\hh_{\RR}^{\text{iso}})^{\ast}$.
If we fix a Borel subalgebra $\bb$ and a finite-dimensional simple supermodule $L(\Lambda)$ with $\Lambda \in P^{++}$, we consider
\begin{equation}
(\hh_{\RR}^{\mathrm{iso}})^{\ast}\longrightarrow \End_{\CC}(L(\Lambda)\otimes \overline{M}(\nn^{-})),
\qquad \eta \longmapsto T(\eta).
\end{equation}

The set of weights of $L(\Lambda)$ is $\mathcal{P}_{L(\Lambda)}=\{\Lambda - \ZZ_{+}[A] : A \subset \Delta^{+}\}$ and the set of weights of $\overline{M}(\pp)$ is 
$\mathcal{P}_{\overline{M}(\pp)}=\{\rho-\ZZ_{+}[B] : B \subset \Delta^{+}\}$. Both are subsets of $\hh_{\RR}^{\ast}$. A general weight of $L(\Lambda)\otimes \overline{M}(\nn^{-})$ takes the form $\Lambda-\mu_{1}+\rho-\mu_{2}$ such that $L(\Lambda)^{\Lambda-\mu_{1}}$ and $\overline{M}(\nn^{-})^{\rho-\mu_{2}}$ are non-trivial, and $\mu_{1},\mu_{2}$ are positive sums of positive roots. As in the proof of Lemma~\ref{lemm::action_Laplace_weight_space}, one obtains the following.

\begin{lemma}
For any weight $\mu \in \hh_{\RR}^{\ast}$, the operator $T(\eta)$ acts on the weight space
$(L(\Lambda)\otimes \overline{M}(\nn^{-}))^{\mu}$ by the scalar $2B(\mu,\eta)$. In particular,
\[
\ker T(\eta)=\ker T(\eta)^{2}.
\]
\end{lemma}

Any element $\eta\in(\hh_{\RR}^{\text{iso}})^{\ast}$ can be written as
\begin{equation}
\eta=a_{1}\alpha_{1}+\cdots+a_{r}\alpha_{r},\qquad a_{i}\in\RR,
\end{equation}
where $\alpha_{i}\in\Delta_{\bar{1}}$ are pairwise orthogonal odd isotropic roots. The integer
$0\le r\le\operatorname{def}(\gg)$ is called the \emph{rank} of $\eta$ and is denoted by $\rk(\eta)$.
For fixed $\mu\in\hh_{\RR}^{\ast}$, define the real vector subspace
\begin{equation}
X_{\mu}\coloneqq \{\eta\in(\hh_{\RR}^{\text{iso}})^{\ast}\mid B(\mu,\eta)=0\},
\end{equation}
called the \emph{isotropic annihilator} of $\mu$.

\begin{corollary} \label{cor::kernel_T}
Let $\mu=\Lambda-\mu_{1}-\mu_{2}+\rho$ be a weight of $L(\Lambda)\otimes \overline{M}(\nn^{-})$ and
$\eta\in(\hh_{\RR}^{\text{iso}})^{\ast}$. If $(L(\Lambda)\otimes \overline{M}(\nn^{-}))^{\mu}\subset\ker
T(\eta)$, then $\Lambda-\mu_{1}-\mu_{2}$ is atypical of degree at least $\rk(\eta)$. Conversely,
if $\Lambda-\mu_{1}-\mu_{2}$ is atypical of degree $k$, then
$(L(\Lambda)\otimes \overline{M}(\nn^{-}))^{\mu}\subset\ker T(\eta)$ for all $\eta\in X_{\mu}$. In
particular,
\[
\ker T(\eta)\big|_{(L(\Lambda)\otimes \overline{M}(\nn^{-}))^{\mu}}\neq\{0\}
\quad\Longleftrightarrow\quad
\eta\in X_{\mu}.
\]
\end{corollary}

For a finite-dimensional $\gg$-supermodule $L(\Lambda)$, Corollary~\ref{cor::kernel_T} shows that the kernel of $T$ on $L(\Lambda)\otimes \overline{M}(\nn^{-})$ is too large to yield a useful localization. It is therefore natural to replace $L(\Lambda)\otimes \overline{M}(\nn^{-})$ by a smaller subspace. By construction (\emph{cf.}~\eqref{eq::definition_M_p}) the oscillator supermodule $\overline{M}(\nn^{-})$ is a highest weight module with highest weight vector $1_{S}\otimes 1_{\overline{M}(\nn_{\bar{1}}^{-})}$. Set $E_{\Lambda}:=L(\Lambda)\otimes S$ and write $1:=1_{\overline{M}(\nn_{\bar{1}}^{-})}$. We then consider the $\even$-submodule $\langle E_{\Lambda}\otimes 1\rangle_{\even}$, where parity is disregarded and $\even$ acts diagonally. This submodule already suffices for the application of the Dirac operator to the Dirac inequality; see \cite{SchmidtDirac}. 

\begin{lemma}[\cite{SchmidtDirac}] \label{lemm::decomposition_E_1}
    If $L(\Lambda)\big\vert_{\even} \cong \bigoplus_{\mu}L_{0}(\mu)^{n(\mu)}$, then 
   \[
   \langle E_{\Lambda}\otimes 1\rangle_{\even} \cong \bigoplus_{\mu} L_{0}(\mu+\rho)^{n(\mu)},
   \]
   and any $\even$-highest weight vector is of the form $v_{\mu}\otimes 1_{S}\otimes 1_{\overline{M}(\nn_{\bar{1}}^{-})}$ for some $v_{\mu}\in L(\Lambda)^{\mu}$.
\end{lemma}

By Lemma~\ref{lemm::decomposition_E_1}, both $\widetilde{\Updelta}_{\even}(\xi)$ and $T(\eta)$ act on $\langle E_{\Lambda}\otimes 1\rangle_{\even}$. For $T(\eta)$, this is immediate from the definition and the $\hh$-semisimplicity of the module. We combine these operators into a single map
\begin{equation}
(T,\widetilde{\Updelta}_{\even})\colon(\hh_{\RR}^{\text{iso}})^{\ast}\times\hh_{\RR}^{\ast}\to \End_{\CC}\bigl(\langle E_{\Lambda}\otimes 1\rangle_{\even},(\langle E_{\Lambda}\otimes 1\rangle_{\even})^{\oplus 2}\bigr),
\end{equation}
called the \emph{detecting family}.

\begin{theorem}\label{thm::main_resul_family_tuple}
Let $L(\Lambda)$ be a finite-dimensional $\gg$-supermodule with $L(\Lambda)\vert_{\even}\cong \bigoplus_{\mu}L_{0}(\mu)^{n(\mu)}$. Then
$\ker((T(\eta),\widetilde{\Updelta}_{\even}(\xi))=\{0\}$ unless both of the following conditions hold:
\begin{enumerate}
\item[a)] $\xi=-\mu-\rho_{\bar{0}}$, where $\mu$ is the highest weight of a $\even$-constituent of $L(\Lambda)$,
\item[b)] $\eta\in X_{\mu+\rho}$.
\end{enumerate}
In particular,
\[
\ker (T(\eta), \widetilde{\Updelta}_{\even}(\xi)) \neq \{0\} \Leftrightarrow (\eta,\xi) \in \bigcup_{\mu\,:\,n(\mu)\neq 0}X_{\mu+\rho} \times \Ad_{G_{\bar{0}}}^{\ast}(-\mu-\rho_{\bar{0}}).
\]
\end{theorem}

\begin{proof}
By construction,
\[
\ker(T(\eta),\widetilde{\Updelta}_{\even}(\xi))=\ker T(\eta)\cap\ker \widetilde{\Updelta}_{\even}(\xi).
\]
It therefore suffices to determine $\ker \widetilde{\Updelta}_{\even}(\xi)$ and the action of $T(\eta)$ on this space. By \eqref{eq::action_Laplace_tilde}, Lemma~\ref{lemm::decomposition_E_1}, and Theorem~\ref{thm::main_family_even}, the kernel of $\widetilde{\Updelta}_{\even}(\xi)$ is non-zero precisely for
\[
\xi\in \bigcup_{\mu\,:\,n(\mu)\neq 0}\Ad_{G_{\bar{0}}}^{\ast}(-\mu-\rho_{\bar{0}}),
\]
and in this case
\[
\ker \widetilde{\Updelta}_{\even}(\xi)\cong \bigoplus_{\mu\,:\,\xi\in \Ad_{G_{\bar{0}}}^{\ast}(-\mu-\rho_{\bar{0}})}
(L_{0}(\mu)^{\mu}\otimes S^{\rho_{\bar{0}}}\otimes \overline{M}(\nn_{\bar{1}}^{-})^{-\rho_{\bar{1}}})^{ \oplus n(\mu)}.
\]
Here $\overline{M}(\nn_{\bar{1}}^{-})^{-\rho_{\bar{1}}}=\CC (1_{\overline{M}(\nn_{\bar{1}}^{-})})$, so the summand indexed by $\mu$ has dimension $n(\mu)$. Now $T(\eta)$ preserves each such summand, and it acts on
$
(L(\Lambda)^{\mu}\otimes S^{\rho_{\bar{0}}})\otimes 1_{\overline{M}(\nn_{\bar{1}}^{-})}
$ by the scalar $2B(\mu+\rho,\eta)$. Consequently, the joint kernel
$
\ker(T(\eta),\widetilde{\Updelta}_{\even}(\xi))
$
is obtained by restricting the above direct sum to those $\mu$ for which both conditions hold:
\[
\xi\in \Ad_{G_{\bar{0}}}^{\ast}(-\mu-\rho_{\bar{0}})
\qquad\text{and}\qquad
B(\mu+\rho,\eta)=0.
\]
The assertion follows from Corollary~\ref{cor::kernel_T}.
\end{proof}

\section{Nilpotent Perturbations} \label{sec::nilpotent_perturbations}
In this section, let $\gg$ be a basic classical Lie superalgebra with $\odd\neq\{0\}$. Two standard tools for studying $\gg$-supermodules are:
\begin{enumerate}
 \item \emph{Dirac cohomology}, which detects the infinitesimal character.
 \item \emph{Duflo--Serganova cohomology} (DS cohomology), a symmetric monoidal functor preserving the superdimension of finite-dimensional supermodules.
\end{enumerate}
We construct a family of relative cubic Dirac operators $\Dirac_{\gg,\ll}^{x}$, parametrized by the self-commuting variety, whose cohomology interpolates between Dirac cohomology and Duflo--Serganova cohomology. We start by recalling DS cohomology.

\subsection{Duflo--Serganova Cohomology} \label{subsec::Duflo_Serganova_Cohomology} The \emph{Duflo--Serganova functor} (DS functor), introduced by Duflo and Serganova~\cite{duflo2005associated}, is a symmetric monoidal tensor functor associated with an odd square-zero element. 

In what follows, we briefly recall the DS functor, mainly following
\cite{duflo2005associated,gorelik2022duflo}; see also \cite{serganova2011superdimension}.

The DS functor depends on a choice of element in the \emph{self-commuting variety}
\begin{equation}
 \YY\coloneqq \{x\in\odd:[x,x]=0\}.
\end{equation}
Let $G_{\bar{0}}$ be the connected simply connected Lie group with Lie algebra $\even$. It acts on $\YY$ by the adjoint action, and $\YY$ is a $G_{\bar 0}$-stable Zariski-closed cone in $\odd$. Its $G_{\bar 0}$-orbits are in bijection with $W$-orbits of subsets of mutually orthogonal, linearly independent, odd isotropic roots; in particular, $\YY$ has finitely many $G_{\bar 0}$-orbits. Concretely, fix $x\in\YY$. Then there exist $g\in G_{\bar 0}$ and mutually orthogonal, linearly independent isotropic roots
$\alpha_1,\dots,\alpha_k$ such that
\begin{equation}\Ad_g(x)=x_1+\cdots+x_k,\qquad x_i\in\gg^{\alpha_i}.
\end{equation}
The integer $\rk(x)\coloneqq k$ is independent of the choices and is called the \emph{rank} of $x$; equivalently, it is the rank of $x$ in the standard representation.

To each $x\in\YY$ we associate a Lie superalgebra
\begin{equation}
 \gg_x\coloneqq \ker(\ad_x)/\Im(\ad_x),
\end{equation}
which is well-defined since $\Im(\ad_x)=[x,\gg]$ is an ideal in $\ker(\ad_x)$. Moreover, $\gg_x$ has Cartan subalgebra
\begin{equation}\hh_x\coloneqq \Bigl(\bigcap_{i=1}^k \ker(\alpha_i)\Bigr)\Big/\bigl(\hh^{\alpha_1}\oplus\cdots\oplus\hh^{\alpha_k}\bigr),
 \qquad \hh^{\alpha}\coloneqq [\gg^{\alpha},\gg^{-\alpha}],
\end{equation}
and the corresponding root system is
\begin{equation}
 \Delta_x=\{\alpha\in\Delta : B(\alpha,\alpha_i)=0\ \text{for }i=1,\dots,k,\ \alpha\neq \alpha_i\}.
\end{equation}

We now define the \emph{DS functor}. Given a highest weight $\gg$-supermodule $M$, each $x\in\YY$ acts on $M$ by an endomorphism $x_M\in\End_{\CC}(M)$ with $x_M^{2}=0$. Define
\begin{equation}
 M_x\coloneqq \ker(x_M)/\Im(x_M).
\end{equation}
Then $M_x$ is naturally a $\gg_x$-supermodule, since $\ker(x_M)$ is $\ker(\ad_x)$-stable and
$[x,\gg]\cdot \ker(x_M)\subseteq \Im(x_M)$. Thus $M\mapsto M_x$ defines a functor from
$\gg$-supermodules to $\gg_x$-supermodules, denoted
\begin{equation}
 \DS_x(M)\coloneqq M_x.
\end{equation}
This is the \emph{Duflo--Serganova functor} (DS functor). For finite-dimensional $\gg$-supermodules, it has the following basic properties.

\begin{lemma}[{\cite[Section~2]{gorelik2022duflo}}]\label{DSTensorFunctor} Let $\gg$\textbf{-smod} denote the category of finite-dimensional $\gg$-supermodules. For any $x\in\YY$, the functor $\DS_x:\gsmod\to \gg_x\textbf{-smod}$ is additive and symmetric monoidal. Moreover, $\DS_x$ preserves superdimension.
\end{lemma}

For the remainder of this section, fix a highest weight $\gg$-supermodule $M$. We recall a construction from
\cite[Section~6]{duflo2005associated} (see also \cite{serganova2011superdimension}).
Let $\mathfrak{U}(\gg)^{\ad_x}$ be the subsuperalgebra of $\ad_x$-invariants in $\mathfrak{U}(\gg)$, and set
$
 I_x\coloneqq [x,\mathfrak{U}(\gg)]\subset \mathfrak{U}(\gg).
$
Consider $\phi=\pi\circ\iota$ in the sequence
\begin{equation}
 \UE(\gg_x)\xrightarrow{\ \iota\ } \mathfrak{U}(\gg)^{\ad_x}\xrightarrow{\ \pi\ }
 \mathfrak{U}(\gg)^{\ad_x}/\bigl(I_x\cap \mathfrak{U}(\gg)^{\ad_x}\bigr),
\end{equation}
where $\iota$ is the inclusion and $\pi$ the projection. Both are morphisms of $\gg_x$-supermodules (adjoint action), and $\phi$ is an isomorphism of super vector spaces~\cite[Lemma~6.6]{duflo2005associated}. Define
\begin{equation}\label{eq::definition_upeta}
 \upeta\coloneqq \phi^{-1}\circ\pi:\ \mathfrak{U}(\gg)^{\ad_x}\to \UE(\gg_x).
\end{equation}
Then $u\cdot m=\upeta(u)\cdot m$ for $u\in \mathfrak{U}(\gg)^{\ad_x}$ and $m\in M_x$, since $\ker(x_M)$ is
$\mathfrak{U}(\gg)^{\ad_x}$-stable and $I_x\ker(x_M)\subseteq \Im(x_M)$.

Let $\mathfrak{Z}(\gg_x)$ be the center of $\UE(\gg_x)$. Since $\mathfrak Z(\gg)\subset \mathfrak{U}(\gg)^{\ad_x}$ and $\upeta$
is a morphism of $\gg_x$-supermodules, we have $\upeta\bigl(\mathfrak Z(\gg)\bigr)\subset \mathfrak Z(\gg_x),
$
and the dual map
\begin{equation}
 \upeta^{*}:\ \Hom(\mathfrak Z(\gg_x),\CC)\to \Hom(\mathfrak Z(\gg),\CC)
\end{equation}
is injective~\cite[Theorem~6.11]{duflo2005associated}.

Let $\chi$ be the infinitesimal character of $M$. For $z\in\mathfrak Z(\gg)$ and $m\in\ker(x_M)$,
\begin{equation}
 \chi(z)m=zm\equiv \upeta(z)m \pmod{xM},
\end{equation}
so if $M_x$ contains a submodule with infinitesimal character $\xi$, then $\upeta^{*}(\xi)=\chi$. Moreover,
\begin{equation}
 \at\bigl(\upeta^{*}(\xi)\bigr)=\at(\xi)+\rk(x),
\end{equation}
and hence $\at(\xi)=\at(\chi)-\rk(x)$. Thus we obtain:

\begin{theorem}[{\cite[Theorem~2.1]{serganova2011superdimension}}]\label{thm::DScharacter}
Let $M$ be a simple highest weight $\gg$-supermodule with infinitesimal character $\chi$.
Then $M_x$ decomposes as a direct sum of $\gg_x$-supermodules whose generalized infinitesimal characters lie in
$(\upeta^{*})^{-1}(\chi)$, and each $\xi\in(\upeta^{*})^{-1}(\chi)$ satisfies $\at(\xi)=\at(\chi)-\rk(x)$.
\end{theorem}

\begin{corollary}\label{DSNM}
Let $x\in\YY\setminus\{0\}$ have rank $l$, and let $M$ be a highest weight $\gg$-supermodule of atypicality $k$.
If $k<l$, then $\DS_x(M)=0$.
\end{corollary}

\subsection{A Family of Nilpotent Perturbations}
Fix a quadratic Lie subsuperalgebra $\ll\subset\gg$ such that $B_{\ll}\coloneqq B|_{\ll}$ is non-degenerate and
\begin{equation}
 \gg=\ll\oplus\pp,\qquad \pp=\ll^{\perp},
\end{equation}
where the decomposition is orthogonal with respect to $B$. In what follows, we embed $\WW(\ll)$ into $\Weil$ using the embedding \eqref{eq::embedding_j}. Assume $\hh\subset\ll$ and $\ll_{\bar 1}\neq\{0\}$. Let
\begin{equation}
 \YY_{\ll}\coloneqq \{x\in\ll_{\bar 1}:\ [x,x]=2x^{2}=0\}
\end{equation}
be the self-commuting variety of $\ll$. If $L_{\bar 0}$ is a connected simply connected Lie group with Lie algebra $\ll_{\bar 0}$, then $\YY_{\ll}$
is an $L_{\bar 0}$-stable Zariski-closed cone in $\ll_{\bar 1}$. We will define a family of cubic Dirac operators parametrized by
$x\in\YY_{\ll}$.

Recall that the quantum Weil algebra $\WW(\ll)$ is generated by $1\otimes x$ and $\gamma^{\WW}(x)=x\otimes 1+1\otimes \nu_{*}(x)$,
with bidegrees $(\bar 1,p(x))$ and $(\bar 0,p(x))$ (hence total degrees $p(x)+\bar 1$ and $p(x)$, respectively). We write $\deg$
for the total degree on $\WW(\ll)$. Unlike the reductive (purely even) case, this grading allows square-zero perturbations of the
relative cubic Dirac operator while keeping a nice square.

\begin{definition}
For $x\in\YY_{\ll}$ set $$\Dirac^{x}_{\gg,\ll}\coloneqq \Dirac_{\gg,\ll}+j(\gamma^{\WW}(x))\in\Weil.$$
\end{definition}

\begin{remark}\label{rmk::square_gamma_W_x}
Since $j:\WW(\ll)\hookrightarrow \Weil$ is a Lie superalgebra homomorphism, we have $j(\gamma^{\WW}(x))^{2}=0$ for all
$x\in\YY_{\ll}$. Indeed, $x$ is odd, hence $\gamma^{\WW}(x)$ is odd for the total degree, and therefore
\[
j(\gamma^{\WW}(x))^{2}=\tfrac12\,[j(\gamma^{\WW}(x)),j(\gamma^{\WW}(x))]_{\WW}
=\tfrac12\,j([\gamma^{\WW}(x),\gamma^{\WW}(x)]_{\WW})
=\tfrac12\,j(\gamma^{\WW}([x,x]_{\gg}))=0.
\]
\end{remark}

Let $C_{\ll}(x)\coloneqq \{y\in\ll:[x,y]=0\}$ be the \emph{commutant} of $x$ in $\ll$. The following lemma summarizes the main properties $\Dirac^{x}_{\gg,\ll}$.

\begin{lemma}\label{lemm::properties_D_x}
For $x\in\YY_{\ll}$ the following hold:
\begin{enumerate}
\item[a)] $\Dx$ is odd in $\Weil$ with respect to the total degree.
\item[b)] $\Dx$ is $C_{\ll}(x)$-invariant.
\item[c)] For $g\in L_{\bar 0}$, one has $\Ad_{g}(\Dx)=\Dirac^{\Ad_{g}(x)}_{\gg,\ll}$.
\item[d)] $(\Dx)^{2}=\Dirac_{\gg,\ll}^{2}$.
\end{enumerate}
\end{lemma}

\begin{proof}
a) Both $\Dirac_{\gg,\ll}$ and $j(\gamma^{\WW}(x))$ have total degree $\bar 1$, hence so does $\Dx$.
b) Let $y\in C_{\ll}(x)$. Since $\Dirac_{\gg,\ll}$ is $\ll$-invariant, it commutes with $y$ under the embedding $\ll\hookrightarrow\Weil$.
Moreover,
\[
[\gamma^{\WW}(x),\gamma^{\WW}(y)]_{\WW}=\gamma^{\WW}([x,y]_{\gg})=0,
\]
so $\Dx$ is $C_{\ll}(x)$-invariant.
c) This follows from $\Ad$-equivariance of $\Dirac_{\gg,\ll}$ and of $j\circ\gamma^{\WW}$.
d) Since $\Dx$ is odd,
\[
(\Dx)^2=\tfrac12[\Dx,\Dx]_{\WW}
=\tfrac12[\Dirac_{\gg,\ll},\Dirac_{\gg,\ll}]_{\WW}
+[\Dirac_{\gg,\ll},j(\gamma^{\WW}(x))]_{\WW}
+\tfrac12[j(\gamma^{\WW}(x)),j(\gamma^{\WW}(x))]_{\WW}.
\]
The middle term vanishes by $\ll$-invariance of $\Dirac_{\gg,\ll}$, and the last term is zero by
Remark~\ref{rmk::square_gamma_W_x}. Hence $(\Dx)^2=\Dirac_{\gg,\ll}^2$.
\end{proof}

Fix a highest weight simple $\gg$-supermodule $M$ such that we can consider the family
\begin{equation}
 \YY_{\ll} \to \End(M\otimes \overline{M}(\pp)), \qquad x \mapsto \Dx.
\end{equation}
Recall that we assigned to any morphism $T$ of a super vector space $V$ a cohomology $\H(T;V)$ (Section~\ref{subsec::Dirac_cohomology}), which yields the Dirac cohomology $\H_{\Dirac_{\gg,\ll}}(M)$ for $T=\Dirac_{\gg,\ll}$ and $V=M\otimes \overline{M}(\pp)$. Analogously, we define 
\begin{equation}
 \H_{\Dx}(M)\coloneqq \H(\Dx; M\otimes \overline{M}(\pp))=\ker \Dx / (\ker \Dx \cap \Im \Dx).
\end{equation}

Our goal is to describe this cohomology. First, by Lemma~\ref{lemm::properties_D_x}, we have the following lemma

\begin{lemma}
 For any $x \in \YY_{\ll}$ is $\H_{\Dx}(M)$ is a $C_{\ll}(x)$-supermodule.
\end{lemma}

Computing the cohomology on $M\otimes \overline{M}(\pp)$ reduces to the Laplacian
$\Updelta_{\gg,\ll}\coloneqq \Dirac_{\gg,\ll}^{2}$. We need:

\begin{lemma}\label{lemm::decomposition_under_Laplace}
The following hold:
\begin{enumerate}
\item[a)] $M\otimes \overline{M}(\pp)\cong \ker \Updelta_{\gg,\ll}\oplus \Im \Updelta_{\gg,\ll}$.
\item[b)] $\H_{\Updelta_{\gg,\ll}}(M)\coloneqq \H(\Updelta_{\gg,\ll};\,M\otimes \overline{M}(\pp))=\ker \Updelta_{\gg,\ll}$.
\end{enumerate}
\end{lemma}

\begin{proof}
Part~a) is proved in \cite[Proposition~4.1.2]{Schmidt_Noja}. Part~b) follows immediately from~a).
\end{proof}

\begin{lemma}\label{lemm::cohomology_restriction}
One has
\[
\H_{\Dx}(M)=\H(\Dx;\,\ker(\Dx)^{2})
=\H(\Dx;\,\ker \Dirac_{\gg,\ll}^{2})
=\H(\Dx;\,\H_{\Updelta_{\gg,\ll}}(M)).
\]
\end{lemma}

\begin{proof}
By definition, $\Dx=\Dirac_{\gg,\ll}+j(\gamma^{\WW}(x))$, and Lemma~\ref{lemm::properties_D_x} gives $(\Dx)^{2}=\Dirac_{\gg,\ll}^{2}=\Updelta_{\gg,\ll}$. Together with Lemma~\ref{lemm::decomposition_under_Laplace}, this yields
\[
M\otimes \overline{M}(\pp)\cong \ker \Updelta_{\gg,\ll}\oplus \Im \Updelta_{\gg,\ll}.
\]
Decompose $M\otimes \overline{M}(\pp)$ into eigenspaces of $\Updelta_{\gg,\ll}$. Since $\Updelta_{\gg,\ll}$ commutes with $\Dirac_{\gg,\ll}$ and with
$j(\gamma^{\WW}(x))$, it commutes with $\Dx$, hence $\ker\Dx$, $\Im \Dx$, and $\H(\Dx;\,M\otimes \overline{M}(\pp))$ split as direct sums over these eigenspaces.

Let $\lambda\neq 0$ and let $V_{\lambda}$ be the $\lambda$-eigenspace of $\Updelta_{\gg,\ll}$. It suffices to show that
$\ker(\Dx|_{V_{\lambda}})\subset \Im(\Dx|_{V_{\lambda}})$. Take $v\in V_{\lambda}$ with $\Dx v=0$. Then
$\Dirac_{\gg,\ll}v=-j(\gamma^{\WW}(x))v$, and hence
\[
\begin{split}
\Dx\,\Dirac_{\gg,\ll}v
&=(\Dirac_{\gg,\ll}+j(\gamma^{\WW}(x)))\Dirac_{\gg,\ll}v\\
&=\Dirac_{\gg,\ll}^{2}v+j(\gamma^{\WW}(x))\Dirac_{\gg,\ll}v\\
&=\Updelta_{\gg,\ll}v-j(\gamma^{\WW}(x))^{2}v\\
&=2\Updelta_{\gg,\ll}v=2\lambda v,
\end{split}
\]
where we used $j(\gamma^{\WW}(x))^{2}=0$ and $\Dirac_{\gg,\ll}j(\gamma^{\WW}(x))=-j(\gamma^{\WW}(x))\Dirac_{\gg,\ll}$ (both are odd).
Thus $v=\frac{1}{2\lambda}\Dx(\Dirac_{\gg,\ll}v)\in\Im(\Dx|_{V_{\lambda}})$, proving the claim. Therefore only the
$\lambda=0$ eigenspace contributes, \emph{i.e.}, $\ker(\Dx)^{2}=\ker\Updelta_{\gg,\ll}=\H_{\Updelta_{\gg,\ll}}(M)$, and the stated equalities follow.
\end{proof}

For a unitarizable highest weight supermodule $M$ (Section~\ref{subsec::unitarizable_supermodules}), the previous lemma yields an explicit description of $\H_{\Dx}(M)$
in terms of Dirac cohomology (Section~\ref{subsec::Dirac_cohomology}) and Duflo--Serganova cohomology
(Section~\ref{subsec::Duflo_Serganova_Cohomology}). Recall $\upeta_{x}$ from \eqref{eq::definition_upeta} and its dual injective map
$\upeta^{*}_{x}:\Hom(\mathfrak Z(\ll_x),\CC)\to \Hom(\mathfrak Z(\ll),\CC)$ for $x\in\YY_{\ll}$. Also let
$\eta_{\ll}:\ZG\to\mathfrak Z(\ll)$ be the homomorphism from Theorem~\ref{thm::Casselman_Osborne}. We then have:

\begin{theorem} \label{thm::Dirac_and_DS}
 Assume $M$ is a unitarizable highest weight supermodule with highest weight $\Lambda$. Then for any $x \in \YY_{\ll}$ one has
 \[
 \H_{\Dx}(M)=\DS_{x}(\H_{\Dirac_{\gg,\ll}}(M)).
 \]
 In particular, $\H_{\Dx}(M)$ is a $\ll_{x}$-supermodule and decomposes as a direct sum of $\ll_{x}$-supermodules. If $V$ is a $\ll_{x}$-supermodule with infinitesimal character $\chi_{\nu}^{\ll_{x}}$, then 
 \[
 \chi_{\nu}^{\ll_{x}} \in (\upeta_{x}^{\ast})^{-1}(\chi_{\Lambda} \circ \eta_{\ll}).
 \]
\end{theorem}

\begin{proof}
 If $M$ is unitarizable, then $\ker \Dirac_{\gg,\ll}^{2}=\ker \Dirac_{\gg,\ll}$ by Lemma~\ref{thm::Dirac_unitarizable} and
$\H_{\Dirac_{\gg,\ll}}(M)=\ker \Dirac_{\gg,\ll}$ by unitarity. Hence, by
Lemma~\ref{lemm::cohomology_restriction},
\[
\begin{split}
\H_{\Dx}(M)
&=\H(\Dx;\,M\otimes \overline{M}(\pp))
=\H(\Dx;\,\ker \Dirac_{\gg,\ll}^{2})
=\H(\Dirac_{\gg,\ll}+j(\gamma^{\WW}(x));\,\ker \Dirac_{\gg,\ll})\\
&=\H(j(\gamma^{\WW}(x));\,\ker \Dirac_{\gg,\ll}).
\end{split}
\]
The action of $\ll$ on $M\otimes \overline{M}(\pp)$ is induced by its embedding into $\WW(\ll)$ via $\gamma^{\WW}$. Hence, by the embedding \eqref{eq::embedding_j}, one has
\[
\H_{\Dx}(M)=\DS_{x}(\ker \Dirac_{\gg,\ll})=\DS_{x}\bigl(\H_{\Dirac_{\gg,\ll}}(M)\bigr).
\]

This proves the first part of the theorem. The remaining claims follow by combining Theorem~\ref{thm::Casselman_Osborne} with Theorem~\ref{thm::DScharacter}.
\end{proof}

\begin{remark}
The proof of Theorem~\ref{thm::Dirac_and_DS} relies on the identity
\[
\ker \Dirac_{\gg,\ll}=\ker \Dirac_{\gg,\ll}^{2},
\]
which holds for unitarizable $M$. In general, this is equivalent to each of the following conditions:
\begin{enumerate}
 \item[(a)] $\ker \Dirac_{\gg,\ll}=\ker \Dirac_{\gg,\ll}^{2}$;
 \item[(b)] $\ker \Dirac_{\gg,\ll}\cap \Im \Dirac_{\gg,\ll}=\{0\}$;
 \item[(c)] $M\otimes \overline{M}(\pp)=\ker \Dirac_{\gg,\ll}\oplus \Im \Dirac_{\gg,\ll}$;
 \item[(d)] $\H_{\Dirac_{\gg,\ll}}(M)=\ker \Dirac_{\gg,\ll}$.
\end{enumerate}
Thus, the assumption is equivalent both to the direct sum decomposition in \textup{c)} and to the identification of Dirac cohomology with the full kernel in \textup{d)}.
\end{remark}

The following corollary is immediate.

\begin{corollary}
Let $M$ be a unitarizable highest weight supermodule with highest weight $\Lambda$. Then:
\begin{enumerate}
\item[a)] $\H_{\Dx}(M)=0$ if $\rk(x) > \at(\Lambda)$.
\item[b)] For $g\in L_{\bar 0}$, there is a canonical isomorphism
\[
\H_{\Dx}(M)\cong \H_{\Dirac^{\!\Ad_{g}(x)}_{\gg,\ll}}(M).
\]
\end{enumerate}
\end{corollary}

If $M$ is not unitarizable, we can only formulate a vanishing conjecture. By Lemma~\ref{lemm::cohomology_restriction},
\[
\H_{\Dx}(M)=\H(\Dx,\ker \Dirac_{\gg,\ll}^{2}).
\]
Moreover, on $V\coloneqq \ker \Dirac_{\gg,\ll}^{2}$, both $\Dirac_{\gg,\ll}$ and $j(\gamma^{\WW}(x))$ are anti-commuting differentials, since
$
[\Dirac_{\gg,\ll},j(\gamma^{\WW}(x))]_{\WW}=0.
$
This suggests the following vanishing statement from the perspective of spectral sequences.

\begin{conjecture}
Let $x\in\YY_{\ll}$. If $\DS_{x}\bigl(\H_{\Dirac_{\gg,\ll}}(M)\bigr)=0$, then
$
\H_{\Dirac^{x}_{\gg,\ll}}(M)=0.
$
\end{conjecture}

\section{Bismut--Quillen's Superconnection} \label{sec::Bismut_Quillen}

We perturb the relative cubic Dirac operator by the covariant Weil differential associated to a finite-dimensional (super)module $M$, obtained from the universal connection $1$-form of the colour Weil algebra. The resulting operator defines an element of $\End_{\CC}(\Weil\otimes E)$, where $E\coloneqq M\otimes S$, and admits a well-defined exponential after completing $\Weil$. Taking the relative supertrace yields a cohomology class in the completed colour Weil algebra, thereby assigning to each finite-dimensional module a canonical class. We first consider reductive complex Lie algebras, and then adapt the construction to finite-dimensional supermodules over basic classical Lie superalgebras. Moreover, Section~\ref{subsec::an_example} contains an explicit example of the construction in the case of the complex simple Lie algebra $\mathfrak{sl}(2,\CC)$.

\subsection{Semisimple Complex Lie Algebra}

Let $\gg$ be a finite-dimensional semisimple complex Lie algebra with Killing form $B$. In the following, we use $B$ to identify $\gg$ and $\gg^{\ast}$. Fix a Cartan subalgebra $\mathfrak{h}\subset\gg$ with root system $\Delta$, and choose a positive system $\Delta^{+}$.
This yields a triangular decomposition
\begin{equation}
\gg=\mathfrak{n}^{-}\oplus\mathfrak{h}\oplus\mathfrak{n}^{+},
\qquad
\mathfrak{n}^{\pm}\coloneqq \bigoplus_{\alpha\in\Delta^{+}}\gg^{\pm\alpha}.
\end{equation}
Fix a basis $\{e_{a}\}$ of $\gg$ and its $B$-dual basis $\{e^{a}\}$. We consider the quantum Weil algebra $\mathcal{W}(\gg)$, and fix the tensor product generators (\emph{cf.}~Remark~\ref{rmk::TP_generators_QW})
\begin{equation}
 u^{a}\coloneqq e^{a}\otimes 1\in \mathcal{W}^{2}(\gg),\qquad
 \theta^{a}\coloneqq 1\otimes e^{a}\in \mathcal{W}^{1}(\gg),
\end{equation}
where the $u^{a}$ are even and the $\theta^{a}$ are odd. We define the $\gg$\emph{-valued quantum Weil algebra} to be $\mathcal{W}(\gg)\otimes\gg$, and we equip it with the Lie bracket
\begin{equation}
 [A,B]\coloneqq \sum_{i,j}(a_{i}b_{j})\otimes [x_{i},y_{j}]_{\gg},
\end{equation}
where $A=\sum_{i}a_{i}\otimes x_{i}, B=\sum_{j}b_{j}\otimes y_{j}\in \mathcal{W}(\gg)\otimes\gg$.
We regard $\gg$ as purely even; this induces a natural $\ZZ_{2}$-grading on
$\mathcal{W}(\gg)\otimes\gg$ from $\mathcal{W}(\gg)$.

Henceforth, let $(\rho_{M},M)$ be a fixed finite-dimensional simple $\gg$-module; when $M$ is clear from context, we suppress $\rho_{M}$.
Recall the spin module $S$ from Section~\ref{subsec::oscillator_supermodule}. Note that $S\coloneqq \bigwedge \mathfrak{n}^{-}$ carries a natural $\ZZ_{2}$-grading
$S=S_{\bar 0}\oplus S_{\bar 1}$, making it a $\gg$-supermodule, \emph{i.e.}, the $\gg$-action preserves the grading.
Set
\begin{equation}
 E\coloneqq M\otimes S.
\end{equation}
Equipped with the induced grading from $S$, the $\gg$-module $(\rho_{E},E)$ becomes a $\gg$-supermodule. 

Recall that $\mathcal{W}(\gg)$ is $\ZZ_{2}$-graded. Using the $\ZZ_{2}$-graded tensor product $\otimes$, the space $\mathcal{W}(\gg)\otimes E$ is a super vector space. Hence $\End_{\CC}(\mathcal{W}(\gg)\otimes E)$ inherits a $\ZZ_{2}$-grading: its even part consists of parity-preserving endomorphisms and its odd part of parity-reversing endomorphisms. With the canonical supercommutator $[\cdot,\cdot]_{\End}$, it is a Lie superalgebra.

We have moreover a natural action of $\gg$ on $\mathcal{W}(\gg)\otimes E$:
\begin{equation}
 \upalpha(x)(v\otimes w)\coloneqq L_{x}v \otimes w + v\otimes \rho_{E}(x)w, \qquad x \in \gg, \ v\otimes w \in \mathcal{W}(\gg)\otimes E,
\end{equation}
where $L_{x}$ denotes the Lie derivative \eqref{eq::Lie_derivative_Weil} on $\Weil$.

\subsubsection{The Weil Covariant Differential} The Bismut--Quillen superconnection may be viewed as a perturbation of the cubic Dirac operator by the \emph{Weil covariant differential} $\nabla^{M}$ associated with $M$, which we now define. 

The \emph{universal connection $1$-form} is the odd element
\begin{equation}
\Theta\coloneqq \sum_{a}\theta^{a}\otimes e_{a}\in \mathcal{W}(\gg)\otimes\gg.
\end{equation}
It is independent of the choice of basis. Intrinsically, $\Theta$ is the canonical element
corresponding to $\id_{\mathfrak g}$ under the identification
$\mathfrak g^{*}\otimes\mathfrak g\simeq \End(\mathfrak g)$, viewed in
$\mathcal W^{1}(\mathfrak g)\otimes\mathfrak g$. Moreover, it is $\gg$-invariant for the diagonal action $L_{x}$ on the first factor and $\ad_{x}$ on the second factor. 

\begin{lemma} \label{lemm::invariance_theta}
 One has for all $x \in \gg$
 \[
 (L_{x}\otimes 1+1\otimes \ad_{x})\Theta=0, \qquad (\iota_{x}\otimes 1)\Theta=x. 
 \]
\end{lemma}

\begin{proof}
 We have by definition
 \[
 \begin{aligned}
 (L_{x}\otimes 1)\Theta &=\sum_{a}L_{x}\theta^{a}\otimes e_{a}=\sum_{a} (1\otimes [x,e^{a}]) \otimes e_{a}=\sum_{a}(1 \otimes \sum_{b} B([x,e^{a}],e_{b})e^{b}) \otimes e_{a} \\ &=
 \sum_{b} \theta^{b}\otimes \sum_{a}B([x,e^{a}],e_{b})e_{a}=-\sum_{b}\theta^{b}\otimes \sum_{a}B([x,e_{b}],e^{a})e_{a}= -(1\otimes \ad_{x})\Theta,
 \end{aligned} 
 \]
 where we used invariance of $B$, that is, $B([x,e^{a}],e_{b})=B(x,[e^{a},e_{b}])=-B(x,[e_{b},e^{a}])=-B([x,e_{b}],e^{a})$. The second property $(\iota_{x}\otimes 1)\Theta=x$ follows by definition.
\end{proof}

Furthermore, $\Theta$ satisfies a Maurer--Cartan curvature type equation. The following lemma is a straightforward computation.

\begin{lemma} \label{lemm::MCE}
 One has 
 \[
 (d^{\mathcal{W}}\otimes 1)\Theta + \tfrac{1}{2}[\Theta,\Theta]=\sum_{a}u^{a}\otimes e_{a}.\]
\end{lemma}

The universal connection $1$-form $\Theta$ induces a $\mathbb{C}$-linear endomorphism
\begin{equation}
 \Theta^{M}\coloneqq \sum_{a}\theta^{a} \otimes \rho_{E}(e_{a}) \in \End_{\mathbb{C}}(\mathcal{W}(\gg) \otimes E),
\end{equation}
acting on $\mathcal{W}(\gg)\otimes E$ componentwise. 

\begin{lemma} \label{lemm::properties_Theta_M}
 The following hold:
 \begin{enumerate}
 \item[a)] $\Theta^{M} \in \End_{\CC}(\mathcal{W}(\gg) \otimes E)_{\bar{1}}$.
 \item[b)] $\Theta^{M}$ is $\gg$-equivariant, that is, $
 [\upalpha(x),\Theta^{M}]_{\End}=0
 $ for all $x \in \gg$.
 \item[c)] $(d^{\mathcal{W}}\otimes 1)\Theta^{M} + \tfrac{1}{2}[\Theta^{M},\Theta^{M}]_{\End}=F^{M}$ with $F^{M}\coloneqq \sum_{a} u^{a} \otimes \rho_{E}(e_{a})$.
 \item[d)] $(\iota_{x}\otimes 1)\Theta^{M}=\rho_{E}(x)$ for all $x \in \gg$.
 \end{enumerate}
\end{lemma}

\begin{proof}
 a) follows from the definition of the $\mathbb{Z}_{2}$-grading. For b), a direct calculation gives
 \[
 [\upalpha(x),\Theta^{M}]_{\End}=\sum_{a}L_{x}\theta^{a}\otimes \rho_{E}(e_{a})+\sum_{a}\theta^{a}\otimes \rho_{E}([x,e_{a}]).
 \]
 Using Lemma~\ref{lemm::invariance_theta} and applying $\rho_{E}$ to the $\gg$-factor, one gets
 \[
 \sum_{a}(L_{x}\theta^{a})\otimes \rho_{E}(e_{a})+\sum_{a}\theta^{a}\otimes \rho_{E}([x,e_{a}])=0,
 \]
 hence b) follows.
 Finally, c) and d) follow from Lemma~\ref{lemm::invariance_theta} and Lemma~\ref{lemm::MCE} by applying $1\otimes \rho_{E}$.
\end{proof}

Let $\Dirac_{\gg}$ be the cubic Dirac operator of $\gg$, and let $d^{\mathcal{W}}\coloneqq [\Dirac_{\gg},\cdot]$ denote the associated Weil differential on $\mathcal{W}(\gg)$. The \emph{Weil covariant differential} combines both, $\Theta^{M}$ and $d^{\mathcal{W}}$ into a single operator acting on $\mathcal{W}(\gg)\otimes E$.

\begin{definition}
 The \emph{Weil covariant differential} is 
 \[
 \nabla^{M}\coloneqq d^{\mathcal{W}} \otimes 1 + \Theta^{M} \in \End_{\mathbb{C}}(\mathcal{W}(\gg) \otimes E).
 \]
\end{definition}

The following lemma summarizes the main properties of $\nabla^{M}$.

\begin{lemma} \label{lemm::properties_covariant_differential}
 The following hold:
 \begin{enumerate}
 \item[a)] $\nabla^{M} \in \End_{\mathbb{C}}(\mathcal{W}(\gg)\otimes E)_{\bar{1}}$.
 \item[b)] $\nabla^{M}$ is $\gg$-equivariant, that is, $[\upalpha(x),\nabla^{M}]_{\End}=0$ for all $x \in \gg$.
 \item[c)] $\nabla^{M}$ and $1\otimes \Dirac_{\gg}$ anti-commute, that is, $[\nabla^{M}, 1\otimes \Dirac_{\gg}]_{\End}=0$.
 \item[d)] $(\nabla^{M})^{2}=F^{M}$ with $F^{M}\coloneqq \sum_{a}u^{a}\otimes \rho_{E}(e_{a})$.
 \item[e)] For any $x \in \gg$ one has
\[
(\iota_{x}\otimes 1)\nabla^{M}=-d^{\mathcal{W}}(\iota_{x}\otimes 1)+L_{x}\otimes 1+ 1 \otimes \rho_{E}(x).
\]
 \end{enumerate}
\end{lemma}

\begin{proof}
 a) follows from Lemma~\ref{lemm::properies_D_l}, the definition of $d^{\WW}$ in Section~\ref{subsubsec::Weil_differential} and Lemma~\ref{lemm::properties_Theta_M}. Next, b) follows if we show that $[\upalpha(x), d^{\mathcal{W}}\otimes 1]_{\End}$=0 since $[\upalpha(x),\Theta^{M}]_{\End}=0$ by Lemma~\ref{lemm::properties_Theta_M}. However, by Lemma~\ref{lemm::properties_Weil_differential}, one has
 \[
 [\upalpha(x), d^{\mathcal{W}}\otimes 1]_{\End}=(L_{x} d^{\mathcal{W}}-(-1)^{p(x)}d^{\WW}L_{x}) \otimes 1=0.
 \]
 This proves b). For c), we compute
 \[
 \begin{aligned}
 [\nabla^{M}, 1 \otimes \Dirac_{\gg}]_{\End}=[d^{\mathcal{W}}\otimes 1, 1 \otimes \Dirac_{\gg}]_{\End}+[\Theta^{M}, 1 \otimes \Dirac_{\gg}]_{\End}=\sum_{a} \theta^{a}\otimes [\rho_{E}(e_{a}),\Dirac_{\gg}]=0
 \end{aligned}
 \]
 by the $\gg$-invariance of $\Dirac_{\gg}$. Next, we prove d). Since $\nabla^{M}, d^{\mathcal{W}}\otimes 1$ and $\Theta^{M}$ are odd, one has
 \begin{equation*} \begin{split}(\nabla^{M})^{2} &=\tfrac{1}{2}[\nabla^{M}, \nabla^{M}] \\ &=(d^{\mathcal{W}})^{2}\otimes 1 + \Theta^{M}(d^{\mathcal{W}}\otimes 1)+(d^{\mathcal{W}}\otimes 1)\Theta^{M}+(\Theta^{M})^{2}\\ &=\Theta^{M}(d^{\mathcal{W}}\otimes 1)+(d^{\mathcal{W}}\otimes 1)\Theta^{M}+(\Theta^{M})^{2},
 \end{split}
 \end{equation*}
 where we used $(d^{\mathcal{W}})^{2}=0$.
The statement follows by computing the three appearing summands on any elementary tensor $w \otimes v$ by linearity since the action is componentwise. One has:
\[
\begin{aligned}
 (d^{\mathcal{W}}\otimes 1)\Theta^{M}(w \otimes v)&=\sum_{a} d^{\mathcal{W}}(\theta^{a}w)\otimes v=\sum_{a} ((d^{\mathcal{W}}\theta^{a})w-\theta^{a}d^{\mathcal{W}}w)\otimes e_{a}v \\ 
 \Theta^{M} (d^{\mathcal{W}}\otimes 1)(w\otimes v) &=\sum_{a}(\theta^{a}d^{\mathcal{W}}w)\otimes e_{a}v, 
 \\ 
 (\Theta^{M})^{2}(w \otimes v) &=\sum_{a,b}\theta^{a}\theta^{b}w\otimes e_{a}e_{b}v=\tfrac{1}{2}\sum_{a,b}\theta^{a}\theta^{b}w\otimes [e_{a},e_{b}]v
\end{aligned}
\]
Now, the statement follows by using
\[
d^{\mathcal{W}}\theta^{a}=u^{a}-\tfrac{1}{2}\sum_{b,c}B([e_{b},e_{c}],e^{a})\theta^{b}\theta^{c}.
\]
Finally, e) follows by $[\iota_{x},d^{\mathcal{W}}]_{\mathcal{W}}=\iota_{x}d^{\mathcal{W}}+d^{\mathcal{W}}\iota_{x}=L_{x}$ and Lemma~\ref{lemm::properties_Theta_M}.
\end{proof}

\begin{remark}
 The properties of $\nabla^{M}$ justify the name covariant Weil differential.
\end{remark}

\subsubsection{Bismut--Quillen Superconnection}
The Bismut--Quillen superconnection couples the cubic Dirac operator acting on $E$ to the quantum Weil algebra $\mathcal{W}(\gg)$ using the Weil covariant differential.

\begin{definition}
 The \emph{Bismut--Quillen superconnection} associated with $M$ is 
 \[
 \mathbb{A}_{\gg}^{M}(t)\coloneqq \nabla^{M} + \sqrt{t} (1\otimes \Dirac_{\gg}) \in \End_{\mathbb{C}}(\mathcal{W}(\gg) \otimes E), \qquad t>0.
 \]
\end{definition}

The following proposition collects important properties of the Bismut--Quillen superconnection.

\begin{proposition}\label{prop::prperties_AM_absolute} The following hold:
\begin{enumerate}
\item[a)] $\mathbb{A}^{\!M}_{\gg}(t) \in \End_{\mathbb{C}}(\mathcal{W}(\gg)\otimes E)_{\bar{1}}$.
 \item[b)] $\mathbb{A}^{\!M}_{\gg}(t)$ is $\gg$-equivariant, that is, $[\upalpha(x), \mathbb{A}^{\!M}_{\gg}(t)]_{\End}=0$ for all $x \in \gg$. 
 \item[c)] $\mathbb{A}^{\!M}_{\gg}(t)$ has square 
 \[
 (\mathbb{A}_{\gg}^{\!M}(t))^{2}=F^{M} +t (1\otimes \Dirac_{\gg}^{2}).
 \]
\end{enumerate}
\end{proposition}
\begin{proof} Since $\Dirac_{\gg}$ is an odd operator on $E$ and $\nabla^{M}$ is odd, the Bismut--Quillen superconnection $\mathbb{A}^{\!M}_{\gg}(t)$ is odd. This proves a). Statement b) is a direct consequence of the $\gg$-invariance of $\Dirac_{\gg}$ and Lemma~\ref{lemm::properties_covariant_differential}. It remains to compute the square. Since $\mathbb{A}^{\!M}_{\gg}(t), 1\otimes \Dirac_{\gg}$ and $\nabla^{M}$ are odd, one has
 \[
 \begin{aligned}(\mathbb{A}^{\!M}_{\gg}(t))^{2} &=\tfrac{1}{2}[\mathbb{A}^{\!M}_{\gg}(t),\mathbb{A}^{\!M}_{\gg}(t)]_{\End} \\ &=\tfrac{1}{2}[\nabla^{M},\nabla^{M}]_{\End}+ \tfrac{\sqrt{t}}{2}[\nabla^{M},1\otimes \Dirac_{\gg}]_{\End} + \tfrac{\sqrt{t}}{2}[1\otimes \Dirac_{\gg}, \nabla^{M}]_{\End}+ \tfrac{t}{2}[1\otimes \Dirac_{\gg}, 1 \otimes \Dirac_{\gg}]_{\End} \\ &=(\nabla^{M})^{2}+\sqrt{t}[\nabla^{M},1\otimes \Dirac_{\gg}]_{\End}+t(1\otimes \Dirac_{\gg}^{2}) \\ &=F^{M}+t(1\otimes \Dirac_{\gg}^{2}),
 \end{aligned}
 \]
 where we used in the last equality that $[\nabla^{M},1\otimes \Dirac_{\gg}]_{\End}=0$ and $F^{M}\coloneqq (\nabla^{M})^{2}$.
 This finishes the proof.
\end{proof}

Next, we refine the definition by involving certain quadratic subalgebras yielding a relative version of the Bismut--Quillen superconnection. As in Section~\ref{subsec::relative_Dirac}, fix a quadratic Lie subalgebra $\ll$ of $\gg$ with non-degenerate supersymmetric invariant bilinear form $B_{\ll}\coloneqq B|_{\ll}$ such that we have an orthogonal decomposition
\begin{equation}
\gg=\ll\oplus\pp,\qquad \pp=\ll^{\perp},
\end{equation}
with respect to $B$. Moreover, assume $\mathfrak{h} \subset \ll$. Let $\Dirac_{\ll} \in \WW(\ll)$ denote the associated cubic Dirac operator.

Any $\gg$-module is in particular a $\ll$-module. Under the natural embedding $j : \mathcal{W}(\ll)\hookrightarrow \mathcal{W}(\gg)$ defined in \eqref{eq::embedding_j}, we consider $j(\Dirac_{\ll})$ as an endomorphism of $E$. 

\begin{definition}
 The $\ll$\emph{-relative Bismut--Quillen superconnection} is 
 \[
 \mathbb{A}_{\gg,\ll}^{M}(t)\coloneqq \mathbb{A}_{\gg}^{\!M}(t)-\sqrt{t}(1\otimes j(\Dirac_{\ll}))=\nabla^{M} + \sqrt{t}(1\otimes \Dirac_{\gg,\ll}).
 \]
\end{definition}

Recall that the space of $\ll$-basic elements is the subalgebra $\mathcal{W}(\gg,\ll)$ of elements annihilated by $L_{x}$ and $\iota_{x}$ for all $x \in \ll$. Inside $\mathcal{W}(\gg)\otimes E$, we consider the space of $\ll$-invariant elements
\begin{equation} 
 (\mathcal{W}(\gg,\ll)\otimes E)_{\text{basic}}\coloneqq \{ v \otimes w \in \mathcal{W}(\gg,\ll)\otimes E : \upalpha(x)(v\otimes w)=0 \ \text{for all} \ x \in \ll\}.
\end{equation}

The $\ll$-relative Bismut--Quillen superconnection $\mathbb{A}^{\!M}_{\gg,\ll}(t)$ has the following properties:

\begin{proposition} \label{prop::properties_relative} The following hold:
\begin{enumerate}
\item[a)] $\mathbb{A}^{\!M}_{\gg,\ll}(t) \in \End_{\mathbb{C}}(\mathcal{W}(\gg)\otimes E)_{\bar{1}}$.
 \item[b)] $\mathbb{A}^{\!M}_{\gg,\ll}(t)$ is $\ll$-equivariant, that is, $[\upalpha(x), \mathbb{A}^{\!M}_{\gg,\ll}(t)]_{\End}=0$ for all $x \in \ll$. 
 \item[c)] $\AA^{\!M}_{\gg,\ll}(t)$ restricts to an endomorphism of $(\WW(\gg,\ll) \otimes E)_{\text{basic}}$.
\item[d)] $\mathbb{A}^{\!M}_{\gg,\ll}(t)$ has square 
 \[ (\mathbb{A}^{\!M}_{\gg,\ll}(t))^{2}=F^{M} +t (1\otimes \Dirac_{\gg,\ll}^{2}).
 \]
\end{enumerate}
\end{proposition}

\begin{proof}
 a) follows from Lemma~\ref{lemm::properties_covariant_differential} and Proposition~\ref{prop::prperties_AM_absolute} and b) follows from Lemma~\ref{lemm::properies_D_l} and Proposition~\ref{prop::prperties_AM_absolute}. Moreover, c) is a direct consequence of b). Finally, d) follows by a similar line of argument as in Proposition~\ref{prop::prperties_AM_absolute} using that $\nabla^{M}$, $\Dirac_{\gg}$ and $j(\Dirac_{\ll})$ anti-commute. 
\end{proof}

\subsubsection{Chern Character-Type Invariant} Using the Bismut--Quillen superconnection, we assign to any finite-dimensional module $M$ over a complex semisimple Lie algebra $\gg$ a cohomology class. We use the notation from above.

We first introduce a suitable completion of $\Weil$. Consider the canonical
embedding $\iota:\gg\hookrightarrow\UE(\gg)$. There exists a unique unital algebra
morphism $\epsilon:\UE(\gg)\to\CC$, the \emph{augmentation}, satisfying
$\epsilon(\iota(x))=0$ for all $x\in\gg$. Its kernel
$
I\coloneqq \ker(\epsilon)=\UE(\gg)\gg
$
is the \emph{augmentation ideal} of $\UE(\gg)$. For $\Weil=\UE(\gg)\otimes\Cl(\gg)$, define
$\epsilon_{\WW}\coloneqq \epsilon\otimes\id_{\Cl}:\Weil\to\Cl(\gg)$ and set
$J\coloneqq \ker(\epsilon_{\WW})$. Then $J=I\otimes\Cl(\gg)$. For $n\ge1$, let $J^{n}$
denote the $n$-th power of the two-sided ideal $J$; explicitly,
$J^{n}=I^{n}\otimes\Cl(\gg)$ for all $n\ge1$.

The $J$-adic completion of $\Weil$ is defined as the projective limit
\begin{equation}
\widehat{\WW}(\gg)\coloneqq \varprojlim_{n\ge1}\Weil/J^{n},
\end{equation}
where the transition maps are induced by the canonical surjections
$\Weil/J^{n+1}\to\Weil/J^{n}$. In particular,
\begin{equation}
\widehat{\WW}(\gg)\cong\widehat{\UE}(\gg)\otimes\Cl(\gg),
\qquad
\widehat{\UE}(\gg)\coloneqq \varprojlim_{n\ge1}\UE(\gg)/I^{n}.
\end{equation}
The differential $d^{\WW}$ of $\Weil$ naturally induces a differential on $\widehat{\WW}(\gg)$, denoted by the same symbol, since it preserves the $J$-adic filtration. We denote by $\widehat{\WW}(\gg,\ll)$ the subalgebra of $\ll$-basic elements.

For any $x\in\gg$, identified with its image in $\Weil$, the exponential series $\exp(x)\coloneqq \sum_{n\ge0}\frac{x^{n}}{n!}$
converges in $\widehat{\WW}(\gg)$. Since $E$ is finite-dimensional, algebraic
exponentials of endomorphisms of $E$ are well-defined. Hence:

\begin{lemma}\label{lemm::exp_well_defined}
One has
\[
e^{-(\mathbb{A}^{M}_{\gg,\ll}(t))^{2}}
=\sum_{k=0}^{\infty}\frac{(-1)^{k}\bigl((\mathbb{A}^{M}_{\gg,\ll}(t))^{2}\bigr)^{k}}{k!}
\in \End_{\CC}\bigl((\widehat{\WW}(\gg,\ll)\otimes E)_{\mathrm{basic}}\bigr).
\]
\end{lemma}

Since $E$ is finite-dimensional, we can form the relative supertrace $\str_{E}$.

\begin{lemma} One has
\[
\str_{E}(e^{-\AA^{\!M}_{\gg,\ll}(t)^{2}}) \in \widehat{\WW}(\gg,\ll).
\]
\end{lemma}

\begin{proof}
 Consider a general element $T \in \End_{\CC}((\widehat{\WW}(\gg,\ll)\otimes E)_{\text{basic}})$. In particular, $(\iota_{x}\otimes 1)T=0$ and $\upalpha(x)T=0$ for all $x \in \ll$. We show that $\str_{E}(T) \in \widehat{\WW}(\gg,\ll)$, that is, it suffices to prove that $\iota_{x}\str_{E}(T)=0$ and $L_{x}\str_{E}(T)=0$ for all $x \in \ll$. Then the statement follows. 
 
 Fix $x \in \ll$. By linearity, it suffices to assume that $T$ is a simple tensor of the form $T=v \otimes w$. Then, using $(\iota_{x}\otimes 1)T=0$, one has
 \[
 \iota_{x}\str_{E}(T)=\iota_{x}(v\str_{E}(w))=(\iota_{x}v)\str_{E}(w)=\str_{E}((\iota_{x}v)\otimes w)=\str_{E}((\iota_{x}\otimes 1)T)=0.
 \]
 Next, using $\upalpha(x)T=0$, one has
 \[
 \str_{E}((L_{x}\otimes 1)T)=-\str_{E}((1\otimes \ad_{\rho_{E}(x)})T).
 \]
 Consequently, 
 \[
 L_{x}\str_{E}(T)=(L_{x}v)\str_{E}(w)=\str_{E}((L_{x}\otimes 1)T)=-\str_{E}((1\otimes \ad_{\rho_{E}(x)})T)=v\str_{E}([\rho_{E}(x),w])=0.
 \]
 This finishes the proof.
\end{proof}

This leads to the following definition:
\begin{equation}
\operatorname{ch}_{M}(t):=\str_{E}\bigl(e^{-(\AA^{\!M}_{\gg,\ll}(t))^{2}}\bigr)\in \widehat{\mathcal{W}}(\gg,\ll).
\end{equation}
In general, $\operatorname{ch}_{M}(t)$ is non-zero, although $\str_{E}(e^{-\Dirac_{\gg,\ll}^{2}})=0$ (\emph{cf.}~Section~\ref{subsec::an_example}). Since $E$ is finite-dimensional, $\operatorname{ch}_{M}(t)$ is well defined for every $t>0$ by Proposition~\ref{prop::properties_relative} and Lemma~\ref{lemm::exp_well_defined}. We will show that $\operatorname{ch}_{M}(t)$ is independent of $t$ and determines a class in
\begin{equation}
\H(\widehat{\mathcal{W}}(\gg,\ll),d^{\mathcal{W}_{\gg,\ll}})\cong \widehat{\mathfrak{Z}}(\ll),
\end{equation}
where $\widehat{\mathfrak{Z}}(\ll)$ denotes the completion of the center of $\mathfrak{U}(\ll)$.

\begin{lemma} \label{lemm::differential_of_trace}
 For any $T \in \End_{\mathbb{C}}((\widehat{\mathcal{W}}(\gg,\ll) \otimes E)_{\text{basic}})$ one has 
 \[
 d^{\mathcal{W_{\gg,\ll}}}\str_{E}(T)=\str_{E}([\nabla^{M},T]_{\End})
 \]
\end{lemma}

\begin{proof}
 By linearity of $\str_{E}$, it suffices to consider $T=v \otimes w$ for $v\in \widehat{\mathcal{W}}(\gg,\ll)$ and $w \in \End_{\CC}(E)$. Then 
 \[
 \begin{aligned}
 \str_{E}([\nabla^{M},v \otimes w]_{\End}) &=\str_{E}(d^{\mathcal{W}_{\gg,\ll}}v\otimes w + \sum_{a}(\theta^{a}v\otimes [\rho_{E}(e_{a}),w]_{\End})) \\ &=(d^{\mathcal{W}_{\gg,\ll}}v)\str_{E}(w)+ \sum_{a}(\theta^{a}v)\str_{E}([\rho_{E}(e_{a}),w]_{\End})\\ &=(d^{\mathcal{W}_{\gg,\ll}}v)\str_{E}(w),
 \end{aligned}
 \]
 since $\str_{E}([\rho_{E}(e_{a}),w])=0$ for all $a$ by cyclicity of the supertrace.
\end{proof}

\begin{lemma} \label{lemm::differential_chi}
 For all $t > 0$, one has
 \[
d^{\mathcal{W}_{\gg,\ll}}\operatorname{ch}_{M}(t)=0.
 \]
\end{lemma}

\begin{proof}
 By Lemma~\ref{lemm::differential_of_trace}, one has 
 \[
 d^{\mathcal{W}_{\gg,\ll}}\str_{E}(e^{-(\AA^{\!M}_{\gg,\ll}(t))^{2}})=\str_{E}([\nabla^{M}, e^{-(\AA^{\!M}_{\gg,\ll}(t))^{2}}]_{\End})=\str_{E}(0)=0
 \]
 since $\nabla^{M}$ commutes with $(\nabla^{M})^{2}$ and $\Dirac_{\gg,\ll}^{2}$, and $\AA^{\!M}_{\gg,\ll}(t)^{2}=(\nabla^{M})^{2}+t(1\otimes \Dirac_{\gg,\ll}^{2})$ by Proposition~\ref{prop::properties_relative}.
\end{proof}

\begin{theorem} \label{thm::main_Bismut_Quillen}
 For any finite-dimensional $\gg$-module $M$ one has
 \[
[\operatorname{ch}_{M}(t)] \in \H(\widehat{\mathcal{W}}(\gg,\ll), d^{\mathcal{W}_{\gg,\ll}}).
 \]
 Moreover, the class $[\operatorname{ch}_{M}(t)]$ is independent of $t$.
\end{theorem}

\begin{proof}
 By $\H(\widehat{\mathcal{W}}(\gg,\ll), d^{\mathcal{W}_{\gg,\ll}})=\H^{0}(\widehat{\mathcal{W}}(\gg,\ll), d^{\mathcal{W}_{\gg,\ll}})\cong \widehat{\mathfrak{Z}}(\mathfrak{U}(\ll))$, and Lemma~\ref{lemm::differential_chi}, it follows that $[\operatorname{ch}_{M}(t)] \in H(\widehat{\mathcal{W}}(\gg,\ll), d^{\mathcal{W}_{\gg,\ll}})$. It remains to show that it is independent of $t$. 
 
For this purpose, we consider $\AA^{\!M}_{\gg,\ll}(t)$ as a smooth family of endomorphisms in the variable $t$. Then, using that $\AA^{\!M}_{\gg,\ll}(t)$ is odd for all $t$, one has
 \[
 \frac{\mathrm{d}}{\mathrm{d}t}(\AA_{\gg,\ll}^{M}(t)^{2})=\AA_{\gg,\ll}^{M}(t)\frac{\mathrm{d}}{\mathrm{d}t}\AA_{\gg,\ll}^{M}(t)+\frac{\mathrm{d}}{\mathrm{d}t}\AA_{\gg,\ll}^{M}(t) \AA_{\gg,\ll}^{M}(t)=[\AA_{\gg,\ll}^{M}(t), \frac{\mathrm{d}}{\mathrm{d}t}\AA_{\gg,\ll}^{M}(t)]_{\End}
 \]
 and thus
 \[
 \begin{aligned}
 \frac{\mathrm{d}}{\mathrm{d}t} e^{-(\AA_{\gg,\ll}^{M}(t))^{2}} &=- \int_{0}^{1} e^{-s(\AA_{\gg,\ll}^{M}(t))^{2}}(\frac{\mathrm{d}}{\mathrm{d}t} (\AA_{\gg,\ll}^{M}(t))^{2})e^{-(1-s)(\AA_{\gg,\ll}^{M}(t))^{2}} \mathrm{d}s \\ &=- \int_{0}^{1} e^{-s(\AA_{\gg,\ll}^{M}(t))^{2}}[\AA_{\gg,\ll}^{M}(t),\frac{\mathrm{d}}{\mathrm{d}t}\AA_{\gg,\ll}^{M}(t)]_{\End}e^{-(1-s)(\AA_{\gg,\ll}^{M}(t))^{2}} \mathrm{d}s.
 \end{aligned}
 \]
Applying the supertrace and using $[\AA_{\gg,\ll}^{M}(t), e^{-(\AA_{\gg,\ll}^{M}(t))^{2}}]_{\End}=0$, one has
\[
\begin{aligned}
 \frac{\mathrm{d}}{\mathrm{d}t} \operatorname{ch}_{M}(t) &=- \str_{E}([\AA_{\gg,\ll}^{M}(t), \frac{\mathrm{d}}{\mathrm{d}t}\AA_{\gg,\ll}^{M}(t)]_{\End}e^{-(\AA_{\gg,\ll}^{M}(t))^{2}}) \\ &=- \str_{E}([\AA_{\gg,\ll}^{M}(t),\frac{\mathrm{d}}{\mathrm{d}t}\AA_{\gg,\ll}^{M}(t))e^{-(\AA_{\gg,\ll}^{M}(t))^{2}}]_{\End}) \\ &=-\str_{E}([\nabla^{M},(\frac{\mathrm{d}}{\mathrm{d}t}\AA_{\gg,\ll}^{M}(t))e^{-(\AA_{\gg,\ll}^{M}(t))^{2}}]_{\End}) \\ &=-d^{\mathcal{W}_{\gg,\ll}} \str_{E}((\frac{\mathrm{d}}{\mathrm{d}t}\AA_{\gg,\ll}^{M}(t))e^{-(\AA_{\gg,\ll}^{M}(t))^{2}}). 
\end{aligned}
\]
This finishes the proof.
\end{proof}

Altogether, we have assigned to each finite-dimensional simple module a cohomology class in the colour quantum Weil algebra, which can be seen as a coarse Chern--Weil type invariant of $\gg$-modules. We close this section with an example.

\subsubsection{Example: \texorpdfstring{$\mathfrak{sl}(2,\CC)$}{sl(2,C)} and the Bismut--Quillen Superconnection}\label{subsec::an_example}
 We consider $\gg\coloneqq \mathfrak{sl}(2,\CC)$. Fix a standard $\gg$-triple $(e,f,h)$ with
\begin{equation}
 [h,e]=2e,\qquad [h,f]=-2f,\qquad [e,f]=h.
\end{equation}
Let $\hh\coloneqq \CC h$ be the Cartan subalgebra and let $B(X,Y)\coloneqq \tr_{\CC}(\ad_X\ad_Y)$ be the Killing form (standard normalization). Set
$\nn^{+}\coloneqq \CC e$ and $\nn^{-}\coloneqq \CC f$, so that $\gg=\nn^{-}\oplus\hh\oplus\nn^{+}$.

For $n\in\ZZ_{\ge 0}$, let $V_n$ be the simple $\gg$-module of highest weight $n$. It has basis
$v_0,\dots,v_n$ with action
\[
 h v_k=(n-2k)v_k,\qquad
 f v_k=\begin{cases} v_{k+1}, & 0\le k<n,\\ 0, & k=n,\end{cases}
 \qquad
 e v_k=\begin{cases} 0, & k=0,\\ k(n-k+1)v_{k-1}, & 1\le k\le n.\end{cases}
\]
Thus the weights are $n,n-2,\dots,-n$, each with multiplicity one. The center $\mathfrak{Z}(\gg)$ of the universal enveloping algebra is generated by the quadratic Casimir element
\begin{equation}
 \Omega_{\gg}=ef+fe+\tfrac12 h^{2},
\end{equation}
which acts on $V_n$ by the scalar $\tfrac{n(n+2)}{2}$.

Let $S:=\bigwedge \mathfrak n^{-}$ be the $\rho$-shifted spin module for the Clifford algebra of $\mathfrak p:=\mathfrak n^{+}\oplus \mathfrak n^{-}$, with respect to the pairing induced by $B$. It is generated by $s_{\bar 0}:=1\in \bigwedge^{0}\mathfrak n^{-}$ and $s_{\bar 1}:=f\in \bigwedge^{1}\mathfrak n^{-}$, of parity $\bar 0$ and $\bar 1$, respectively. The action of $\mathfrak g$ is given by
\begin{equation*}
e\cdot s_{\bar 0}=s_{\bar 1}, \qquad e\cdot s_{\bar 1}=0, \qquad f\cdot s_{\bar 0}=0, \qquad f\cdot s_{\bar 1}=s_{\bar 0}, \qquad h\cdot s_{\bar 0}=s_{\bar 0}, \qquad h\cdot s_{\bar 1}=-s_{\bar 1}.
\end{equation*}
Hence $S_{\bar 0}$ has weight $-\rho$ and $S_{\bar 1}$ has weight $\rho$. Therefore $S$ is the simple $\mathfrak g$-module of highest weight $\rho$, that is, $S\cong V_{1}$. Its $\ZZ_{2}$-grading is $S_{\bar 0}=\bigwedge^{0}\mathfrak n^{-}=\CC s_{\bar 0}$ and $S_{\bar 1}=\bigwedge^{1}\mathfrak n^{-}=\CC s_{\bar 1}$.

As a quadratic subalgebra we take $\ll\coloneqq \hh$. It has quadratic Casimir $\Omega_{\hh}=\tfrac{1}{8}h^{2}$. It acts on a one-dimensional $\hh$-module $\CC_{\lambda}$ (where $h$ acts by $\lambda$) by $\tfrac{\lambda^{2}}{8}$.

As an $\hh$-supermodule,
$V_n\otimes S$ decomposes into one-dimensional weight spaces, with parity determined by
the $S$-factor:
\begin{equation}
 (V_n\otimes S)_{\bar 0}=\bigoplus_{k=0}^n \CC\,(v_k\otimes s_{\bar 0}),\qquad
 (V_n\otimes S)_{\bar 1}=\bigoplus_{k=0}^n \CC\,(v_k\otimes s_{\bar 1}),
\end{equation}
and
\begin{equation} \label{eq::explicit_action_of_h}
 h(v_k\otimes s_{\bar 0})=(n-2k+1)(v_k\otimes s_{\bar 0}),\qquad
 h(v_k\otimes s_{\bar 1})=(n-2k-1)(v_k\otimes s_{\bar 1}).
\end{equation}
Equivalently, writing $\CC_\lambda$ for the one-dimensional $\hh$-module of weight $\lambda$,
\begin{equation}
 V_n\otimes S \cong
 \left(\bigoplus_{k=0}^n \CC_{\,n-2k+1}\right)_{\bar 0}
 \oplus
 \left(\bigoplus_{k=0}^n \CC_{\,n-2k-1}\right)_{\bar 1}.
\end{equation}
Thus the odd weights are $n-1,n-3,\dots,-n-1$, while the even weights are
$n+1,n-1,\dots,-n+1$.

Let $\Dirac_{\gg,\hh}$ be the $\hh$-relative cubic Dirac operator. It squares to 
\begin{equation}
 \Dirac_{\gg,\hh}^{2}=\Omega_{\gg} \otimes 1 - \frac{1}{8}(h\otimes 1 + 1 \otimes \gamma^{\WW}(h))-\frac{1}{8}(1\otimes 1).
\end{equation}
Consequently, if $\CC_{\mu}$ is a $\hh$-weight space of $V_{n}\otimes S$, it acts as scalar multiplication by 
\begin{equation}
 \frac{(n+1)^{2}-\mu^{2}}{8}.
\end{equation}
We conclude that $\CC_{\mu}$ belongs to $\ker \Dirac_{\gg,\hh}^{2}$ iff $\mu=\pm (n+1)$. Because each finite-dimensional highest weight module $V_n$ carries a $\mathfrak{su}(2)$-invariant positive Hermitian form (so the compact real form acts by skew-Hermitian operators), the basis-independent cubic Dirac operator $\Dirac_{\gg,\hh}$ is therefore selfadjoint, and for a selfadjoint operator one has $\ker \Dirac_{\gg,\hh}=\ker \Dirac_{\gg,\hh}^2$. Hence
\begin{equation} \label{eq::example_kernel_Dirac}
 \ker \Dirac_{\gg,\hh}=\ker \Dirac_{\gg,\hh}^{2}=\CC_{n+1} \oplus \CC_{-(n+1)}.
\end{equation}
Note that $\CC_{n+1}$ is spanned by $v_{0}\otimes s_{\bar{0}}$ and $\CC_{-(n+1)}$ is spanned by $v_{n}\otimes s_{\bar{1}}$, that is, $\CC_{n+1}$ is even while $\CC_{-(n+1)}$ is odd. In particular, $\sdim_{\CC}(\ker \Dirac_{\gg,\hh})=1-1=0$.

A naive construction of an invariant would be 
\begin{equation}
 \str_{V_{n}\otimes S}(e^{-\Dirac_{\gg,\mathfrak{h}}^{2}}).
\end{equation}
However, this is always trivial. Indeed, by Lemma~\ref{lemm::decomposition_under_Laplace}, we have
$V_{n}\otimes S\cong \ker \Dirac_{\gg,\hh}^{2}\oplus \Im \Dirac_{\gg,\hh}^{2}$, and since $\Dirac_{\gg,\hh}^{2}$ is positive,
all its eigenvalues are positive. Let $\lambda>0$ be an eigenvalue, and set $E\coloneqq V_{n}\otimes S$. Then 
\begin{equation}
 \Dirac_{\gg,\hh} : \ker (\Dirac_{\gg,\hh}^{2}-\lambda) \cap E_{\bar{0}} \to \ker(\Dirac_{\gg,\hh}^{2}-\lambda)\cap E_{\bar{1}}
\end{equation}
defines an isomorphism. We conclude
\begin{equation}
\str_{E}(e^{-\Dirac_{\gg,\hh}^{2}})=\str_{E}(\operatorname{id}_{\ker \Dirac_{\gg,\hh}^{2}})=\sdim({\ker \Dirac_{\gg,\hh}^{2}})=1-1=0.
\end{equation}

If, instead, we couple $\Dirac_{\gg,\hh}$ to the covariant Weil differential coming from the universal connection $1$-form, the resulting class is non-trivial.

For the standard triple $\{e,f,h\}$, denote the corresponding even and odd
generators in $\Weil$ by $u^{e},u^{f},u^{h}$ and
$\theta^{e},\theta^{f},\theta^{h}$. Then
\begin{equation}
\nabla^{V_{n}}=\tfrac{1}{4}\theta^{e}\otimes\rho_{E}(f)
+\tfrac{1}{4}\theta^{f}\otimes\rho_{E}(e)
+\tfrac{1}{8}\theta^{h}\otimes\rho_{E}(h),\quad
F^{V_{n}}=\tfrac{1}{4}u^{e}\otimes\rho_{E}(f)
+\tfrac{1}{4}u^{f}\otimes\rho_{E}(e)
+\tfrac{1}{8}u^{h}\otimes\rho_{E}(h).
\end{equation}
The $\gg$-action on $E$ is given by
\begin{equation}
\rho_{E}(e)=\rho_{V_{n}}(e)\otimes1,\quad
\rho_{E}(f)=\rho_{V_{n}}(f)\otimes1,\quad
\rho_{E}(h)=\rho_{V_{n}}(h)\otimes1+1\otimes\gamma'(h),
\end{equation}
where the last formula is specified in~\eqref{eq::explicit_action_of_h}.

If $n=0$, so that $V_{0}\cong\CC$ is trivial, then
$\rho_{V_{0}}(e)=\rho_{V_{0}}(f)=\rho_{V_{0}}(h)=0$, hence
\begin{equation}
\nabla^{V_{0}}=\tfrac{1}{8}u^{h}\otimes\rho_{E}(h),\quad
F^{V_{0}}=\tfrac{1}{8}u^{h}\otimes\rho_{E}(h),\quad
\Dirac_{\gg,\hh}^{2}=0,
\end{equation}
and therefore
\begin{equation}
\operatorname{ch}_{V_{0}}(t)=\str_{S}(e^{-F^{V_{0}}}).
\end{equation}
Using $\rho_{E}(h)s_{\bar{0}}=s_{\bar{0}}$ and
$\rho_{E}(h)s_{\bar{1}}=-s_{\bar{1}}$, one obtains
\begin{equation}
\operatorname{ch}_{V_{0}}(t)
=e^{\tfrac{1}{8}u^{h}}-e^{-\tfrac{1}{8}u^{h}}
=2\sinh\!\bigl(\tfrac{1}{8}u^{h}\bigr)
\in\widehat{\WW}(\gg,\ll),
\end{equation}
which defines a class in
$\H(\widehat{\WW}(\gg,\ll),d^{\WW_{\gg,\ll}})\cong\CC[[h]]$.

For $n\ge1$, equip $V_{n}$ with a Hermitian form such that
$\Dirac_{\gg,\hh}$ is selfadjoint and
$\Dirac_{\gg,\hh}^{2}$ is non-negative. Decomposing $E$ into
$\Dirac_{\gg,\hh}^{2}$-eigenspaces $E(\lambda)$ gives
\begin{equation}
e^{-t\Dirac_{\gg,\hh}^{2}}\xrightarrow{t\to\infty}
\begin{cases}
1,&\lambda=0,\\
0,&\lambda\ne0.
\end{cases}
\end{equation}
Since $\Dirac_{\gg,\hh}$ and $\nabla^{V_{n}}$ commute and $[\operatorname{ch}_{V_{n}}(t)]$ is independent of $t$,
\begin{equation}
[\operatorname{ch}_{V_{n}}(t)]
=[ \str_{E}(e^{-F^{V_{n}}}e^{-t\Dirac_{\gg,\hh}^{2}}]=[\str_{\Ker\Dirac_{\gg,\hh}}
\bigl(e^{-F^{V_{n}}}\bigr)],
\end{equation}
and by~\eqref{eq::example_kernel_Dirac},
\begin{equation}
[\operatorname{ch}_{V_{n}}(t)]
=\bigl[e^{\tfrac{n+1}{8}u^{h}}-e^{-\tfrac{n+1}{8}u^{h}}\bigr]
=\bigl[2\sinh\!\bigl(\tfrac{n+1}{8}u^{h}\bigr)\bigr]
\in\H(\widehat{\WW}(\gg,\hh),d^{\WW_{\gg,\hh}})
\cong\CC[[h]].
\end{equation}

\subsection{A Word on Basic Classical Lie Superalgebras}

Let $\gg$ be a basic classical Lie superalgebra and let $\ll\subseteq\gg$ be a quadratic Lie subsuperalgebra as in Section~\ref{subsec::relative_Dirac}, so that $\gg=\ll\oplus\pp$. For every finite-dimensional $\gg$-supermodule $M$ we may form the Bismut--Quillen superconnection $\AA^M_{\gg,\ll}(t)$ introduced above; with the usual sign conventions, the construction and its formal properties extend to the super setting. In contrast to the case of a complex semisimple Lie algebra, the oscillator module $\overline{M}(\pp)$ is infinite-dimensional, and hence the supertrace on $E\coloneqq M\otimes\overline{M}(\pp)$ is not a priori available. Nevertheless, for unitarizable finite-dimensional $\gg$-supermodules $M$ (Section~\ref{subsec::unitarizable_supermodules}) one can extract a canonical substitute. In this case $\Dirac_{\gg,\ll}$ is selfadjoint, hence $\Dirac_{\gg,\ll}^2$ is positive. Writing $E=\bigoplus_{\lambda\ge0}E(\lambda)$ for the $\Dirac_{\gg,\ll}^2$-eigenspace decomposition, one has on $E(\lambda)$
\[
e^{-t\Dirac_{\gg,\ll}^2}\xlongrightarrow[t\to\infty]{}\begin{cases}\id_{E(0)},&\lambda=0,\\0,&\lambda\neq0.\end{cases}
\]
Moreover $\ker\Dirac_{\gg,\ll}^2=\ker\Dirac_{\gg,\ll}=\H_{\Dirac_{\gg,\ll}}(M)$ and $\ker\Dirac_{\gg,\ll}$ is finite-dimensional by \cite{Schmidt_Noja}. Hence, whenever a relative supertrace on $E$ can be defined so that the transgression argument applies, the class $[\chi_M(t)]$ is independent of $t$, and letting $t\to\infty$ yields the substitute
\begin{equation}
\widetilde{\operatorname{ch}}_M\coloneqq \str_{\ker\Dirac_{\gg,\ll}}\!\bigl(e^{F^M}\bigr).
\end{equation}
Since $\nabla^M$ commutes with $\Dirac_{\gg,\ll}$, the curvature term $F^M$ preserves $\ker\Dirac_{\gg,\ll}$, so this expression is well-defined. In particular, it determines a class $[\widetilde{\operatorname{ch}}_M]\in \H(\widehat{\WW}(\gg,\ll),d^{\WW_{\gg,\ll}})$.

\thispagestyle{empty}